\documentclass[review,onefignum,onetabnum]{siamart190516}

\usepackage{bm,bbm}
\usepackage{enumitem}
\usepackage{amsmath}
\usepackage{mathdefs}
\usepackage{amsfonts}
\usepackage{algpseudocode}
\usepackage{mathrsfs,mathtools}

\numberwithin{equation}{section}
\nolinenumbers
\newsiamremark{remark_new}{Remark}
\newsiamremark{hypothesis}{Hypothesis}
\crefname{hypothesis}{Hypothesis}{Hypotheses}
\newsiamthm{claim}{Claim}

\headers{Hierarchical Low-Rank Matrices for GP Regression}{V. Keshavarzzadeh, S. Zhe, R. M. Kirby, and A. Narayan}

\title{GP-HMAT: Scalable, $\mathcal{O}(\lowercase{n} \log(\lowercase{n}))$ Gaussian Process Regression with Hierarchical Low-Rank Matrices\thanks{Submitted to the editors 2021.
\funding{This research was sponsored by ARL under Cooperative Agreement Number W911NF-12-2-0023. The views and conclusions contained in this document are those of the authors and should not be interpreted as representing the official policies, either expressed or implied, of ARL or the U.S. government. The U.S. government is authorized to reproduce and distribute reprints for government purposes notwithstanding any copyright notation herein. The first and fourth authors are partially supported by AFOSR FA9550-20-1-0338. The fourth author is partially supported by DARPA EQUiPS N660011524053.}}}

\author{
  Vahid Keshavarzzadeh\thanks{Scientific Computing and Imaging Institute, University of Utah, Salt Lake City, UT (\email{vkeshava@sci.utah.edu}, \email{kirby@sci.utah.edu}, \email{akil@sci.utah.edu}).}
  \and
  Shandian Zhe\thanks{School of Computing, University of Utah, Salt Lake City, UT (\email{zhe@cs.utah.edu})}
  \and
  Robert M. Kirby\footnotemark[2]\hskip 3pt$^,$\footnotemark[3]
    \and
  Akil Narayan\footnotemark[2]\hskip 3pt$^,$\thanks{Department of Mathematics, University of Utah, Salt Lake City, UT}
}

\ifpdf
\hypersetup{
  pdftitle={GP-HMAT: Scalable, $\mathcal{O}(\lowercase{n} \log(\lowercase{n}))$ Gaussian Process Regression with Hierarchical Low Rank Matrices},
  pdfauthor={V. Keshavarzzadeh, S. Zhe, R. M. Kirby, and A. Narayan}
}
\fi

\begin{document}

\maketitle

\begin{abstract}
A Gaussian process (GP) is a powerful and widely used regression technique.
The main building block of a GP regression is the covariance
 kernel, which characterizes the relationship between pairs in the random field. The optimization to
 find the optimal kernel, however, requires several large-scale and often unstructured matrix inversions.
 We tackle this challenge by introducing a hierarchical matrix approach, named HMAT, which effectively decomposes the matrix structure, in a recursive manner, into significantly smaller matrices
 where a direct approach could be used for inversion. Our matrix partitioning uses a particular aggregation strategy for data points, which promotes the low-rank structure of off-diagonal blocks in the hierarchical kernel matrix. We employ a randomized linear algebra method for matrix reduction on the
 low-rank off-diagonal blocks without factorizing a large matrix. We provide analytical error
 and cost estimates for the inversion of the matrix, investigate them empirically with numerical computations, and demonstrate the application of our approach on three numerical examples involving
 GP regression for engineering problems and a large-scale real dataset. We provide the
 computer implementation of GP-HMAT, HMAT adapted for GP likelihood and derivative computations, and the implementation of the last numerical example on a real dataset. We demonstrate superior scalability of the HMAT approach compared
 to built-in $\backslash$ operator in MATLAB for large-scale linear solves $\bm A \bm x = \bm y$ via a repeatable and verifiable empirical study. An extension to hierarchical semiseparable (HSS) matrices is discussed as future research.
\end{abstract}

\begin{keywords}
Scalable Gaussian Process, Hierarchical Matrices, Matrix Partitioning, Randomized Linear Algebra, Marginal Likelihood Optimization 
\end{keywords}

\begin{AMS}
15A23, 65F05, 65F20, 65F55 , 65N55 	
\end{AMS}

\section{Introduction}\label{Sec1}

\subsection{Motivation and background}\label{Sec1_1}
A Gaussian process (GP) is a random field or a stochastic process defined by a collection of random variables that are typically associated with spatial points or time instances. Every finite collection of such random variables has a multivariate normal distribution in the Gaussian process. This Gaussian characteristic makes GP a versatile and amenable tool for numerical computations in science and engineering~\cite{Rasmussen06}. An indispensable part of any GP regression is the task of matrix inversion and determinant computation, which are major numerical linear algebra tasks. To be precise, let $\mathcal{D} = \{(\bm x_i, \bm y_i)\}_{i=1}^n$ be the training samples where $\bm x_i \in \mathbb{R}^d$ and $\bm y_i \in \mathbb{R}$ are the sampling sites (nodes)  and the observational data, and $\bm A$ a kernel matrix with entries $[\bm A]_{ij} = k(\bm x_i, \bm x_j, \bm \ell)$ where $k(\bm x_i, \bm x_j): \mathbb{R}^d \times \mathbb{R}^d \rightarrow \mathbb{R}$ is the kernel function and $\bm \ell$ is the vector of hyperparameters. Then, a typical objective function in GP is the log likelihood:
\begin{equation}\label{first_eq}
\begin{array}{l l l}
    \mathcal{L}  = \displaystyle -\frac{1}{2} \bm y^{T} (\bm A + \sigma_n^2\bm I)^{-1} \bm y   -\frac{1}{2} \log| \bm A + \sigma_n^2\bm I | - \frac{n}{2} \log(2\pi)
        \end{array}
\end{equation}
where $\bm I$ is the identity matrix, the fixed constant $\sigma_n$ represents the effect of noise in data. The goal in GP training is to maximize the objective function, i.e. log likelihood $\mathcal{L}$ to find the optimal hyperparameters $\bm \ell$. In addition to this well-known objective, the evidence lower bound denoted by $\mathcal{L}_{lower}$ has also been universally considered in the context of variational Bayesian inference and more specifically in variational inference approaches for GP, e.g. sparse variational GP~\cite{Burt19}. In sparse approximations, typically a Nystr\"{o}m approximation of the kernel matrix with $n$ data points $\bm A_{nn}$ in the form of $\hat{\bm A}_{nn} = \bm A^T_{mn} \bm A^{-1}_{mm} \bm A_{nm}$  where $m \ll n$ (hence sparse approximation) is considered~\cite{Titsias09_main}. 


Both objective functions $\mathcal{L}$ and $\mathcal{L}_{lower}$ involve linear system solves for computing the energy term i.e. $\bm y^T (\bm A + \sigma_n^2 \bm I)^{-1} \bm y$ and log determinant computation. In particular, the challenge of GP regression is the $\mathcal{O}(n^2)$ storage and $\mathcal{O}(n^3)$ arithmetic operation complexity for inversion which can hamper its application for large $n$. This issue is addressed to an extent in the sparse GP approaches~\cite{Burt19}, which show $\mathcal{O}(nm)$ storage and $\mathcal{O}(nm^2)$ time complexity. However, the determination of $m \ll n$ data points, referred to as induced points or landmark points in the GP literature, is not a trivial task and, itself, can be cast as a continuous optimization problem~\cite{Snelson6} or a discrete sampling problem (as a subset of original $n$ points)~\cite{Smola01,Keerthi6} in conjunction with the optimization of hyperparameters. 

In this paper, we develop a scalable GP regression approach by leveraging powerful tools/ideas in the rich field of numerical linear algebra~\cite{Golub_96,Anderson99}. The significant potential of using principled linear algebra approaches within the context of GP regression has also been the subject of much active research~\cite{JieChen21,Schafer20_main}. Among many approaches we specifically adopt the hierarchical low-rank matrices~\cite{Hackbush99,gra03} viewpoint that has been proven to be highly effective in reducing the computational complexity of large-scale problems. We pragmatically show that 1) the hierarchical decomposition of the matrix, 2) low-rank factorization of off-diagonal blocks via a randomized algorithm, and  3) performing low-rank updates via Sherman-Morrison-Woodburry (SMW) formula (similar to the Schur-complement formulation in the original $\mathcal{H}$-matrix literature) collectively render a concrete pipeline for inversion of large-scale matrices which we further adapt for the task of GP training and regression.

In what follows, we highlight notable aspects of our computational framework in comparison with a number of existing scalable GP approaches (typically involving Nystr\"{o}m low-rank approximation) and hierarchical matrix approaches (typically involving solution to finite element or finite difference discretization of elliptic partial differential equations). We also briefly discuss limitations associated with the current framework.

\textbf{Hierarchical decomposition}: A limited number of research works have addressed scalable GP with hierarchical matrix decomposition~\cite{JieChen21,Geoga20}. In~\cite{JieChen21}, the authors consider a similar problem but it appears that their hierarchical inversion procedure requires large matrix-matrix multiplications similar to original $\mathcal{H}$-matrix works which results in $\mathcal{O}(n\log(n)^2k^2)$ scalability. A GP regression approach for large data sets is presented in~\cite{Sudipto08} where the hierarchy is assumed within a Bayesian inference context. Similar SMW computations are utilized for likelihood evaluation in their MCMC sampling. In contrast to the GP literature, the literature on hierarchical matrix approaches for solving sparse linear systems $\bm A \bm x= \bm y$ associated with the discretization of elliptic PDEs such as Laplace’s equation is vast~\cite{Kenneth16,LIN20114071,Martinsson2009,Xia10multi}. The majority of research works in the context of sparse solvers incorporate geometric information. Geometric multigrid~\cite{Brandt77,Briggs2000} and graph partitioning~\cite{karyp95,karyp98} approaches such as nested dissection~\cite{George73} are among the well-established methods used within sparse solvers. Contrarily, the GP kernel matrices are often dense and constructed on an unstructured set of nodes. We are mainly interested in forming low-rank off-diagonal blocks for which we devise a new and simple data aggregation strategy inspired by (geometry-independent) algebraic multigrid~\cite{AMG85,Vanek96}. 

\textbf{Low-rank approximation}: Many scalable GP approaches leverage efficient Nystr\"{o}m approximation~\cite{Drineas05_main} or ensemble Nystr\"{o}m method~\cite{Kumar09_main} for low-rank factorization. These approaches are less systematic and potentially less accurate compared to the (randomized) linear algebra procedures that we consider in this paper (See Section~\ref{S5_1}). An example of incorporating randomized SVD into GP computation (however in a different setting than matrix inversion) is in~\cite{Boulle21}. Another conceptually hierarchical approach for GP is considered in~\cite{Schafer20_main} where the authors perform incomplete Cholesky factorization using maximin elimination ordering for invoking a sparsity pattern.  We emphasize that the application of randomized algorithms for low-rank approximation in hierarchical matrices is not new~\cite{Martinsson2011}; however, the successful application and error estimation of such approaches in the matrix-free hierarchical inversion of dense matrices is less common. For example, the authors in~\cite{LIN20114071} consider forward computation for hierarchical matrices with a randomized algorithm. The construction is based on the simplifying assumption that the matrix-vector multiplication in their randomized SVD approach is fast. We eliminate this assumption and demonstrate a successful subsampling approach for fast computation of the randomized range finder that yields the overall desirable scalability $\mathcal{O}(n\log(n)k)$.


\textbf{Optimization of parameterized hierarchical matrices}:
The gradient-based optimization with parameterized hierarchical matrices has not been the central focus of the linear algebra community. However, gradient computation is a vital part of many machine learning procedures, including GP. The authors in~\cite{Geoga20} consider a similar GP problem and discuss an approach for HODLR matrices~\cite{Ambikasaran13} (HODLR: hierarchical off-diagonal low-rank, similar to our hierarchical construction); however, the low-rank factorization is based on Nystr\"{o}m approximation and the hierarchical gradient computation is not apparent from the formulation.  In this paper, we provide concrete algorithmic steps for computing the log likelihood gradient (with respect to hyperparameters $\bm \ell$) in a generic hierarchical setting. 

 \subsection{Contributions of this paper}
Our contribution in this paper was initially motivated by recent similar works~\cite{JieChen21,Geoga20}. We develop a scalable hierarchical inversion algorithm using a principled randomized linear algebra method. We provide analytical error and cost estimates for the matrix inversion in addition to their empirical investigations. The practical aspects of our contribution are as follows: 1) We provide a numerical scheme, HMAT, that requires $\mathcal{O}(n \log(n)k)$ arithmetic operations for the linear solve $\bm A \bm x = \bm y$~\footnote{The empirical results in Section~\ref{Sec5} exhibit even superior scalability in the last term, $k$. In particular, we present results with $\mathcal{O}(n \log(n)k^{\alpha})$ scalability where $\alpha \leq 1/2$.}. Linear solves with HMAT on $n=10^6$ nodes take slightly more than a minute on a single CPU on a relatively modern computer. 2) Our hierarchical matrix approach is adapted to likelihood $\mathcal{L}$ cf. Equation~\eqref{first_eq} and its derivative $\partial \mathcal{L}/\partial  \bm \ell$ computation  for the task of GP training. We provide another scheme, GP-HMAT for the GP training/regression that includes the likelihood gradient computations. The code infrastructure for these schemes is implemented in a modularized way. As a result, several forms of kernels (or hyperparameters) can be added to the main solver with minimal effort. The HMAT and GP-HMAT solvers are available in~\cite{Keshavarzzadeh_MATLAB_GPHMAT}. 

\textbf{Limitations of the approach}:
 A limitation of our approach is the sequential nature of computations. One major advancement is parallelizing the factorization of off-diagonal blocks. Recursive factorization of these blocks in a hierarchical semiseparable format is another advancement that we plan to pursue in our future work. Another limitation is the range of applicability with respect to various kernels. The rank structure of matrices (whether kernel or finite difference/element matrices) is not generally well understood. Our framework similarly to many others (e.g.~\cite{Schafer20_main}) is well suited to the low-rank kernel matrices (or smooth analytical kernels). We promote low-rankness (of off-diagonal blocks) via our data aggregation strategy, which to the best of our knowledge is done for the first time. We provide simple guidelines for practical application of our approach in Section~\ref{Sec5}.

The organization of the paper is described via the main algorithmic components of the computational framework:
\textbf{Matrix partitioning}: In Section~\ref{Sec2}, we describe a new and simple yet effective permutation strategy for aggregation of data points which results in lower rank off-diagonal blocks in the hierarchical matrix.   
\textbf{Low-rank approximation of off-diagonal blocks}: Section~\ref{Sec3} discusses our randomized approach for fast low-rank factorization of off-diagonal blocks. We report relevant error estimates that will be used for the hierarchical error estimation in the subsequent section.      
\textbf{Hierarchical matrix inversion via SMW formula}: In Section~\ref{Sec4}, we present the main algorithm for linear solve $\bm A \bm x = \bm y$ leveraging hierarchical low-rank SMW updates. This section also includes theoretical error and cost analyses in addition to computation of likelihood and its derivative for GP training. 

Section~\ref{Sec5} presents numerical experiments on three examples involving regression on an analytical problem for empirical studies on the code scalability and approximation errors, an engineering problem and a real dataset with a large size. Finally, Section~\ref{Sec6} discusses the final remarks and ideas for future research, including development of HSS matrices~\cite{Chandrasekaran05,Xia10_main,Xia10} in combination with high-performance computing for GP regression.

\section{Hierarchical decomposition}\label{Sec2}

\subsection{Notation and setup}\label{S2_1}
Throughout the paper, the following notation and setup are frequently used: We use bold characters to denote matrices and vectors. For example, $\bm x \in \R^d$ indicates a vector of variables in the domain of a multivariate function. We refer to computing the solution $\bm x = \bm A^{-1} \bm y$ as linear solve and to $\bm x$ as the solution to linear solve. The vector $\bm x \in \mathbb{R}^n$ is measured with its Euclidean norm $\| \bm x \|_2= (\sum_{i=1}^n x_i^2)^{1/2}$ and the matrices are mostly measured by their Frobenius norm denoted by $\| A \|_F = (\sum_{i,j} A_{i,j}^2)^{1/2} $.  In some cases we report results on $l_2$-norm, i.e. $\|\bm A \|_2 =\sup_{x \in \mathbb{R}^n, \bm x \neq \bm 0} \|\bm A \bm x\|_2/\|\bm x\|_2$. The notation \# denotes the cardinality of a set (or size of a set), e.g. given $\mathcal{I} =\{0,1,2,3\} \rightarrow \#\mathcal{I}=4$.

\subsection{Matrix partitioning}\label{S2_2}

Our main strategy to decompose large matrices into smaller ones involves a dyadic hierarchical decomposition. The process depends on two user-defined parameters: $\eta \in \N$ corresponding to the size of the largest square blocks that we can computationally afford to directly invert, and $\nu: \N \rightarrow \N$ corresponding to how finely each set of indices is dyadically partitioned. In this work, we use the following definition of $\nu$:
\begin{align}\label{sl_size}
  \nu(k) \coloneqq 10^{\left\lfloor \log_{10} (k - 0.5) \right\rfloor},
\end{align}
which essentially computes the largest power of 10 that is \textit{strictly} smaller than $k$. 
If we are given a set of indices of size $k$, then we split this set of indices into two subsets, one of size $\nu(k)$ and the other of size $k - \nu(k)$.

With $\bm A$ an $n \times n$ symmetric matrix, we create a tree, where each node corresponds to a subset of the indices $[n]$, and the branching factor/outdegree is equal to 2. The creation of the tree proceeds by starting at level 0, where a single parent (root) node associated with the full ordered index set $(1, \ldots, n)$ is declared. The construction of the tree is now recursively accomplished by looping over the level $l$: For each node at level $l$ associated with an ordered index set $\mathcal{I}$, the following operations are performed:
\begin{enumerate}[itemsep=1pt]
  \item If $\# \mathcal{I} \leq \eta$, then this node is marked as a leaf node, and no further operations are performed.
  \item If $\# \mathcal{I} > \eta$, then we subdivide this node into two child nodes that are inserted into level $l+1$.
    \begin{enumerate}[itemsep=1pt]
      \item Let $k = \# \mathcal{I}$. With $\mathcal{I} = (i_1, \ldots, i_{k})$, let $\bm a $ be a $1 \times k$ row vector corresponding to row $i_1$ and columns $\mathcal{I}$ of $\bm A$, i.e., $\bm a = \bm A(i_1, \mathcal{I})$. 
      \item We compute a permutation $(p_1, \ldots, p_{k})$ of $(1,\ldots, k)$, constructed by ordering elements of $\bm a$ in non-increasing magnitude, i.e., 
        \begin{align*}
          |a_{p_{q+1}}| &\leq |a_{p_q}|, & q &= 1, \ldots, k-1.
        \end{align*}
        In terms of work, this corresponds to generating/extracting a size-$k$ row from $\bm A$ and sorting in decreasing order.
      \item A child node of size $\nu(k)$ is created at level $l+1$ associated with the first $\nu(k)$ permuted indices of $\mathcal{I}$, i.e., $(j_{p_1}, \ldots, j_{p_{\nu(k)}})$.
      \item A second child node of size $k - \nu(k)$ is created at level $l+1$ associated with the remaining indices, i.e., $(j_{p_{\nu(k)+1}}, \ldots, j_{k})$.
    \end{enumerate}
\end{enumerate}
The procedure above is repeated for each node at level $l$. If, after this procedure, level $l+1$ is empty, then the decomposition process terminates. Otherwise, the above procedure is applied again with $l \gets l+1$. Figure \ref{Hmat_merged} illustrates the entire decomposition on a matrix of size $n = 5000$ and $\eta = 100$.

\begin{figure}[h!]
\centering
\includegraphics[width=3.9in]{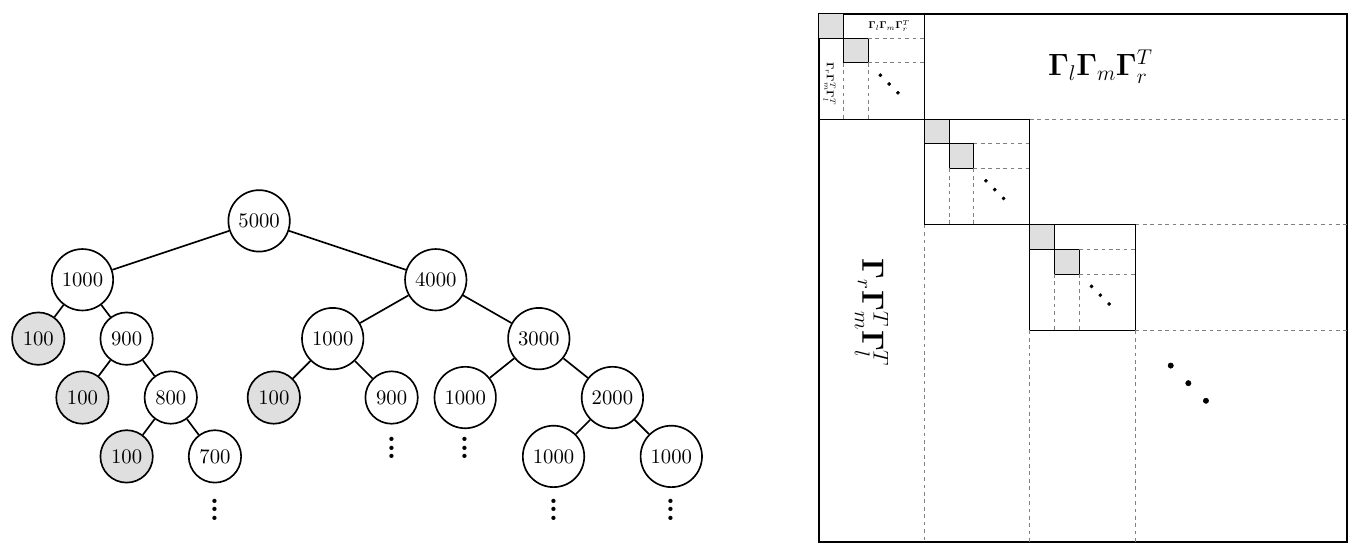}
\caption{\small{An example of the tree structure and its hierarchical matrix counterpart with size $n=5000$. For better visualization, the aspect ratio in the smallest blocks is not respected. The smallest blocks are meant to be $100 \times 100$ matrices within $1000 \times 1000$ blocks. Nodes associated with the smallest blocks are highlighted in gray color.}}\label{Hmat_merged}
\end{figure}

By virtue of this decomposition, each node (that is not the root) has precisely 1 sibling associated with the same parent node. If $\mathcal{I}$ and $\mathcal{J}$ are index sets associated with a node and its sibling, then the off-diagonal block $\bm A({\mathcal{I}, \mathcal{J}})$ (and also its symmetric counterpart) are approximated by an SVD $\bm \Gamma_l \bm \Gamma_m \bm \Gamma_r^T$ that we describe in the next sections. 

The above strategy for matrix partitioning is adopted from a more generic approach, \emph{maximal independent set} (MIS), that is often used in algebraic multigrid~\cite{Vanek96,Bell12,Treister15}. The main difference between our approach and MIS is the size-dependent parameter $\nu(k)$ for division of index sets instead of a parameter that is dependent on the magnitude of matrix entries (often denoted by $\theta$ in the MIS literature). Consideration of $\nu(k)$ (and consequently $k-\nu(k)$) directly results in the desirable number of aggregates, which is 2 in our construction. As mentioned in Section~\ref{Sec1}, our approach is also inherently different from graph partitioning approaches (used in sparse solvers), which are based on min-cut algorithms~\cite{stoerwagner} or nested dissection~\cite{George73}. In such sparse solvers, the goal is to minimize the communication between individual processors or reduce the number of fill-ins in Cholesky factorization. Our goal is to form off-diagonal blocks that are of low rank. The favorable effect of our aggregation strategy on the rank of the off-diagonal blocks is empirically demonstrated in Section~\ref{Sec5}. 

\section{Low-rank matrix approximation}\label{Sec3}
This part of our framework mainly adopts the procedures presented in~\cite{nhalko11}. In particular, we employ the randomized SVD approach with interpolative decomposition (ID) with some minor modifications that we briefly explain in this section. 

\subsection{Computing the range approximator Q}\label{Sec3_1}
The most significant building block of a randomized SVD factorization is the range approximator matrix $\bm Q$. Given a matrix $\bm A \in \mathbb{R}^{m \times n}$, the range approximator matrix $\bm Q \in \mathbb{R}^{m \times k}$ with $k$ orthonormal columns, is sought that admits 
\begin{equation}\label{randomized_app}
\|\bm A -\bm Q \bm Q^T \bm A\|_F \leq \epsilon
\end{equation}

where $\epsilon$ is a positive error tolerance (and is a crucial part of our error analysis). The high-level idea is then to approximate SVD of $\bm A$ from the SVD of a smaller matrix $\bm B = \bm Q^T \bm A$. A key consideration is efficient computation of $\bm Q$, and that is indeed achieved via a randomized approach. In particular, the range approximator $\bm Q \in \mathbb{R}^{m \times k}$ is found from a QR factorization on another matrix $\bm Y = \bm A \bm \Omega$ where $\bm \Omega \in \mathbb{R}^{n \times (k+p)}$ is a random matrix drawn from standard normal distribution with $k$ being a target rank and $p$ an oversampling parameter.

One significant bottleneck in performing this randomized approach is the computation of $\bm Y = \bm A \bm \Omega$. The complexity of computing $\bm Y$ is $\mathcal{O}(mnk)$ which for large $m$ and $n$ (and assuming $m \approx n$) yield quadratic scalability i.e. $\mathcal{O}(n^2k)$ which will deteriorate the scalability for the main inversion algorithm. We ideally aim to keep the complexity of individual steps within the main algorithm to $\mathcal{O}(n)$ such that the overall algorithm exhibits the desired scalability. In our case, the main algorithm achieves $\mathcal{O}(n \log(n))$ scalability where the $\log(n)$ term appears due to the depth of the hierarchical tree structure cf. Lemma~\ref{lemma:storage_complexity}.

The complexity issue of $\bm Y$ and the ways to address it have been discussed in~\cite{nhalko11, Rokhlin13212, WOOLFE08}; however, the actual implementation of the approaches therein decreases the complexity of $\mathcal{O}(mnk)$ to only $\mathcal{O}(mn\log(k))$. To achieve the desired scalability for the inversion algorithm, we subsample quite significantly in the second dimension of $\bm A \in \mathbb{R}^{m \times n}$, so that the inner product of each subsampled row of $\bm A$ and respective columns of smaller $\bm \Omega$ is computed fast.  

We set a maximum number of entries that we afford to evaluate from the kernel denoted by $n_{max}$ and find the number of subsampled columns as $n_{innprod}=\lfloor n_{max}/m \rfloor$. We also ensure that the number of subsampled columns is within the range $[2k,~10k]$ via
\begin{equation}\label{n_inn}
n_{innprod} \gets \min(\max(2k,n_{innprod}),10k).
\end{equation}
The above relation and the choice of $n_{max}$ are user-defined. Once this number is found, then the samples are uniformly drawn from integer numbers $\{1,\ldots,n\}$. With the above consideration, the number of subsampled columns will be at most, $10k$ which yields the $\mathcal{O}(nk^2)$ scalability (where again we assumed $m \approx n$ just to report the scalability) for computation of $\bm Y = \bm A \bm \Omega$. This is a desirable estimate as the cost is linearly proportional to the number of degrees of freedom. The minimum $2k$ samples also facilitates the justification for the subsampling, which we empirically investigate via a numerical example.

Note that evaluating $\tilde{\bm A} \in \mathbb{R}^{m \times n_{innprod}}$ is significantly cheaper than evaluating ${\bm A} \in \mathbb{R}^{m \times n}$ from the kernel function. However, the cost associated with this evaluation (i.e. calculation of $\exp(.)$) is not considered in our cost analysis, which is beyond the scope of this paper. We mainly consider basic matrix operations, e.g. addition and multiplication, in our cost analysis.
  
Once $\bm Q$ is computed, the computation of $\bm B = \bm Q^T \bm A$ can be followed; however, matrix $\bm A \in \mathbb{R}^{m \times n}$ is a full matrix that renders computation of $\bm B$ costly. Therefore, we consider a sketch of matrix $\bm A$ via interpolative decomposition (ID)~\cite{GuMing96,Cheng05} cf.~\ref{sec:ID}, a useful and well-established linear algebra procedure. 

\subsection{Randomized SVD with ID}~\label{sec:RSVD_ID} To compute the randomized SVD with ID, we specifically follow Algorithm 5.2 in~\cite{nhalko11}. The ID matrix, denoted by $\bm X$, is found from the interpolative decomposition procedure, applied to the matrix $\bm Q \in \mathbb{R}^{m \times k}$ cf. Equation~\eqref{randomized_app}.

 Algorithm~\ref{alg:rsvd_id} presents various steps in computing the SVD factorization, given GP nodes $\bm N \in \mathbb{R}^{d \times n}$ where $d$ is the dimension of data points and $n$ is the number of GP nodes,  and two index sets $\mathcal{I}_1,~\mathcal{I}_2$ where $\# \mathcal{I}_1=n_1,~\# \mathcal{I}_2=n_2$.  As an example for $n = 10^6$, using the relation~\eqref{sl_size} results in $n_1=10^5$ and $n_2=9\times 10^5$ for the largest (and first) off-diagonal block in the hierarchical decomposition. Also, considering $n_{max}=5\times10^6$ and $k=20$ yields $n_{innprod}=50$.

\begin{algorithm}
\caption{Randomized SVD with ID:~$[\bm \Gamma_l, \bm \Gamma_m, \bm \Gamma_r] \gets$ \texttt{rsvd\_id}($\bm N, [\mathcal{I}_1,\mathcal{I}_2],k$)}

\begin{algorithmic}[1]
\State Compute the matrix $\tilde{\bm A} \in \mathbb{R}^{n_1 \times n_{innprod}}$ from the kernel function using $n_{innprod}$ cf. Equation~\eqref{n_inn}
\State Compute $\bm Y = \tilde{\bm A} \bm \Omega$ where $\bm \Omega \in \mathbb{R}^{n_{innprod} \times k}$ is a standard normal random matrix 
\State Compute the range approximator matrix $\bm Q \in \mathbb{R}^{n_1 \times k}$ via $[\bm Q, \bm R_Y]  \gets qr(\bm Y)$
\State Compute the ID matrix $\bm X \in \mathbb{R}^{n_1 \times k}$ from $\bm Q$ and row indices $\mathcal{I}_{ID}$, i.e. $\bm Q \simeq \bm X \bm Q(\mathcal{I}_{ID},:)$ cf. Section~\ref{sec:ID} 
\State Evaluate the row skeleton matrix $\bm A(\mathcal{I}_{ID},:) \in \mathbb{R}^{k \times n_2}$ from the kernel function 
\State Compute the QR factorization of the row skeleton matrix, i.e.  $[\bm W, \bm R]  \gets qr(\bm A(\mathcal{I}_{ID},:)^T)$ where $\bm W \in \mathbb{R}^{n_2 \times k}$ is an orthonormal matrix and $\bm R \in \mathbb{R}^{k \times k}$
\State Compute $\bm Z = \bm X \bm R^T$ where $\bm Z \in \mathbb{R}^{n_1 \times k}$
\State Compute an SVD of $\bm Z = \bm \Gamma_l \bm \Gamma_m \tilde{\bm \Gamma}_r^T$ where $\bm \Gamma_l \in \mathbb{R}^{n_1 \times k}$,~$\bm \Gamma_m \in \mathbb{R}^{k \times k}$ and $\tilde{\bm \Gamma}_r \in \mathbb{R}^{k \times k}$
\State Form the orthonormal matrix $\bm \Gamma_r = \bm W \tilde{\bm \Gamma}_r$ where $\bm \Gamma_r \in \mathbb{R}^{n_2 \times k}$
\end{algorithmic}\label{alg:rsvd_id}
\end{algorithm}

The following two theorems provide bounds for the error tolerance $\epsilon$ cf. Equation~\eqref{randomized_app}. The first one concerns the mean of such error and the second one concerns the probability of failure. We use both estimates, i.e. mean and the threshold (i.e. the right-hand side of inequality~\eqref{err_pf}) to generate random samples of error in the empirical studies.

\begin{theorem}[Theorem 10.5 in~\cite{nhalko11}]\label{theorem1}
Suppose that $\bm A \in \mathbb{R}^{m \times n}$ is a real matrix with singular values $\sigma_1 \geq \sigma_2 \geq \ldots$. Choose  a target rank $k \geq 2$ and an oversampling parameter $p \geq 2$ where $k+p \leq \min\{m,n\}$. Draw an $n \times (k+p)$ standard Gaussian matrix $\bm \Omega$, and construct the sample matrix $\bm Y = \bm A \bm \Omega$. Then the expected approximation error (with respect to Frobenius norm) is
\begin{equation}\label{err_mean}
 \mathbb{E} \| (\bm {I} - \bm Q \bm Q^T) \bm A \|_F \leq \left(1+ \displaystyle \frac{k}{p-1}\right)^{1/2} \sqrt{\min\{m,n\}-k} ~\sigma_{k+1}.
 \end{equation}
\end{theorem}

\begin{theorem}[Theorem 10.7 in~\cite{nhalko11}]\label{theorem2}
Consider the hypotheses
of Theorem~\ref{theorem1}. Assume further that $p \geq 4$. For all $u,~t \geq 1$,
\begin{equation}\label{err_pf}
 \| (\bm {I} - \bm Q \bm Q^T) \bm A \|_F \leq \left [\left(1+ t\displaystyle \sqrt{\frac{3k}{p+1}}\right) \sqrt{\min\{m,n\}-k} + ut \frac{e \sqrt{k+p}}{p+1} \right ] ~\sigma_{k+1} 
 \end{equation}
with the failure probability at most $2t^{-p} + e^{-u^2/2}$.
\end{theorem}
In our empirical studies, we use $t=e \approx 2.718$ and $u=\sqrt{2p}$. As a result, the failure probability simplifies to $3e^{-p}$. The following lemma provides a bound for the SVD factorization in Algorithm~\ref{alg:rsvd_id} using the error tolerance $\epsilon$ provided in the above theorems. 

\begin{lemma}[Lemma 5.1 in~\cite{nhalko11}]\label{lemma2}
Let $\bm A$ be an $m \times n$ matrix and $\bm Q$ an $m \times k$ matrix that satisfy~\eqref{randomized_app}, i.e. $\|\bm A - \bm Q \bm Q^T \bm A\| \leq \epsilon$. Suppose that $\bm U, \bm S$ and $\bm V$ are the matrices consructed by Algorithm~\ref{alg:rsvd_id}. Then,
\begin{equation}\label{epsvd}
 \| \bm A - \bm U \bm S \bm V^T \|_F \leq \left [ 1+ \sqrt{k + 4k(n-k)} \right ] \epsilon.
 \end{equation}
\end{lemma}
In the next section, we use the above bound, $\epsilon_{svd} = \left [ 1+ \sqrt{k + 4k(n-k)} \right ] \epsilon$ as the error in the off-diagonal factorization in our hierarchical construction.

\section{SMW computations and GP approximation}\label{Sec4}
In this section, we discuss the main inversion algorithm cf. Subsection~\ref{Sec4_1}, the analytical error and cost estimates cf. Subsections~\ref{sec:error} and~\ref{sec:cost}, the computation of likelihood and its gradient cf. Subsection~\ref{Sec4_4} and finally, the GP regression using the hierarchical matrix framework cf. Subsection~\ref{Sec4_5}.

\subsection{SMW computations for linear solve}\label{Sec4_1}

We have hitherto discussed matrix decomposition and our way of computing an SVD factorization for a rectangular matrix. In this subsection, we present our main tool for matrix inversion, i.e. Sherman-Morrison-Woodbury formula and discuss the associated algorithm for the hierarchical linear solve. 

The SMW formula for matrix inversion, with a similar formula for determinant computation is as follows:
 \begin{equation}\label{smw_formula}
 \begin{array}{l}
\bm A^{-1} = (\bm A_D+\bm U \bm C \bm V)^{-1} = \bm A_D^{-1} - \bm A_D^{-1} \bm U (\bm C^{-1}+\bm V \bm A_D^{-1} \bm U)^{-1} \bm V \bm A_D^{-1}\\
\\
 \text{det}(\bm A) = \text{det}(\bm A_D+\bm U \bm C \bm V) = \text{det}(\bm A_D) \text{det}(\bm C) \text{det}(\bm C^{-1} + \bm V \bm A^{-1}_D \bm U)
\end{array}
\end{equation}
where $\bm A$ is represented via the following dyadic decomposition to the diagonal blocks $\bm A_D$
and off-diagonal blocks $\bm A_{OD} = \bm U \bm C \bm V$:
{\fontsize{8.50}{8.50}\selectfont
\begin{equation}\label{smw_matrix}
\bm A=
  \left [ \begin{array}{l l}
 \bm A_{11} &  \bm \Gamma_l \bm \Gamma_m \bm \Gamma_r^T\\
 \bm \Gamma_r \bm \Gamma_m^T \bm \Gamma_l^T & \bm A_{22}
 \end{array} \right ]
  =  \underbrace{ \left [ \begin{array}{l l}
 \bm A_{11} & \bm 0\\
 \bm 0  & \bm A_{22}
 \end{array} \right ] }_{\bm A_D} +
 \underbrace{ \left [ \begin{array}{l l}
 \bm \Gamma_l & \bm 0\\
 \bm 0  & \bm \Gamma_r
 \end{array} \right ] }_{\bm U}
\underbrace{ \left [ \begin{array}{l l}
 \bm \Gamma_m  & \bm 0\\
 \bm 0 & \bm \Gamma_m^T
 \end{array} \right ] }_{\bm C}
 \underbrace{\left [ \begin{array}{l l}
 \bm 0 & \bm \Gamma_r^T\\
 \bm \Gamma_l^T  & \bm 0
 \end{array} \right ]}_{\bm V}. 
 \end{equation}}It is beneficial to revisit Figure~\ref{Hmat_merged} (lower right pane) to notice that every diagonal block of the original matrix is decomposed with the dyadic decomposition introduced above. The decomposition of diagonal blocks is continued until the block becomes smaller than or equal to $\eta$. 
 
Including the right-hand side $\bm y = [\bm y_1^T~ \bm y_2^T]^T$ in the first line of~\eqref{smw_formula} yields
\begin{equation}
\bm x = \bm A^{-1} \bm y = \bm A_D^{-1} \bm y- \bm A_D^{-1} \bm U (\bm C^{-1}+\bm V \bm A_D^{-1} \bm U)^{-1} \bm V \bm A_D^{-1} \bm y,
\end{equation}
which is the main equation in our hierarchical linear solve approach. To better understand the hierarchical linear solve algorithm, we transform the above equation to a decomposed matrix form
\begin{equation}~\label{smw_eq}
\begin{array}{l l}
\bm x = \bm A^{-1} \bm y & = 
  \left [ \begin{array}{l}
 \bm x_{D_1} \\
 \bm x_{D_2}  
 \end{array} \right ]  - \left [ \begin{array}{l l}
 \bm q_{l_1} & \bm 0\\
 \bm 0  & \bm q_{l_2}
 \end{array} \right ]  \left (\bm C^{-1}_{smw}  \left [ \begin{array}{l}
 \bm q_{ry_2} \\
 \bm q_{ry_1}  
 \end{array} \right ] \right ) \\
 \\
& = \left [ \begin{array}{l} 
 \bm x_{D_1} \\
 \bm x_{D_2}  
 \end{array} \right ]  - \left [ \begin{array}{ l}
 \bm q_{l_1}\bm s_{qry_1} \\
\bm q_{l_2}\bm s_{qry_2}
 \end{array} \right ]  
 \end{array}
\end{equation}
by introducing the following important variables: $\bm x_{D_1} \coloneqq \bm A^{-1}_{11} \bm y_1 \in \mathbb{R}^{n_1\times d_{y}},~\bm x_{D_2} \coloneqq \bm A^{-1}_{22} \bm y_2 \in   \mathbb{R}^{n_2\times d_{y}},~\bm q_{l_1}\coloneqq \bm A^{-1}_{11} \bm \Gamma_l \in   \mathbb{R}^{n_1\times k},~\bm q_{l_2}\coloneqq \bm A^{-1}_{22} \bm \Gamma^T_r  \in   \mathbb{R}^{n_2\times k},~\bm q_{lr_1}\coloneqq \bm \Gamma^T_l  \bm q_{l_1}  \in \mathbb{R}^{k\times k},~\bm q_{lr_2}\coloneqq \bm \Gamma_r \bm q_{l_2} \in \mathbb{R}^{k\times k},~\bm q_{ry_1}\coloneqq \bm \Gamma^T_l  \bm x_{D_1} \in \mathbb{R}^{k\times d_{y}},~ \bm q_{ry_2}\coloneqq \bm \Gamma_r \bm x_{D_2} \in \mathbb{R}^{k\times d_{y}}$
where 
\begin{equation}\label{Csmw}
\begin{array}{l l}
\bm C_{smw} \coloneqq \left [ \begin{array}{l l}
 \bm \Gamma^{-1}_m & \bm q_{lr2}\\
 \bm q_{lr1} & \bm \Gamma^{-T}_m 
 \end{array} \right ],  & 
 \left [ \begin{array}{l}
 \bm s_{qry_1} \\
 \bm s_{qry_2}  
 \end{array} \right ] \coloneqq \bm C^{-1}_{smw} \left [ \begin{array}{l}
 \bm q_{ry_2} \\
 \bm q_{ry_1}  
 \end{array} \right ]. 
 \end{array}
\end{equation}
In Equations~\eqref{smw_eq} and~\eqref{Csmw}, the matrix $\bm C_{smw}$ is the SMW correction matrix, i.e. $\bm C_{smw} = \bm C^{-1}+\bm V \bm A_D^{-1} \bm U$ and $\bm s_{qry_1} \in \mathbb{R}^{k \times d_y},~\bm s_{qry_2} \in \mathbb{R}^{k \times d_y}$ are two vectors (when $d_y=1$) or matrices (when $d_y>1$) obtained from a small size linear solve, i.e. $2k \times 2k$ with SMW correction matrix as the left-hand side. We also recall from the previous section that $\Gamma_l \in \mathbb{R}^{n_1 \times k}$,~$\Gamma_m \in \mathbb{R}^{k \times k}$ and $\Gamma_r \in \mathbb{R}^{n_2 \times k}$.

It is important to emphasize a crucial point in this hierarchical SMW computation. Finding $\bm x$, the solution to the linear solve problem, involves computing $\bm A_D^{-1} \bm U$. The solver hierarchically solves $\bm A^{-1} \bm y$ or ($\bm A_D^{-1} \bm y$). The solver is also recursively called to solve $\bm A_D^{-1} \bm U$. Only the right-hand side should be changed from $\bm y$ to $\bm U$. Indeed, the solver solves a portion of $\bm y$ and (approximate) right singular vector of the off-diagonal block at a particular level simultaneously. We achieve this by concatenating  $[\bm y_{1}~ \bm \Gamma_{l}]$ or $[\bm y_{2}~ \bm \Gamma_{r}]$. The main psuedocode for linear solve denoted by \texttt{back\_solve} and its  associated psuedocodes are presented in~\ref{App_algs}.

\subsection{Error estimate}\label{sec:error}
In this subsection, we develop an error estimate for the hierarchical computation of the solution $\bm x$. In both error and cost estimates we consider the particular dyadic decomposition shown in Figure~\ref{Hmat_merged} (right). We develop the total error for original (largest) matrix by assuming the smallest diagonal block (with size $n_{min}$) in the lower right corner of the matrix and enlarging the matrix by adding similar size diagonal blocks. Using the SMW formula, the foundation of our error estimation is built on the fact that the full matrix at level $i$ (i.e. combination of diagonal and off-diagonal blocks) forms the diagonal part of the matrix at level $i-1$, i.e. $\bm A_D^{(i-1)} = \bm A^{(i)}_D + \bm A^{(i)}_{OD}$. 

To simplify the presentation, we define two particular notations for computing the Frobenius norm of the matrix and the difference of matrices (i.e. the difference/error between the (full) matrix and its approximation): $f_{\alpha}(\bm A) = \| \bm A \|_F$ and $f_{\epsilon} (\bm A)=\|\bm A - \hat{\bm A} \|_F$.
Similarly, we define scalar variables $\alpha_i = f_{\alpha}(\bm A^{(i)^{-1}}_D),~\beta_i = f_{\alpha}(\bm U^{(i)}),~\epsilon_{D,i} = f_{\epsilon}(\bm A^{(i)^{-1}}_D)$ and $\epsilon_{OD,i} = f_{\epsilon}(\bm U^{(i)})$ to denote $f_{\alpha}$ and $f_{\epsilon}$ quantities at level $i$. 

The main result in Proposition~\ref{prob_frob_err} bounds the error of matrix inversion in a hierarchical matrix, that requires the estimate of error in each component of off-diagonal factorizations (computed individually at each level) cf. Lemma~\ref{lemma_UCV} and the Frobenius norm of the inverse of diagonal blocks  (computed hierarchically) cf. Lemma~\ref{lemma_frob_norm}. 

\begin{lemma}\label{lemma_UCV}
Given $\bm A_{OD} = \bm U \bm C \bm V$, $f_{\epsilon}(\bm A_{OD})=\| \bm A_{OD} - \hat{\bm A}_{OD} \|=\epsilon$ cf. Equation~\eqref{epsvd} and assuming $ \beta = \min (\| \bm U \|,\| \bm C \|,\| \bm V \|)$ then
 \begin{equation*}
 \begin{array}{l}
f_{\epsilon}(\bm U)  \leq \epsilon/\beta^2,~f_{\epsilon}(\bm C) \leq \epsilon/\beta^2,~ f_{\epsilon}(\bm V) \leq \epsilon/\beta^2.\\
\end{array}
\end{equation*}
\end{lemma}
\begin{proof}
See~\ref{proof_lemmaUCV}. 
\end{proof}

When we exercise the result of the above lemma in the numerical examples, we set $\epsilon \gets \sqrt{2} \epsilon_{svd}$ cf. Equation~\eqref{epsvd}. 

\begin{lemma}\label{lemma_frob_norm}%
Consider the hierarchical matrix $\bm A_D^{(i-1)} = \bm A^{(i)}_D + \bm A^{(i)}_{OD}$ where $L$ is the deepest level. Assuming the minimum singular value of the SMW correction matrix across all levels is constant, i.e. $\sigma_{min}(\bm C^{(i)}_{smw}) = \sigma_{C_{min}}~\forall i$, knowing the rank of such matrices is $2k$, i.e. rank$(\bm C^{(i)}_{smw})=2k$, and letting $\beta_{L-j}$ denote $\beta$ (defined in Lemma~\ref{lemma_UCV}) at level $L-j$, then
 \begin{equation*}
 \begin{array}{l}
\alpha_{L-i} = \| \bm A_D ^{(L-i)^{-1}} \|_F \leq \alpha_L ^{2^i} \kappa^{2i-1}\displaystyle \prod_{j=0}^{i-1} \beta_{L-j}^{2^{i-j-1}}  
\end{array}
\end{equation*}
where  $\alpha_L = \| \bm A_D^{(L)^{-1}} \|_F$ is associated with the smallest two diagonal blocks at level $L$ and $\kappa = \sqrt{2k}/\sigma_{C_{min}}$.
\end{lemma}
\begin{proof}
See~\ref{proof_lemmafrobnorm}. 
\end{proof}
Using the results of Lemma~\ref{lemma_UCV} and Lemma~\ref{lemma_frob_norm}, the following proposition states the main result for the inversion error in a hierarchical matrix:
\begin{proposition}\label{prob_frob_err}
Given the settings of the hierarchical matrix $\bm A_D^{(i-1)} = \bm A^{(i)}_D + \bm A^{(i)}_{OD}$, having the estimates for the Frobenius norm of the inverse of hierarchical matrix at level $L-i$ as $\alpha_{L-i}$ cf. Lemma~\ref{lemma_frob_norm} and the error at the deepest level denoted by $\epsilon_{D,L}=\varepsilon$, the error in matrix inversion at level $L-i$ is 
\begin{equation}\label{direct_in_text}
\epsilon_{D,L-i}  \leq \left(\prod_{j=L-i+1}^{k} a_j \right) \varepsilon + b_{L-i+1} +  \displaystyle \sum_{k=L-i+2}^{L} \left (\prod_{j=L-i+1}^{k-1} a_j \right ) b_{k} 
\end{equation}
where $a_i = \kappa \alpha_i^2 \beta_i^4$ and $b_i = 2\kappa \alpha_i^3 \beta_i^3 \epsilon_{OD,i}$. 
\end{proposition}
\begin{proof}
See~\ref{proof_propfroberr}. 
\end{proof}

The above error estimate, indeed, is the error in one larger block as formed with accumulation of smaller size blocks. For example to seed the recursive formula for $n=10^4$, we first compute the error in the $n=10^3$ block, which is computed as $\epsilon_{D,0}$ starting from a $n_{min}=100$ block (which involves direct inversion); $\epsilon_{D,0}$ in that case is computed via the following direct formula:
\begin{equation}\label{direct_error2}
\epsilon_{D,L-i}  \leq b_{L-i+1} +  \displaystyle \sum_{k=L-i+2}^{L} \left (\prod_{j=L-i+1}^{k-1} a_j \right ) b_{k} 
\end{equation}
where simply $\varepsilon$ is set to zero in the beginning of the chain of errors (i.e. in~\eqref{direct_in_text}). Having the error in the matrix inverse, the error in the solution of linear solve $\bm A \bm x = \bm y$ is computed as $f_{\epsilon}(\bm x) = \epsilon_{D,0} f_{\alpha}(\bm y)$ where $\epsilon_{D,0}$ is the error in the inversion of the largest block, i.e. the original matrix.

\subsection{Cost estimate}\label{sec:cost}

The cost estimate for inversion in this section follows the ideas for cost analysis and estimates in~\cite{gra03}. As a result, the inversion cost is obtained based on the complexity of  $\mathcal{H}$-matrix multiplication operations (specifically, in our case, $\mathcal{H}$-matrix-vector multiplication) and that itself is bounded via storage cost for $\mathcal{H}$-matrices in addition to the cost of performing matrix reduction. We do not show the detailed constants associated with these bounds and only present the main complexity terms that give rise to the computational complexity in our hierarchical construction. In particular, it is shown that the computational complexity of finding $\bm x$ (in $\bm A \bm x = \bm y$) is in the form of $\mathcal{O}(n \log(n) k)$ in our HMAT approach.

First, we briefly mention the storage cost for general matrices i.e. $\mathcal{O}(mn)$ and $\mathcal{O}(k(m+n))$ for full and reduced matrices and cost of SVD analysis i.e. $\mathcal{O}(k^2(n+m))$ cf.  Algorithm~\ref{alg:rsvd_id}. Next, we present the sparsity constant and the idea of small and large leaves, which will be used to bound the storage cost. For the formal definitions of $\mathcal{H}$-tree, block $\mathcal{H}$-tree and the set of $\mathcal{H}$-matrices (denoted by $\mathcal{H}(\bm T,k)$) see Section $1$ of~\cite{gra03}.  

\begin{definition}[sparsity constant]\label{Csp}
Let $\bm T_{\mathcal{I} \times \mathcal{J}}$ be a block $\mathcal{H}$-tree based on $\bm T_{\mathcal{I}}$ and $\bm T_{\mathcal{J}}$. The sparsity constant $C_{sp}$ of $\bm T_{\mathcal{I} \times \mathcal{J}}$ is defined by
\begin{equation}
C_{sp} \coloneqq \max \{ \max_{r \in \bm T_{\mathcal{I}}} \# \{s \in \bm T_{\mathcal{J}} | r \times s \in \bm T_{\mathcal{I} \times \mathcal{J}} \}, \max_{s \in \bm T_{\mathcal{J}}} \# \{r \in \bm T_{\mathcal{I}} | r \times s \in \bm T_{\mathcal{I} \times \mathcal{J}} \} \}. 
\end{equation}
\end{definition}

In simple terms, the constant $C_{sp}$ is the maximum number of sets in an individual $\mathcal{H}$-tree at a certain level that give rise to a number of blocks in the block $\mathcal{H}$-tree. In other words, the total number of blocks in the block $\mathcal{H}$-tree, based on $\bm T_{\mathcal{I}}$ and $\bm T_{\mathcal{I}}$,  is bounded by $C_{sp}$ multiplied by the total number of (index) sets that exist in the individual $\mathcal{H}$-tree. The maximum number of blocks at any level in our hierarchical construction is $2$; therefore, an equivalent sparsity constant is $C_{sp}=2$. 

\begin{definition}[small and large leaves]
Let $\bm T$ be a block $\mathcal{H}$-tree based on $\bm T_{\mathcal{I}}$ and $\bm T_{\mathcal{J}}$. The set of \emph{small} leaves of $\bm T$ is denoted by $\mathcal{L}^{-}(\bm T) \coloneqq \{ r \times s~|~ \# r \leq n_{min} ~\text{or}~ \# s \leq n_{min}\}$, and the set of \emph{large} leaves is denoted as $\mathcal{L}^{+}(\bm T) \coloneqq \mathcal{L}(\bm T) \backslash \mathcal{L}^{-}(\bm T)$.
\end{definition}
The above definition formally distinguishes between the leaves that admit full ($\mathcal{L}^{-}(\bm T)$) and low-rank representation ($\mathcal{L}^{+}(\bm T)$). These leaves or blocks are clearly apparent in Figure~\ref{Hmat_merged}. The gray ones belong to $\mathcal{L}^{-}(\bm T)$ and the rest of blocks belong to $\mathcal{L}^{+}(\bm T)$. In addition to the above definitions, we provide the computational complexity of operations within the SMW formula as discussed in Section~\ref{Sec4_1} cf. Table~\ref{Table1}. 
\begin{table}[!h]
\caption{Computational complexity of various terms in SMW linear solve computations.}
\normalsize
\centering
\begin{tabular}{c c c c c c}
\hline\hline
  Variable & $\bm x_{D_i}$  & $\bm q_{l_i}$  & $\bm q_{lr_i}$ & $\bm q_{ry_i}$ & $\bm s_{qry_1}$  \\
\hline
Computational complexity &	$n^2$	 & $k n^2$ & $k^2 n$ & $kn$ & $k^2 $ 	\\
\hline
\end{tabular}
\label{Table1}
\end{table}
The estimates in Table~\ref{Table1} are based on the assumption that for small matrices, the inversion has the same complexity as multiplication. This assumption is one building block of the proof in Theorem 2.29 in~\cite{gra03}. This assumption can be true when considering iterative techniques such as a conjugate gradient on small linear solve problems. In those cases, the solution is found via a possibly small number of iterations involving matrix-vector operations with $n^2$ complexity. Among various terms in Table~\ref{Table1}, the most significant cost is associated with the computation of $\bm q_{l_i}$ which is equivalent to the cost of a linear solve with a square matrix with size $\mathbb{R}^{n \times n}$ and a right-hand side matrix with size $\mathbb{R}^{n \times k}$, i.e. $\bm q_{l_i} = \bm A^{-1} \bm \Gamma$. 

Next, we present the result for storage complexity which is used to bound matvec multiplication in hierarchical matrices. In the next result, the number of degrees of freedom $\# \mathcal{I}$ is replaced with $n$ and the number of levels in the tree $\# L$ is replaced with $\log(n)$. The storage requirement for a hierarchical matrix whose off-diagonal blocks have rank $k$ is denoted by $C_{\mathcal{H},St}(\bm T,k) $.
\begin{lemma}[storage complexity]\label{lemma:storage_complexity}
 Let $\bm T$ be a block $\mathcal{H}$-tree based on $\bm T_{\mathcal{I}}$ and $\bm T_{\mathcal{I}}$ with sparsity constant $\bm C_{sp}$ and minimal block size $n_{min}$. Then, the storage requirements for an $\mathcal{H}$-matrix  are bounded by:

\begin{equation}\label{eq:storage_complexity}
\begin{array}{l l l}
C_{\mathcal{H},St}(\bm T,k) & \leq & 2 C_{sp} n_{min}(\# \mathcal{I} \#L) \\
\\
& \leq & C_1 n \log(n)
\end{array}
\end{equation} 
where $C_1 \coloneqq 2C_{sp} n_{min}$.
\end{lemma}

\begin{proof}
See~\ref{prooflemma:storage_complexity}.
\end{proof}

\begin{lemma}[matrix-vector product]\label{lemma:matvec}
Let $\bm T$ be a block $\mathcal{H}$-tree. The complexity $C_{\mathcal{H} \cdot v}(\bm T,k)$ of the matrix-vector product for an $\mathcal{H}$-matrix  $\bm A \in \mathcal{H}(\bm T,k)$  can be bounded from above and below by
\begin{equation}
C_{\mathcal{H},St}(\bm T,k) \leq  C_{\mathcal{H} \cdot v}(\bm T,k) \leq 2 C_{\mathcal{H},St}(\bm T,k).
\end{equation}
\end{lemma}
\begin{proof}
See~\ref{prooflemma:matvec}.
\end{proof}

\begin{lemma}[truncation]\label{lemma:truncation}
Let $\bm T$ be a block $\mathcal{H}$-tree based on the $\mathcal{H}$-trees $\bm T_{\mathcal{I}}$ and $\bm T_{\mathcal{J}}$. A trancation of an $\mathcal{H}$-matrix  $\bm A \in \mathcal{H}(\bm T,k)$ can be computed with complexity $C_{\mathcal{H},Trunc} \leq k C_{\mathcal{H},St}(\bm T,k)$.
\end{lemma}
\begin{proof}
See~\ref{prooflemma:truncation}.
\end{proof}

Using the results in Lemmata~\ref{lemma:storage_complexity},~\ref{lemma:matvec} and~\ref{lemma:truncation}, we bound the cost of linear solve $\bm A \bm x = \bm y$ in the form of $\mathcal{O}(n\log(n) k)$:  

\begin{theorem}\label{theorem:cost}
Let $\bm T$ be a block $\mathcal{H}$-tree of the index set $\mathcal{I} \times \mathcal{I}$ with $\# \mathcal{I} = n$, $\bm A \in \mathcal{H}(\bm T,k)$ an $\mathcal{H}$-matrix and $\bm y \in \mathbb{R}^n$ be a vector. We assume that the linear solve associated with the small blocks  $r \times s \in \mathcal{L}^{-1}(\bm T)$ has the same complexity as matrix-vector multiplication. Then, the cost of computing $\bm A^{-1}\bm y$ denoted by $C_{\mathcal{H}, \backslash}$ is bounded by
\begin{equation}
C_{\mathcal{H}, \backslash} \leq 3 C_1 k n \log(n).
\end{equation}
\end{theorem}
\begin{proof}
The proof adopts the same idea as the proofs of Theorems 2.24 and 2.29 in~\cite{gra03}, which bound the complexity of factorizing the inverse of an $\mathcal{H}$-matrix with the complexity of matrix-matrix multiplication for two $\mathcal{H}$-matrices. The cost involves two main parts: 1) the truncation cost for the $\mathcal{H}$-matrix; and once the matrix is in the hierarchical form i.e. large leaves are reduced matrices, 2) the cost of the inversion according to SMW formula~\eqref{smw_eq} bounded by the matrix-vector multiplication complexity. The statement in the second part is proven by induction over the depth of the tree. 
\begin{itemize}
\item Truncation of a full matrix to $\mathcal{H}$-matrix: This cost is provided in Lemma~\eqref{lemma:truncation}: $C_{\mathcal{H},Trunc} \leq k C_{\mathcal{H},St}\stackrel{\text{Lemma}~\eqref{lemma:storage_complexity}}{\leq}  C_1 k n \log(n)$.
\item Cost of hierarchical inversions (linear solves) for an $\mathcal{H}$-matrix: To explain the induction, consider a hierarchical matrix that at the deepest level is a single square small block. The matrix gets larger by adding (similar size) diagonal and off-diagonal blocks to previous levels, similarly to the discussion for hierarchical error estimation. It is assumed that the cost of linear solve for small blocks is bounded by matrix-vector multiplication. This is justified in~\cite{gra03} by considering the smallest possible scenario which is $n_{min}=1$. In this scenario, both inversion and multiplication involve one elementary operation. However, in our case the base of the induction is satisfied if $n^2 < kn\log(n)$ i.e. $n < k \log(n)$~\footnote{$n^2$ is the complexity of the matvec operation and $kn\log(n)$ is the target bound.}, which is typically satisfied for small enough $n$ and large enough $k$. 

For larger matrices, we assume that a large $\bm A_{22}$ yields the linear solve $\bm q_{l_2}$ with complexity bounded by $k C_{\mathcal{H} \cdot v}$. We make this assumption by considering that computing $\bm q_{l_2} \in \mathbb{R}^{n \times k}$ involves $k$ time $\mathcal{H}$-matrix-vector multiplication i.e. $k C_{\mathcal{H} \cdot v} \leq 2C_1 k n \log(n)$. Note that the most significant cost in SMW computations is associated with computation of $\bm q_{l_2}$. This term also appeared as the most significant term in our error analysis. Adding another block i.e. going up in the hierarchy, we perform computations associated with a smaller block plus SMW correction computations that have smaller complexity compared to computation of $\bm q_{l_2}$. Therefore, it is concluded that the complexity of the inversion of an $\mathcal{H}$-matrix via SMW computations is bounded by  $2C_1 k n \log(n)$.
\end{itemize}
Adding two parts of the computational complexity yields $C_{\mathcal{H}, \backslash} \leq 3 C_1 k n \log(n)$.
\end{proof}

\subsection{Computation of log likelihood and its derivative}~\label{Sec4_4}
In this subsection, we discuss the computation of the energy term $\bm y^T \bm A^{-1} \bm y$, the log of the determinant, i.e. $\log(\text{det}(\bm A))$ and their sensitivities. Before proceeding further, we emphasize that all sensitivity derivations stem from two main expressions
\begin{equation}
\begin{array}{l l}
\displaystyle \frac{\partial \bm A^{-1}}{\partial \ell} = \displaystyle  -{\bm A^{-1}} \frac{\partial \bm A}{\partial \ell} \bm A^{-1}, &
\displaystyle  \frac{\partial \log(\text{det}(\bm A))}{ \partial \ell } =  Tr\left(\bm A^{-1} \displaystyle \frac{\partial \bm A}{\partial \ell} \right )
\end{array}
\end{equation}
with the application of the chain rule.

To compute the log determinant, we introduce the following notation: Denoting the eigenvalues of a square matrix $\bm A \in \mathbb{R}^{n \times n}$ by $\{\sigma_i\}_{i=1}^n$, the log of determinant is computed by $\Lambda = f_{\Lambda}(\bm A) = \log(\text{det}(\bm A)) =\log( \prod_{i=1}^n \sigma_i )=\sum_{i=1}^n \log(\sigma_i)$.

Multiplying $[\bm y^T_{01}~ \bm y^T_{02}]$ (where the notation $\bm y_{01}$ denotes the first column of $\bm y_1$) to the left of Equation~\eqref{smw_eq} and taking the $\log$ of Equation~\eqref{smw_formula}.b, i.e. $\log(\text{det}(\bm A))= \log(\text{det}(\bm A_D)) + \log(\text{det}(\bm C))+\log(\text{det}(\bm C^{-1}+\bm V \bm A_D^{-1} \bm U))$, the energy term and log determinant are computed hierarchically as 
\begin{equation}~\label{smw_eq_energy}
\begin{array}{l l}
(\text{Energy}) & {\Pi}=\Pi_1 + \Pi_2   - \left([\bm y^T_{01} \bm q_{l1} \quad \bm y^T_{02} \bm q_{l_2}]~\bm C^{-1}_{smw}  \left [ \begin{array}{l}
 \bm q^T_{l_2} \bm y_{02}  \\
 \\
 \bm q^T_{l_1} \bm y_{01}
 \end{array} \right ] \right) \\
(\text{Sum~log~eigs}) & \Lambda = \Lambda_1 + \Lambda_2 + 2 f_{\Lambda}(\bm \Gamma_m) + f_{\Lambda}\left(\bm C_{smw} \right )
 \end{array}
\end{equation}
where  $\Pi_1 \coloneqq  \bm y^T_{01} \bm x_{D_{01}} \in \mathbb{R},~\Pi_2 \coloneqq  \bm y^T_{02} \bm x_{D_{02}} \in \mathbb{R},~\Lambda_1 \coloneqq f_{\Lambda}(\bm A_{11}) \in \mathbb{R},~\Lambda_2 \coloneqq f_{\Lambda}(\bm A_{22}) \in \mathbb{R}$. Defining $T_1= \Pi_1 + \Pi_2$ (T for term), and
\begin{equation}
\begin{array}{l c l}
T_{21} =  [\bm y^T_{01} \bm q_{l1} \quad \bm y^T_{02} \bm q_{l_2}] ,& 
T_{22} = \bm C^{-1}_{smw},  &
T_{23} = [\bm y^T_{02} \bm q_{l_2}  \quad \bm y^T_{01} \bm q_{l_1} ] ^T
\end{array}
\end{equation}
sensitivities of $\Pi$ and $\Lambda$ are written as
\begin{equation}\label{pi_lambda_der}
\begin{array}{l l l}
\displaystyle \frac{\partial \Pi}{\partial \ell} & = &\displaystyle  \frac{\partial \Pi_1}{\partial \ell} + \frac{\partial \Pi_2}{\partial \ell} - \frac{\partial T_{21}}{\partial \ell} T_{22} T_{23}-T_{21} \frac{\partial T_{22}}{\partial \ell} T_{23} -  T_{21} T_{22} \frac{\partial T_{23}}{\partial \ell}\\
\\
\displaystyle \frac{\partial \Lambda}{\partial \ell} &=& \displaystyle   \frac{ \partial \Lambda_1} {\partial \ell} + \frac{ \partial \Lambda_2} {\partial \ell} + 2~\text{sum}\left(\text{diag}({\partial \bm \Gamma_m}/{\partial \ell}) \oslash \text{diag}(\bm \Gamma_m) \right )+ Tr\left(\bm C^{-1}_{smw} \left[\frac{\partial \bm C_{smw}}{\partial \ell}\right]\right )
\end{array}
\end{equation}
where $\oslash$ denotes Hadamard division. The sensitivity computation for $\Pi$ and $\Lambda$ depends on several terms e.g. $\partial T_{21}/\partial \ell$ and ${\partial \bm C_{smw}}/{\partial \ell}$. The derivations for these terms are provided in Appendix~\ref{Appendix_loglkl_d}. The derivative terms such as ${\partial \Pi_1}/{\partial \ell}$ are computed hierarchically, and they depend on the hierarchical computation of $(\partial \bm A/\partial \ell)\bm x$ where $\bm x$ is the solution to the linear solve problem which can be in the form of a vector or matrix. The key part of sensitivity computation is the hierarchical computation of $(\partial \bm A/\partial \ell)\bm x$. The detailed pseudocodes for the calculation of likelihood and its gradient are provided in~\ref{App_algs}. 

\subsection{GP regression}~\label{Sec4_5} 
The (log) likelihood (cf. Equation~\eqref{first_eq}) optimization is possible after finding $\Pi$, $\Lambda$, and their gradients. To compute $\mathcal{L}$, we replace only $\bm A$ (in all previous expressions) with $\tilde{\bm A} \coloneqq \bm A + \sigma_n^2 \bm {I}$, where $\sigma_n$ is a fixed regularization parameter (or noise parameter in GP), set by the user. We also distinguish between training and test data by subscripts $._{\circ}$ and $._{\ast}$., i.e. $\bm y_{\circ}$ and $\bm y_{\ast}$, denote the vectors of training and test data, respectively.

We find the optimal hyperparameter $\ell^{\ast}$ by maximizing $\mathcal{L}$ (in particular, we minimize $-\mathcal{L}$). After finding the optimal hyperparameter $\ell^{\ast}$, we compute the mean and variance of the GP test dataset~\cite{Rasmussen06} via $\bm y_{\ast | \circ}  = \bar{\bm y}_{\ast} + \bm \Sigma_{\ast \circ} \bm \Sigma^{-1}_{\circ \circ} (\bm y_{\circ} - \bar{\bm y}_{\circ})$ and $\bm \Sigma_{\ast | \circ}  =  \text{diag}(\bm \Sigma_{\ast \ast}) -  \text{diag}(\bm \Sigma_{\ast \circ} \bm \Sigma_{\circ \circ}^{-1}  \bm \Sigma_{\circ \ast})$
where $\bar{\bm y}_{\circ}$ is the sample average of the training dataset and $\bar{\bm y}_{\ast}$ is defined analogously. As the sample average of the test dataset is not known \emph{a priori}, we assume a simple model for this average and set $\bar{\bm y}_{\ast} = \bar{\bm y}_{\circ}$. The covariance matrix $\bm \Sigma_{\circ \circ} \equiv \bm A + \sigma_n^2 \bm{I}$ is the hierarchical matrix that is inverted in a linear solve format, which is the main subject of this paper. Computation in the first equation is straightforward as  $\bm \Sigma^{-1}_{\circ \circ} (\bm y_{\circ} - \bar{\bm y}_{\circ})$ is already in a linear solve format. However, depending on the number of data points in the test dataset, the matrix $\bm \Sigma_{\ast \circ}$ may need to be reduced. In the GP-HMAT package, we have provided both options, one with the full computation on the covariance matrix between training and test datasets $\bm \Sigma_{\ast \circ}$ and one with reducing $\bm \Sigma_{\ast \circ}$ via randomized SVD with ID (cf. Algorithm~\ref{alg:rsvd_id}) using the rank parameter $k$ as set by the user for off-diagonal matrix factorization.

\section{Numerical experiments}\label{Sec5}

\subsection{Illustrative numerical example}\label{S5_1}
In this example, we illustrate some important aspects of our hierarchical approach. The majority of results throughout this example are reported using a fixed setting: The GP nodes are considered as uniformly distributed points in the 2D square $\bm N \in [0,1]^2$. The right-hand side $\bm y$ is also a randomly generated (uniformly distributed) vector with $\| \bm y\|_2=1$. The normalized errors are computed with respect to \emph{true} solutions that are obtained from direct inversions. We consider only one hyperparameter $\ell$ in this example and, if not stated, results are associated with the squared exponential kernel. We provide results associated with data in higher dimensions i.e. $d=10$ and $d=4$ in the second and third numerical examples cf. Sections~\ref{S5_1} and~\ref{S5_2}. A number of other empirical studies are provided in~\ref{numexI}.

\textbf{Smoothness of kernels and the effect of permutation on rank:} In the first experiment, we consider $n=1500$ nodes to compute the kernel matrix. In the left pane of Figure~\ref{smoothness_rank}, the decay of kernel value with respect to Euclidean distance for three different kernels, i.e. squared exponential, exponential, and a kernel that considers $l_1$ distance as a distance function, is shown. In the first two kernels, $f=\|\bm x_i - \bm x_j \|_2^2/\ell^2$ and we set $\ell=0.3$. To generate the data, we compute the kernel values $k(\bm x_0, \bm x)$ where $\bm x_0 =[0,0]$. From this figure, it is seen that the last kernel with $l_1$ distance function exhibits non-smooth decay with respect to $r$, i.e. the $l_2$ distance. The first two kernels are clearly more smooth in terms of decay of the kernel; however, the more important factor is the rank of submatrices, in particular, the off-diagonal blocks that are generated using these kernels. 

To study the rank of off-diagonal blocks, we decompose the nodes with size $\# \mathcal{I} = 1500$ to two aggregates with predefined sizes of $\# \mathcal{I}_1 = 500$ and $\# \mathcal{I}_2 = 1000$ and we consider two scenarios: 1) we do not permute the original nodes, i.e. we consider the first $500$ random nodes as the first aggregate and the rest i.e. $1000$ random nodes as the second aggregate; the singular value decay of the resulting kernel matrix $\bm A_{12}$ (i.e. $k(\bm x_1, \bm x_2)$ with $\bm x_1 \equiv \bm N(:,\mathcal{I}_1)$ and $\bm x_2 \equiv \bm N(:,\mathcal{I}_2))$  is shown in the middle pane of Figure~\ref{smoothness_rank}, and 2) we permute the original nodes with index $\mathcal{I}$ by applying the algorithm described in Section~\ref{S2_2}, referred to as \texttt{permute}. The singular value decay of the resulting kernel matrix $\bm A_{12}$ is shown in the right pane of Figure~\ref{smoothness_rank}.  

The rank of $\bm A_{12}$ (using the built-in ${rank}$ function) for three cases in the middle pane is $133,500,500$. After permutation with the \texttt{permute} algorithm, the resulting rank of $\bm A_{12}$ is $85,206,500$. It is apparent that the last case, i.e. the one with $l_1$ distance function is always a full rank matrix and the permutation has no impact on it. However, node permutation has a favorable effect on the rank of the squared exponential kernel and even more impact on the exponential kernel. 
\begin{remark_new}\label{rem_main1}
Determining the rank of a matrix $\bm A \in \mathbb{R}^{m \times n}$ (for large $m$ and $n$) is an intensive task. Assuming it is similar to the scalability of full SVD, the cost scales as $\mathcal{O}(nm^2)$ (where $m < n$). A more computationally efficient proxy to rank is the numerical (or stable) rank defined as $r(\bm A) = \|\bm A \|_F^2/\| \bm A \|^2_2$ where $r(\bm A) \leq rank(\bm A)$,~$\forall~\bm A$~\cite{Rudelson2007}. Although numerical rank is efficient, it is not informative in terms of revealing the actual rank of the matrix in the above experiment. Therefore, for unstructured GP matrices, the computation of actual rank is recommended.
\end{remark_new}

These results motivate more in-depth numerical research on the hierarchical low-rank partitioning of matrices e.g. possibly some direct nuclear norm optimization of the matrix. We deem our partitioning to be efficient, easily implementable, and sufficient for the scope of current paper. 

\begin{figure}[h]
\centering
\includegraphics[width=1.5in]{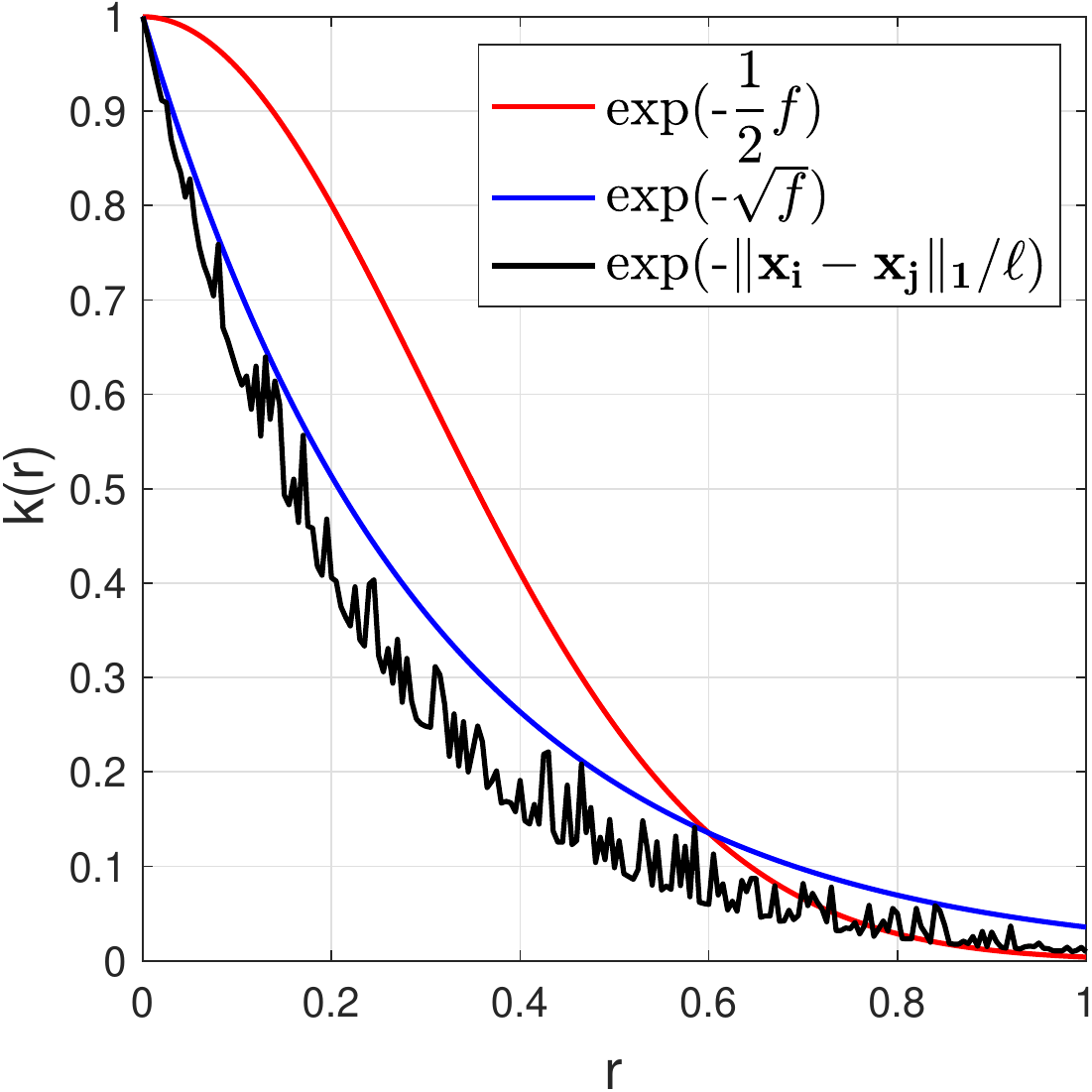}
\includegraphics[width=1.57in]{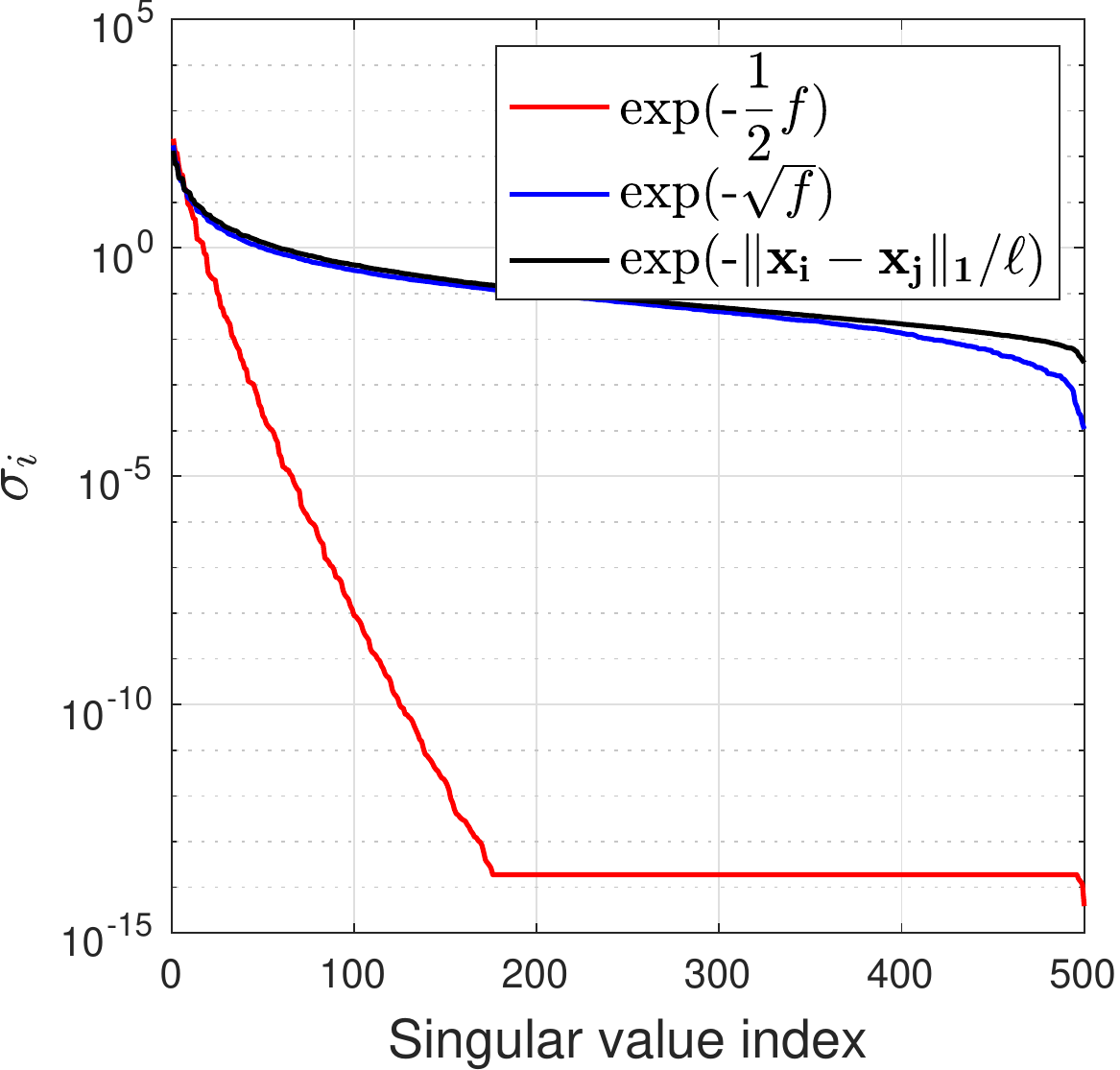}
\includegraphics[width=1.56in]{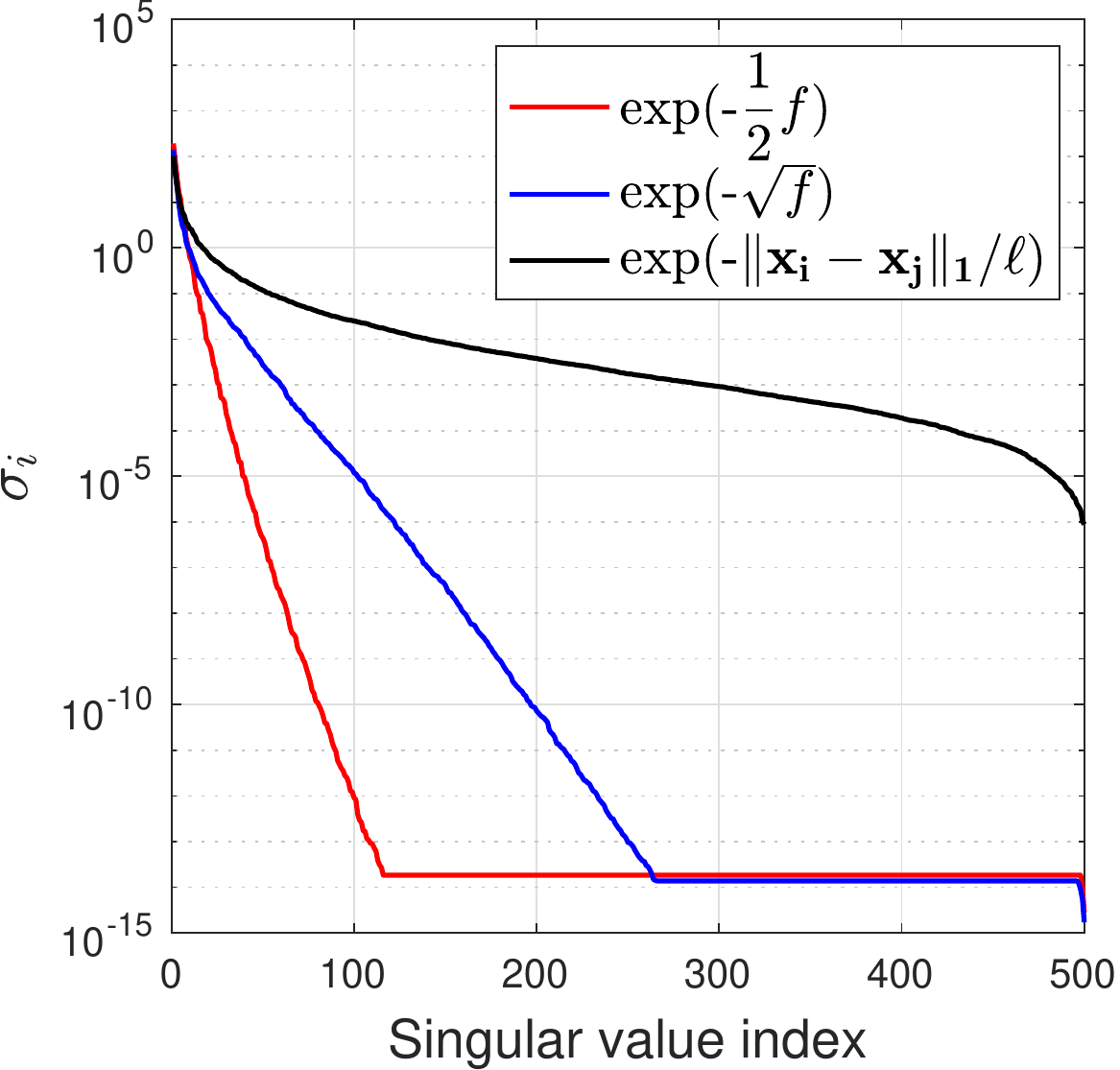}
\caption{\small{Kernel smoothness and the effect of permutation on three kernels.}}\label{smoothness_rank}
\end{figure}

\textbf{Practical Recommendation:} As a practical recommendation for utilizing the approach for inversion of different kernels, one can perform a simple similar test of one-level dyadic decomposition of a small size matrix and inspect rank $k$ of the off-diagonal block $\bm A_{12} \in \mathbb{R}^{m \times n}$ before and after application of the \texttt{permute} procedure (the cost is insignificant for small $m$ and $n$ cf. Remark~\ref{rem_main1}). If after permutation, $k \ll \min\{m,n\}$ (e.g. the red curves associated with squared exponential kernel) is observed, the overall framework will be successful however if the rank reduction is negligible i.e. $k \simeq \min\{m,n\}$ (e.g. the black curves associated with $l_1$ distances), the low-rank factorization of off-diagonal blocks yields inaccurate (or erroneous) results for inversion and determinant computation. In those instances, a different class of approaches, i.e. iterative approaches such as conjugate gradient, might be more effective; however, the rate of convergence in those approaches is also dependent on the decay rate of singular values~\cite{Saad2003}. 

\textbf{Comparison with respect to (1) aggregation strategy and (2) low-rank approximation:}
In this part of the example, we investigate the performance of two main building blocks of our framework with available approaches in the literature: \\
\textbf{(1) Aggregation strategy:} We test accuracy and timing of our aggregation strategy, \texttt{permute}, with the maximal independent set: MIS$(1)$ and MIS$(2)$~\cite{Bell12} cf.~\ref{App_MIS}. The test is performed on a matrix with $n=5000$ and $\eta=1050$. The MIS approach depends on a tunable aggregation parameter denoted by $\theta$. This parameter results in different numbers and sizes of aggregates. In each case, we sort the aggregates based on their size, select the largest as the first aggregate and the combination of the rest as the second one. In cases where the procedure results in one aggregate, we divide the aggregate into two equal size aggregates (roughly equal for odd sizes). The first plot in Figure~\ref{permute_acc} shows the accuracy in the solution of $\bm A \bm x = \bm y$ for two choices of coarse and fine aggregation parameter i.e. $\theta=0.3,~0.9$. The accuracy of the three approaches are almost similar (although the coarse $\theta$ exhibits larger errors), but to demonstrate the effectiveness of our approach we specifically measure only the aggregation time using the three approaches within the hierarchical construction for a wide range of $\theta$, i.e. $\theta \in [0.25,~0.95]$. The efficiency of \texttt{permute} is apparent compared to other approaches.  This experiment also shows that our aggregation strategy is not dependent on $\theta$. It always divides the input set to two sets with pre-defined sizes cf. Equation~\eqref{sl_size}, which results in more efficiency.  \\
\textbf{(2) Low-rank approximation:} As mentioned in Section~\ref{Sec1},  the Nystr\"{o}m approximation and its variants have been extensively used in the GP literature. To test this approach, we replace our low-rank factorization approach, randomized SVD with ID denoted by RSVD\_ID, with two versions of the Nystr\"{o}m approximation. In the Nystr\"{o}m approximation, an efficient factorization is achieved via $ k (\bm x_l, \bm x_r) \simeq  k(\bm x_l,\bm x_m)   k^{-1}(\bm x_m,\bm x_m)  k(\bm x_m,\bm x_r)$ where the nodes $\bm x_l \in \mathbb{R}^{d \times n_1}$ and $\bm x_r \in \mathbb{R}^{d \times n_2}$ are known. The main difficulty, however,  is the determination of  $\bm x_m \in \mathbb{R}^{d \times k}$.  We consider two scenarios for $\bm x_m$: 1) $k$ uniformly distributed random nodes $\bm x_m \in U[0,1]^2$ denoted by Nystr\"{o}m-Rand (similarly to~\cite{JieChen21}), and 2) first obtaining randomized range finder $\bm Q$ from $\bm Y = k(\bm x_l,\bm x_r) \bm \Omega$, i.e. $[\bm Q, \bm R_Y] \gets qr (\bm Y)$ and then finding $k$ pivots i.e. important row indices $\mathcal{I}_k$ by performing pivoted QR factorization on $\bm Q^T$ cf. Section~\ref{Sec3}, which yields $\bm x_m \gets \bm x_l(:,\mathcal{I}_k)$.  The second scenario is denoted by Nystr\"{o}m-QR. From the third plot in Figure~\ref{permute_acc}, which shows the normalized error in $\bm x$, it is again obvious that the randomized SVD approach significantly outperforms Nystr\"{o}m approximations. This empirical result is indeed in accordance with the fundamental result of the Eckart–Young theorem~\cite{Eckart36}, which states that the best rank $k$ approximation to $\bm A$ (in the spectral or Frobenius norm),  is given by its \emph{singular value decomposition} with $k$ terms.  
 
\begin{figure}[h]
\centering
\includegraphics[width=1.6in]{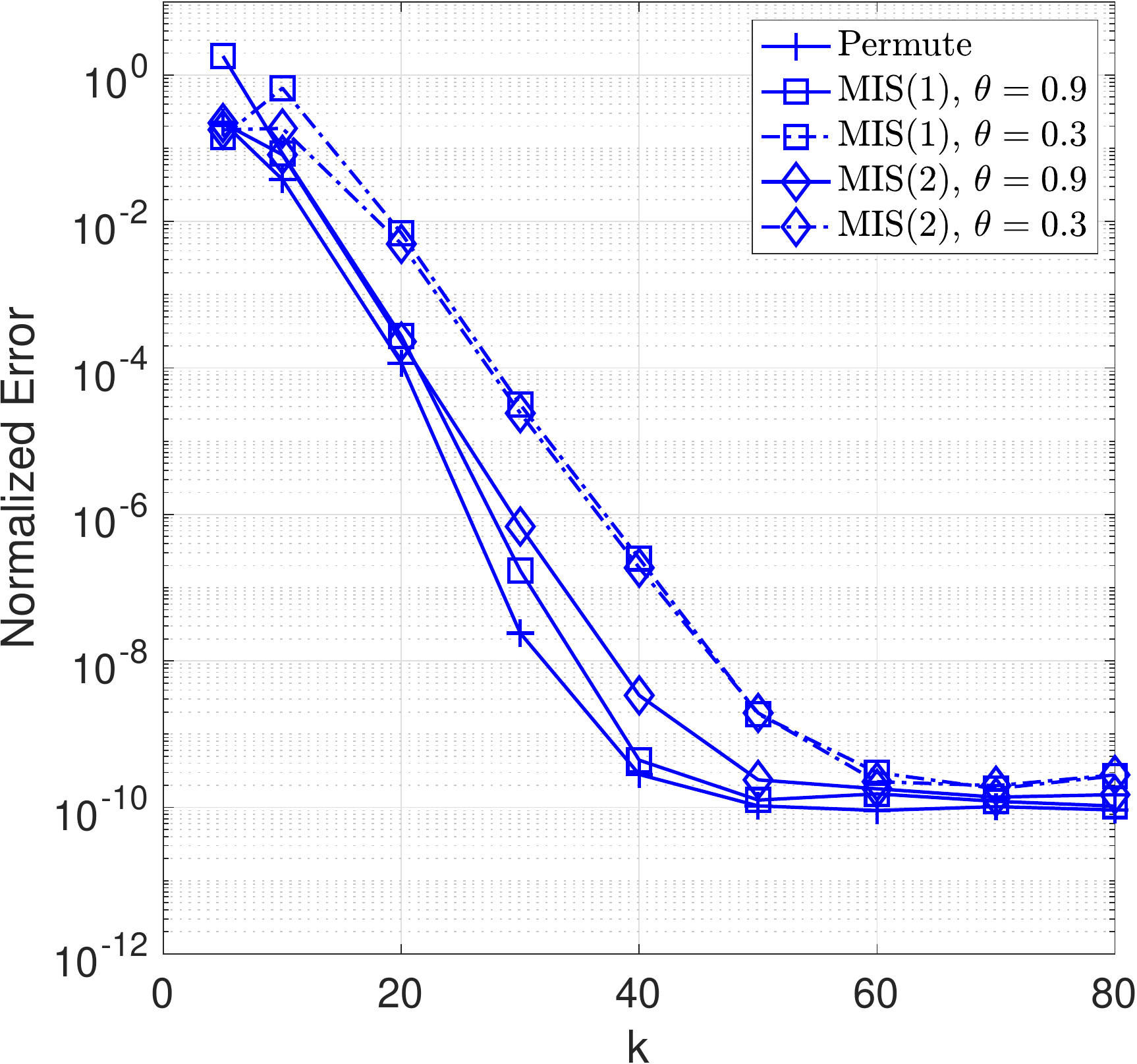}
\includegraphics[width=1.6in]{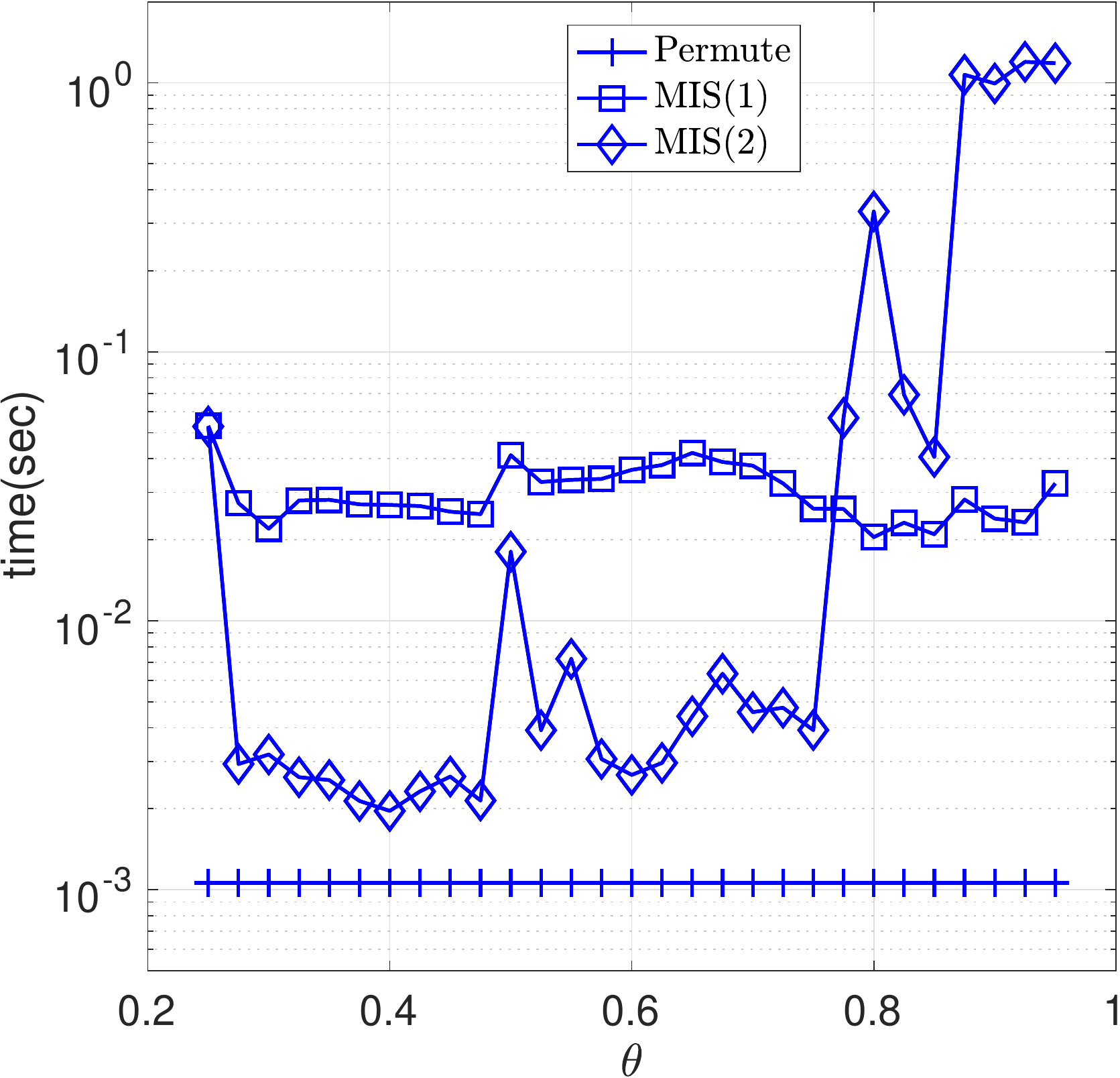}
\includegraphics[width=1.6in]{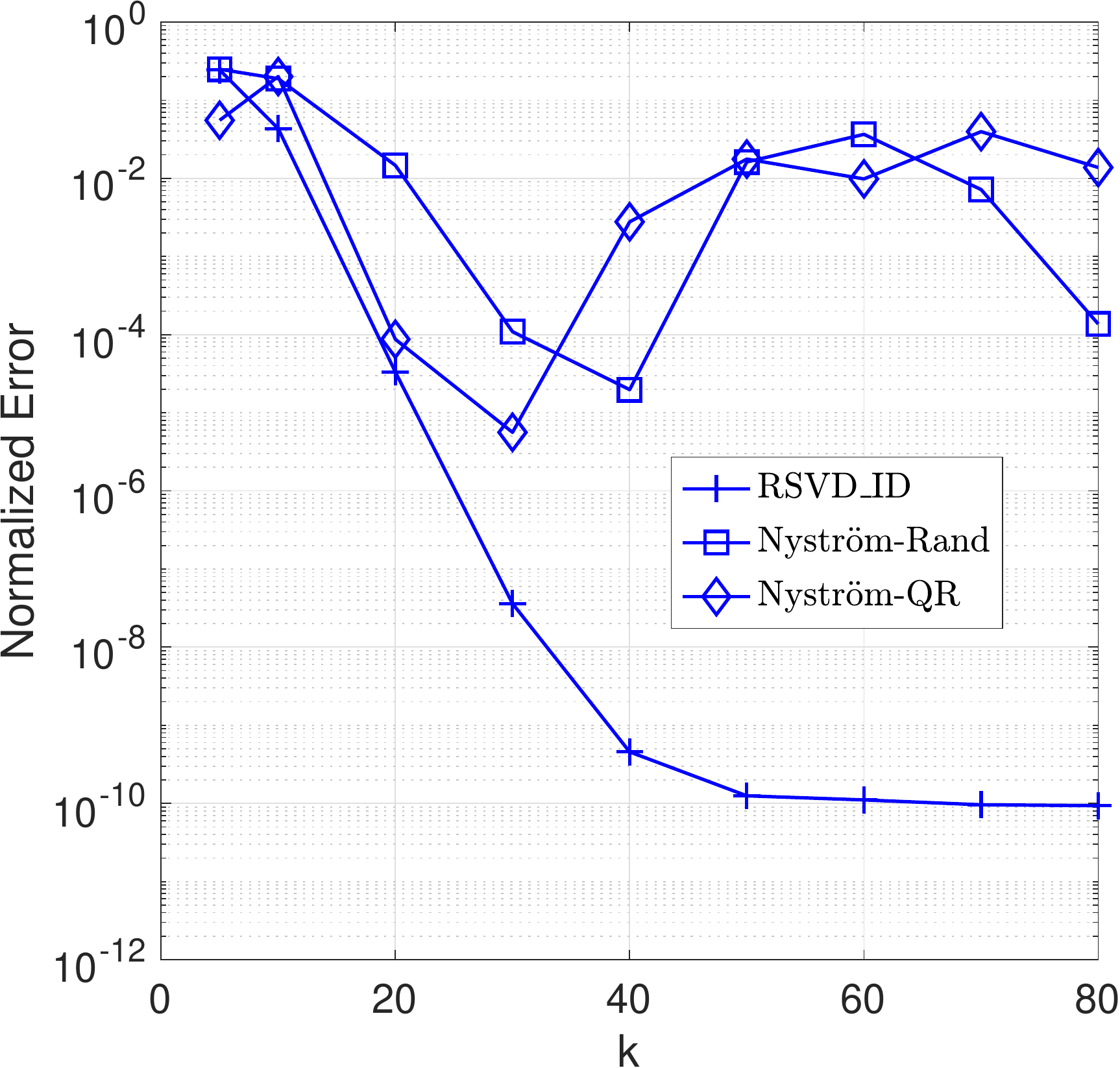}
\caption{\small{Accuracy and computation time of three aggregation strategies (left and middle panes). Comparison of accuracy between randomized SVD and Nystr\"{o}m approximations for low-rank factorization in the linear solve $\bm A \bm x= \bm y$.}}\label{permute_acc}
\end{figure}

\textbf{Empirical cost - scalability:} A major aspect of hierarchical matrix representation is the cost of the approach, which we empirically study via wall-clock time (real time) in this part of the example. In Section~\ref{sec:cost},
we found an estimate in the form of $\mathcal{O}(n \log(n) k)$ where $k$ is the rank in the off-diagonal blocks. The theoretical scalability has linear dependence on $n \log(n)$ and $k$. To study the scalability with respect to these terms, we time the code (via tic-toc in MATLAB) for different choices of $n$ and $k$. In particular, we consider $n=\{10^5,~2 \times10^5,~5 \times 10^5,~1\times 10^6\}$ and $k=\{5,~10,~20,~30,~50\}$, and set $\eta=1050$. The expression for $\nu$ is according to~\eqref{sl_size}.  

Figure~\ref{scalability} (left pane) shows the relationship between actual time and $n \log(n)$. For every choice of $k$, we find the slope of the fitted line (in the log-log scale), denoted by $r$, via least square on the pair $(\log(n\log(n)),\log(t))$ where $t$ is the wall-clock time. As seen, the scalability of the code slightly degrades as $k$ increases. We note that these are actual/empirical results from the code and they depend on the implementation and the platform in which the code is implemented. Nonetheless, for a relatively large $k$, i.e. $k=50$ we deem the slope $r \simeq 1.1$ and as a result the scalability of almost $\mathcal{O}(n \log(n))$, promising for our empirical results. 

We also compute the scalability with respect to $k$ for different choices of $n$, cf. middle pane in Figure~\ref{scalability}. While the relationship between $k$ and time in log-log scale is not necessarily linear as seen in the plot, we perform the same line fitting and find slopes for different cases of $n$. The slope for the scalability of $k$ is around $r=0.5$. Our theoretical cost estimate shows linear scalability with respect to $k$ while the reality is more promising i.e. the empirical scalability based on these experiments is $\mathcal{O}(n \log(n) \sqrt{k})$.

Finally, to compare the code with built-in MATLAB capabilities i.e. $\backslash$, as a baseline approach, we time kernel matrix computation, $\backslash$ execution, and the sum of two operations with respect to different sizes. Note that the data points are associated with significantly smaller $n$ as a direct inversion of larger size matrices e.g. $n=5\times 10^4$, can be very time consuming (or impossible). From these results, the HMAT solver is clearly superior in terms of both scalability and actual time for matrices with large sizes. A sufficiently accurate estimate (i.e. $k=50$) for $n=10^5$ is obtained in almost $10$ seconds with HMAT while using built-in capabilities, but it takes more than $10$ seconds to find a solution for only $n=1.5\times 10^4$ which involves both (vectorized, i.e. no for loop) kernel computation and $\backslash$.  
  
\begin{figure}[h]
\centering
\includegraphics[width=1.66in]{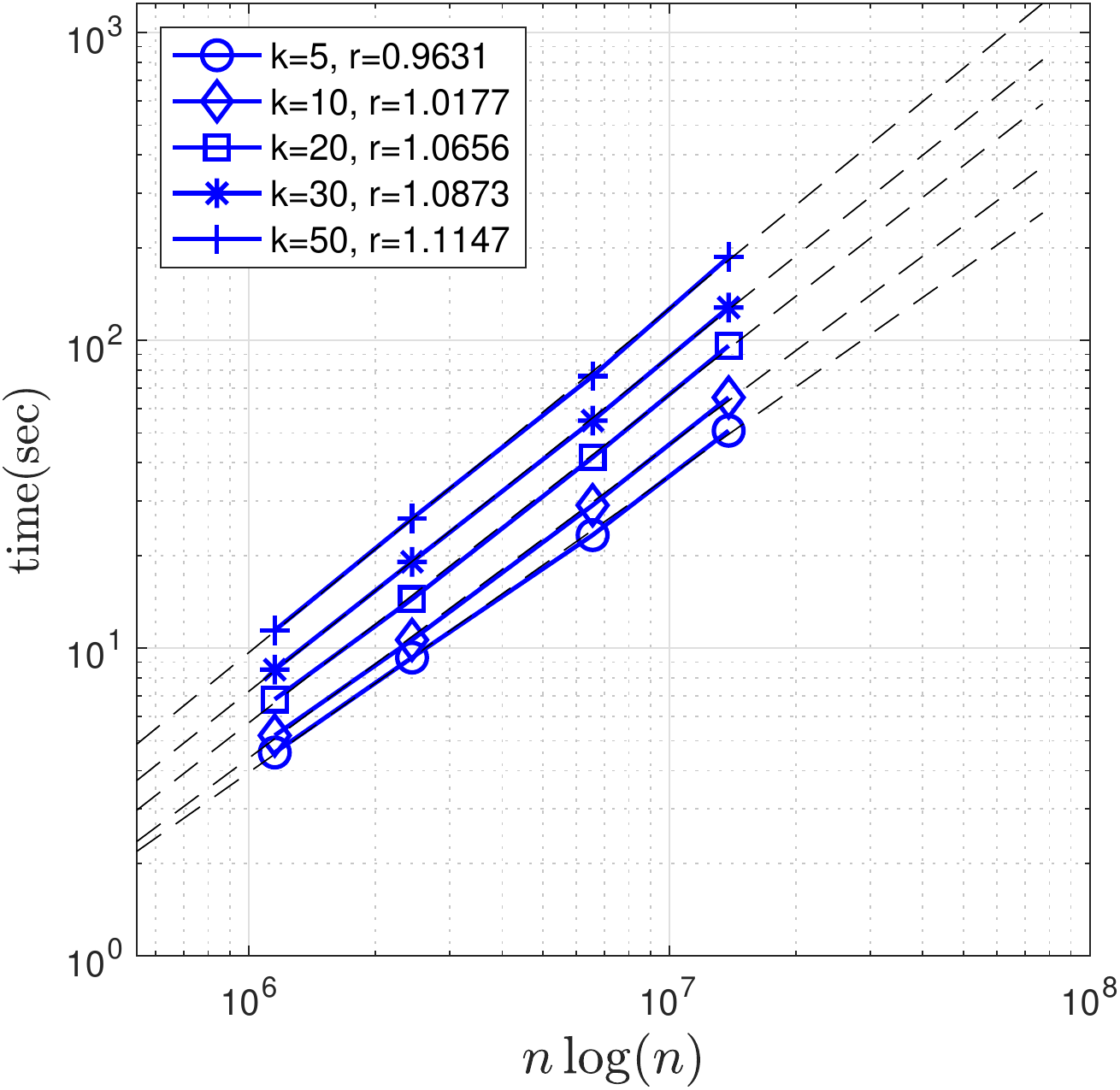}
\includegraphics[width=1.66in]{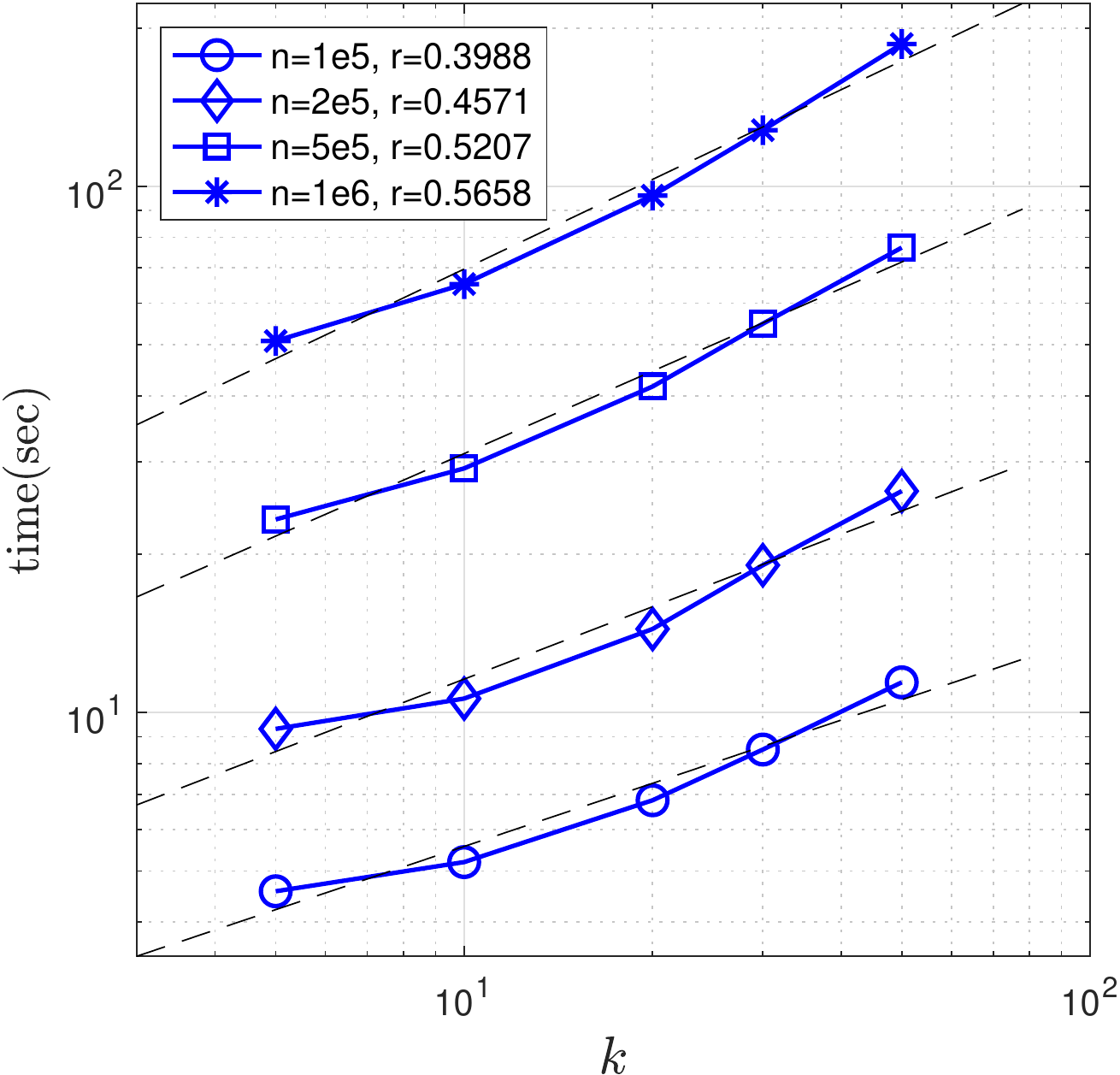}
\includegraphics[width=1.66in]{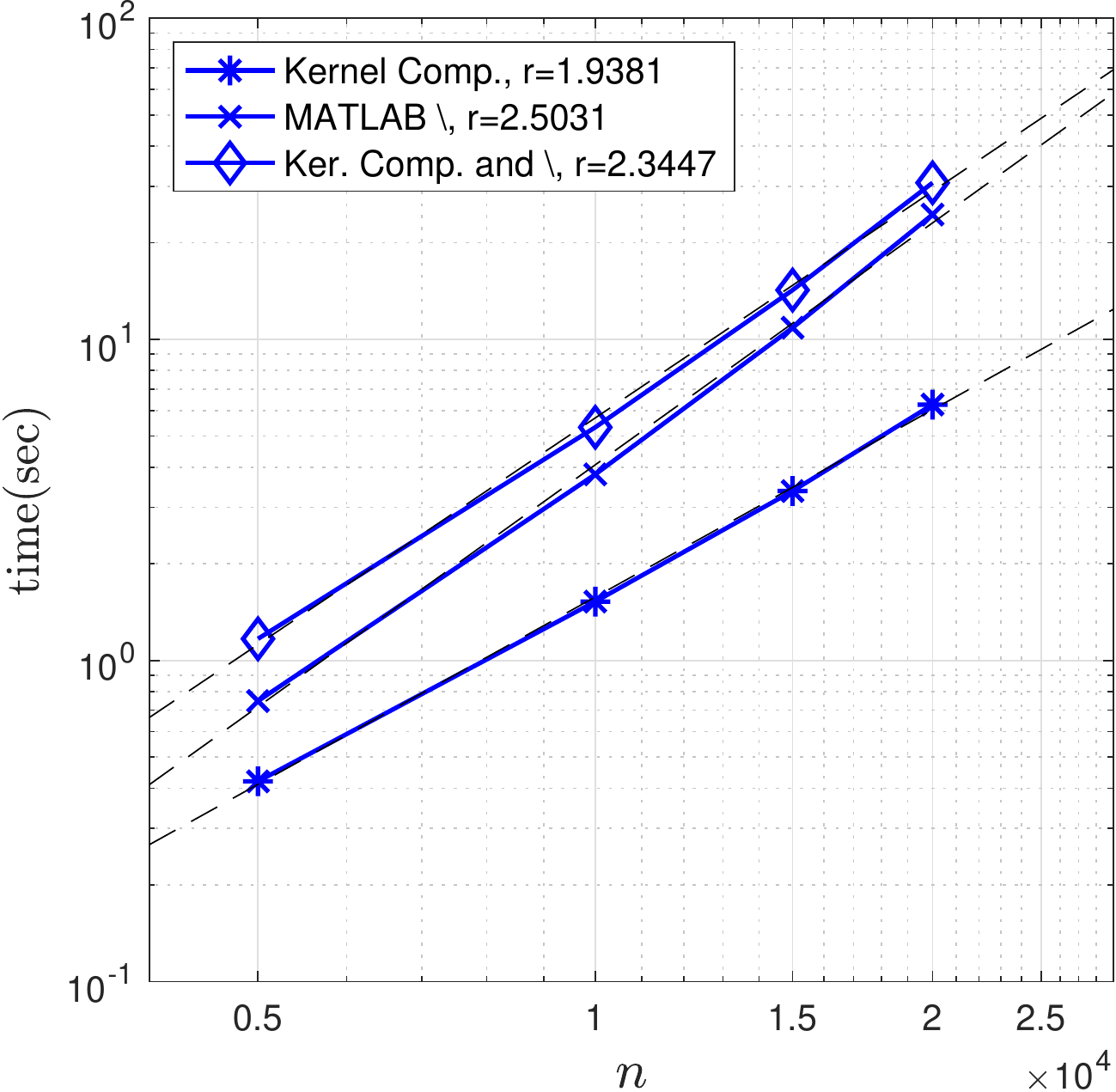}
\caption{\small{Scalability of HMAT solver with respect to size $n$ (left) and rank parameter $k$ (middle); scalability of $\backslash$ and full computation of kernel matrix (right).}}\label{scalability}
\end{figure}

\subsection{Regression on a double-phase linear elastic structure}\label{S5_2}
In the second numerical example, we consider regression on a parametric elliptic partial differential equation (PDE) where the output quantity of interest (QoI) is a spatially-averaged quantity. We particularly consider the average von Mises stress in a double-phase linear elastic structure. The details of the simulation for computing parametric average von Mises stress are provided in~\ref{numexII}. 

The target quantity is the spatial average of von Mises stress in 1) the full domain of the structure denoted by $q_1$ and 2) the region associated with the stiffer material denoted by $q_2$. We train two individual Gaussian processes with $10$-dimensional parametric variables $\bm \xi$ (which model elasticity field) as input (or feature) and $q_1$ and $q_2$ as target variables, i.e. $\bm \xi \stackrel{GP_1}\longrightarrow q_1$ and $\bm \xi \stackrel{GP_2}\longrightarrow q_2$.

To find the best GP model in this example, we perform a statistical boosting strategy. The total number of samples in each case of GP regression is $2000$ samples (in $10$ dimension, i.e. $\bm \xi \in [0,1]^{10}$ ). We consider $90 \%$ of the data points for training and $10\%$ for testing. We divide all data points into $10$ disjoint sets of $200$ samples for testing (in addition to $1800$ samples for training). Each individual case is associated with one model in the boosting study. To perform the hierarchical matrix computations, we consider $k=20$ and use the squared exponential kernel with the noise parameter $\sigma_n=10^{-3}$. 

The result of statistical boosting for optimal hyperparameters is briefly discussed in~\ref{numexII}, cf. Figure~\ref{optimal_ell}. The results of point-wise error for the best $\ell^{\ast}$ in two regression cases $q_1$ and $q_2$ are also shown in ~\ref{numexII}, cf. Figure~\ref{optimal_ell}.  The probability distribution functions (PDFs) of $q_1$ and $q_2$ for both test data and GP prediction are shown in Figure~\ref{vonmin_dist}. The agreement between the PDFs is apparent from this figure which is promising for general surrogate modeling purposes. The numerical values for the normalized error of prediction i.e. $e_{q_{i,\ast}}=\|\bm q_i - \bm q_{i,\ast}\|_2/\|\bm q_i\|_2,~i=1,2$ where $\bm q_i \in \mathbb{R}^{200 \times 1}$, and the normalized error in mean and standard deviation, i.e. $e_{\mu}$ and $e_{\sigma}$ are provided in Table~\ref{GP_vonmin}. 

\begin{figure}[h]
\centering
\includegraphics[width=1.8in]{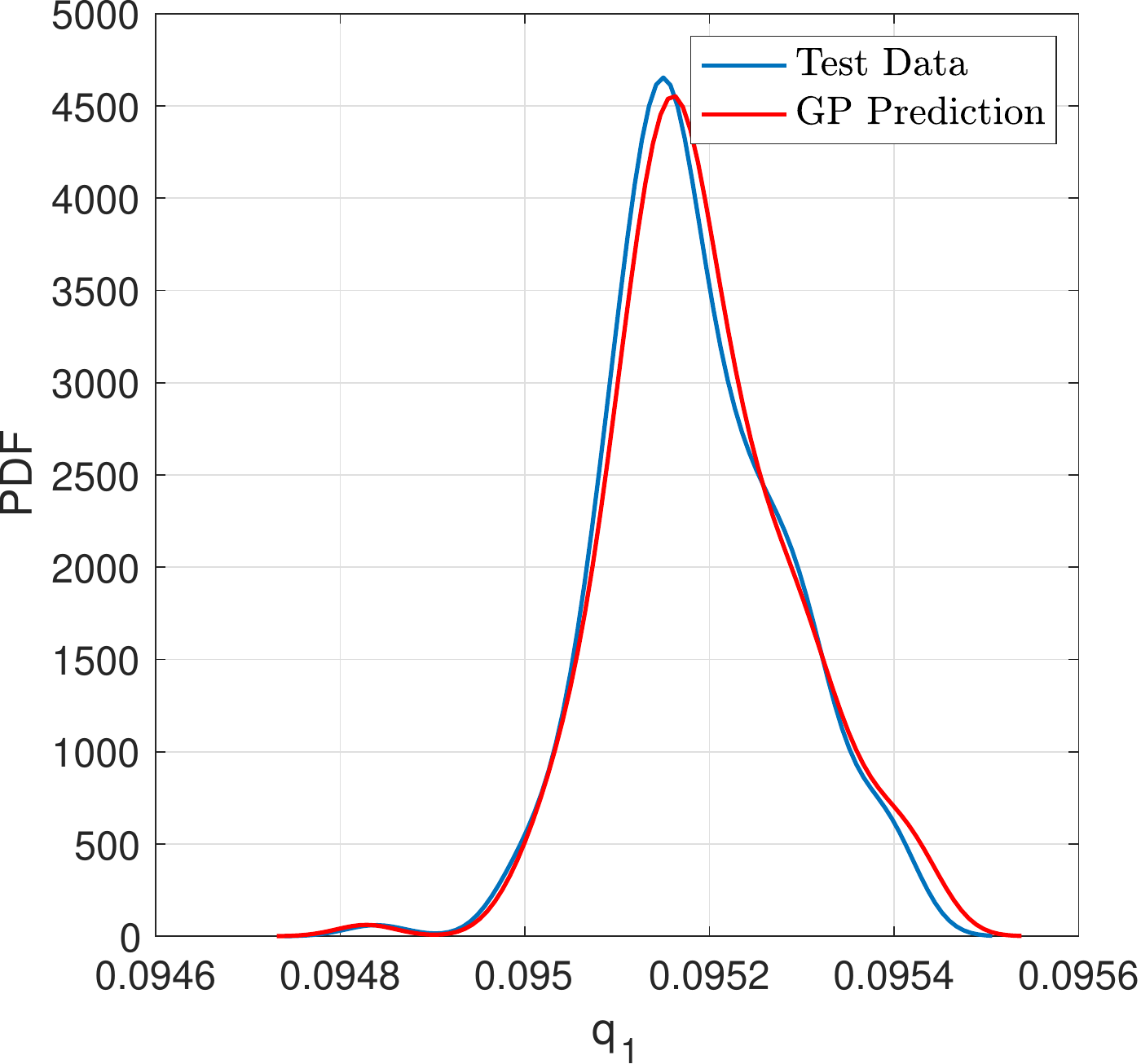}\includegraphics[width=1.75in]{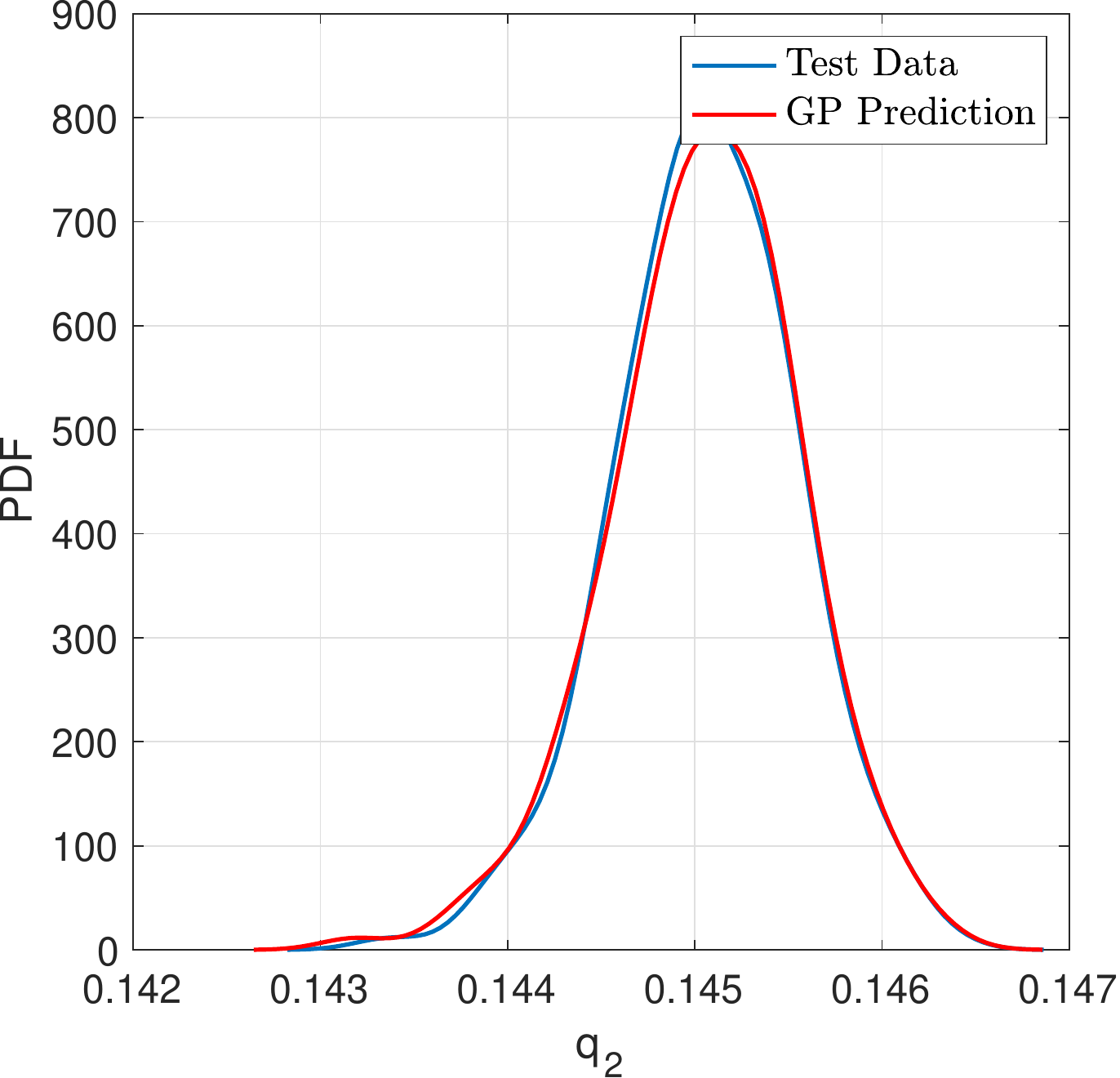}
\caption{\small{Probability distribution functions of test data and GP prediction for $q_1$ and $q_2$.}}\label{vonmin_dist}
\end{figure}

\begin{table}[!h]
\caption{Normalized error of prediction and normalized error in mean and standard deviation for two quantities $q_1$ and $q_2$.}
\normalsize
\centering
\begin{tabular}{l c c c c }
\hline\hline
  & $e_{q_{i,\ast}}$  & $e_{\mu}$  & $e_{\sigma} $     \\
\hline
Computation of $q_1$       &     $1.95 \times 10^{-4}$	   & $7.76 \times 10^{-5}$  &  $3.83 \times 10^{-2}$	\\
Computation of $q_2$       &     $6.41 \times 10^{-4}$   & $2.28 \times 10^{-5}$  &  $4.67 \times 10^{-2}$	\\
\hline
\end{tabular}
\label{GP_vonmin}
\end{table}

\subsection{Regression on NYC taxi trip record}\label{S5_3}

In this example, we study the performance of the approach on a relatively large-scale problem. Again, we perform the regressions with GP-HMAT on a single CPU. Our main objective is to show the applicability of the GP-HMAT approach in the regression of datasets that are not possible with the usual built-in capabilities of scientific computing platforms. 

In this example, we use the datasets available in~\cite{NYCTaxicab}. We use the yellow taxi trip records associated with December 2020. This dataset has $18$ features (including categorical and quantitative variables) with almost $1.4$ million data points. We consider four features, trip distance, payment type, fare amount, and tip amount, as the input to our regression model and consider the total amount as the target variable.  To use the dataset, we perform data cleansing by removing corrupt data points i.e. those including NaN (not a number) and negative feature/target values. We then extract $n=10^5$ data points randomly for performing the regression analysis. Out of these $n=10^5$ points, we randomly choose $n_{\circ}=9 \times 10^4$ points for training and $n_{\ast}=10^4$ for testing. The data points and GP implementation for this example are available in~\cite{Keshavarzzadeh_MATLAB_GPHMAT}. We normalize each dimension of the data with its max value and denote normalized input and output with $\{x_i\}_{i=1}^4$ and $y$, respectively.

Figure~\ref{dist_x_y} shows the PDF for the normalized feature variables as well as the normalized target variable. Note that the second feature is comprised of integer variables, i.e. $\text{payment type}=\{1,2,3,4\}$ which is normalized to $x_3=\{0.25, 0.5, 0.75, 1\}$ and therefore it is represented via the probability mass function (PMF). 

\begin{figure}[h]
\centering
\includegraphics[width=1.55in]{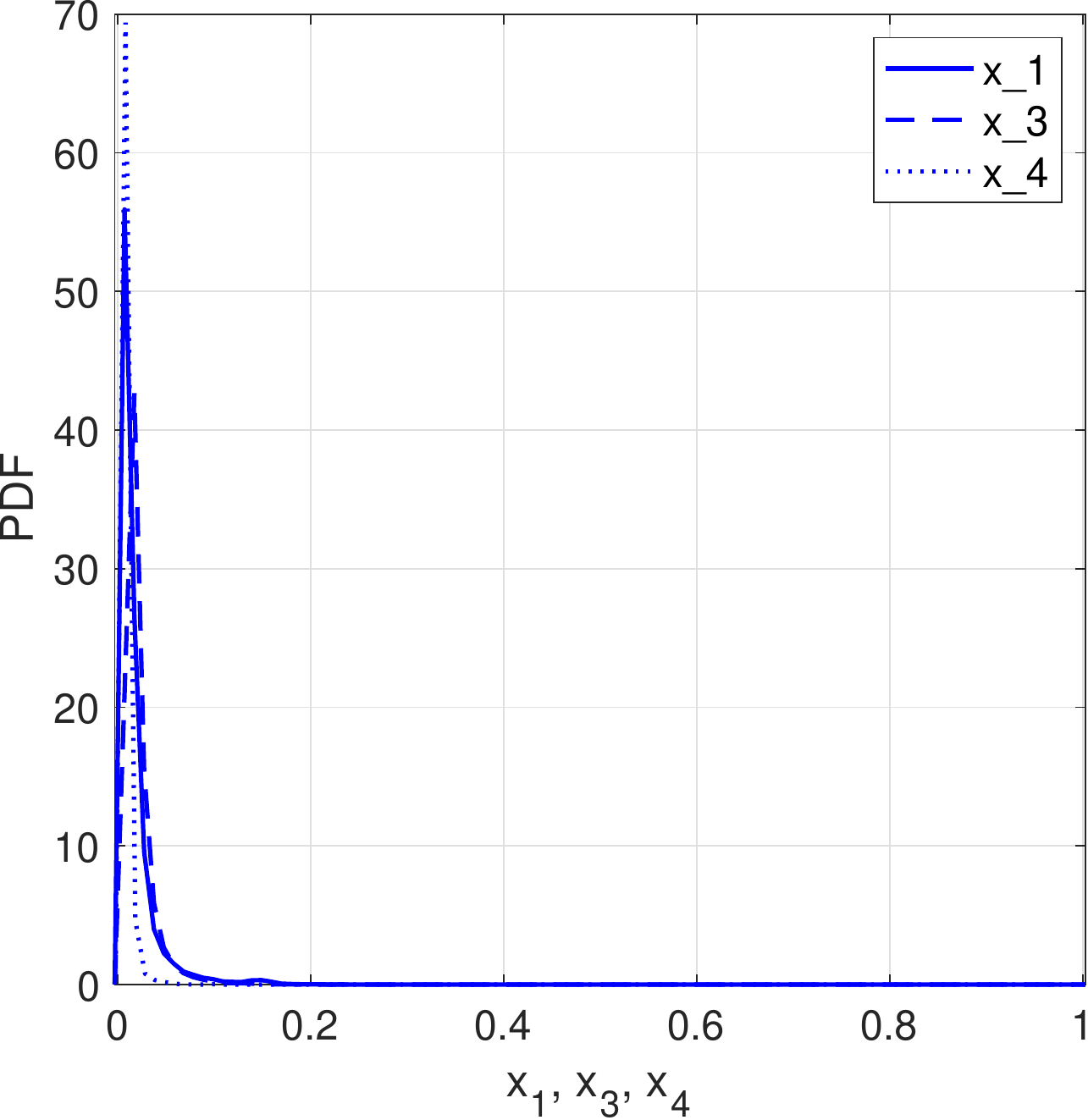}\hspace{0.1cm}\includegraphics[width=1.55in]{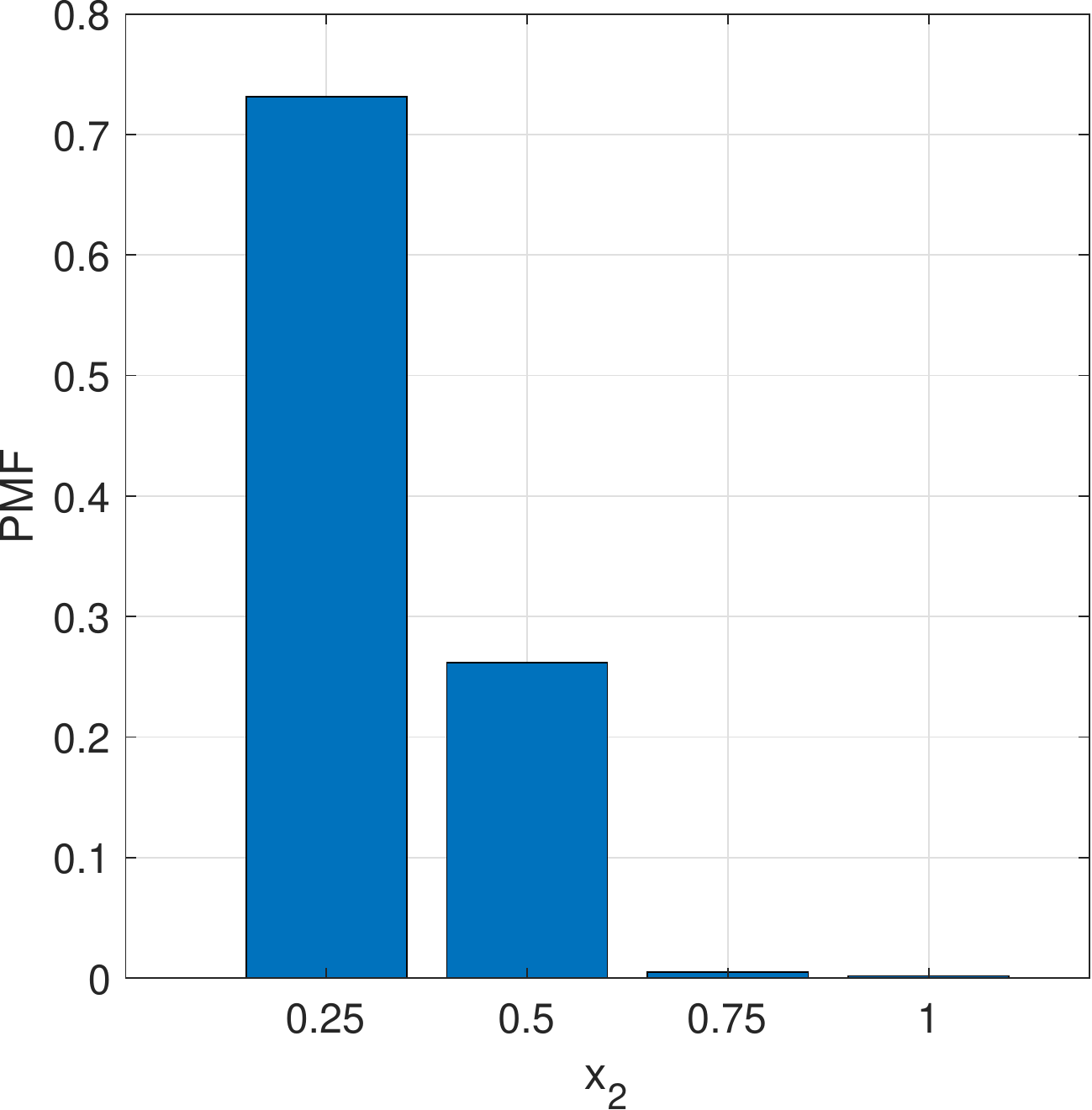}
\includegraphics[width=1.55in]{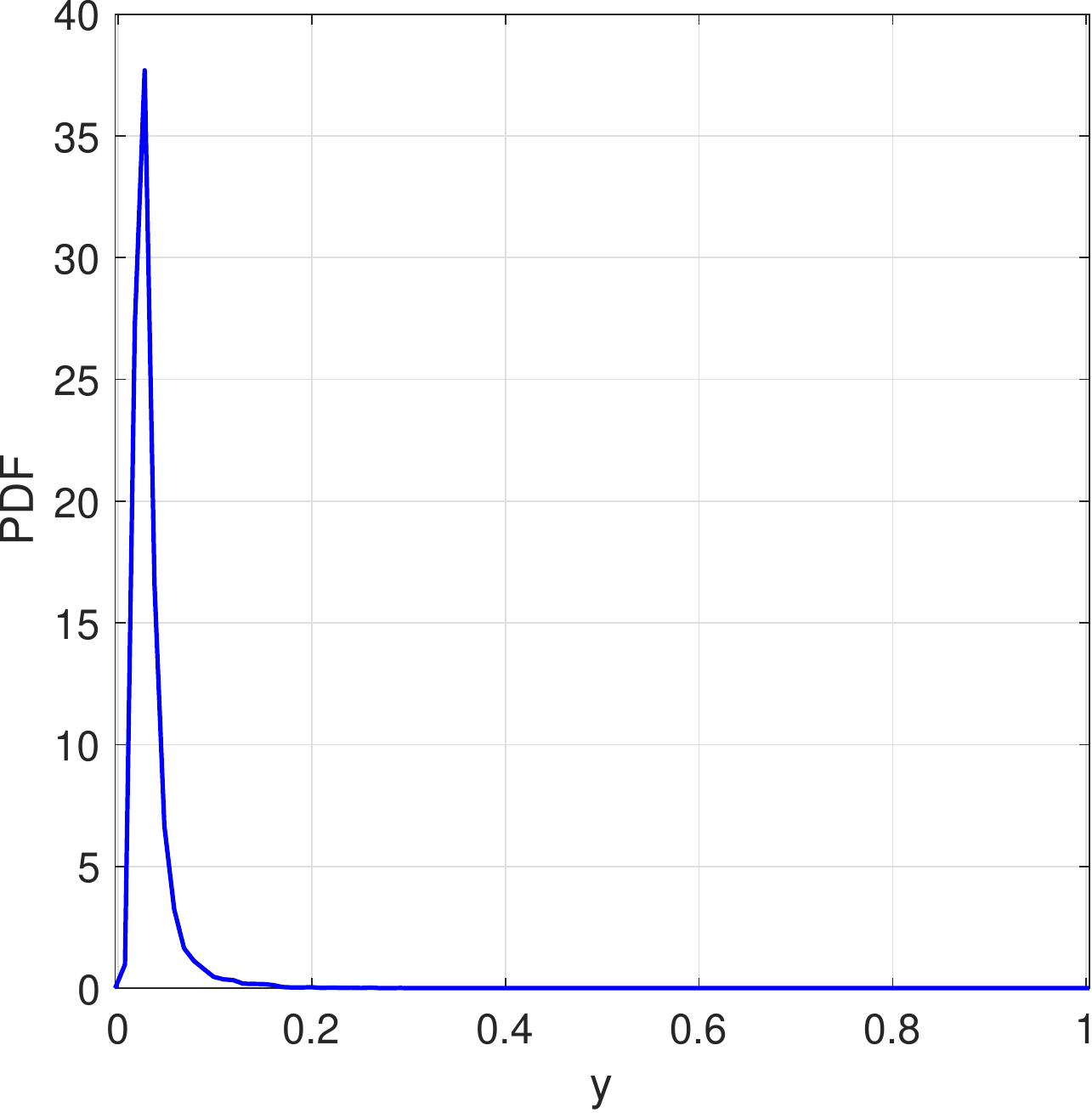}
\caption{\small{Probability distribution functions of trip distance, fare amount, and tip amount (left), probability mass function for payment type (middle) and probability distribution function for the target variable, total amount (right).}}\label{dist_x_y}
\end{figure}

For training, we set $k=30$ and $\eta=1050$. It is beneficial to discuss the timing of the optimization for this particular example. The total time for training (using tic-toc) is $t=4057.26 ~\text{seconds}$, i.e. about 67 minutes. In the first numerical example, we reported that the HMAT solver for $n=10^5$ and $k=30$ takes $8.51$ seconds. The number of iterations and number of function evaluations is $14$ and $70$ respectively for likelihood optimization. In addition, each GP-HMAT evaluation time (including linear solve and sensitivity calculations for every dimension) for four-dimensional data is roughly $5=4+1$ times the evaluation time for one linear solve. Using the above empirical estimates, the time for one GP-HMAT calculation is $\hat{t}=4057.26/70=57.96$ seconds (assuming there is no overhead time in optimization), which yields the ratio $57.96/8.51=6.81$ between the time of one GP-HMAT calculation and one linear solve, i.e. HMAT calculation. We also individually run the GP-HMAT code with the current setting and find that an individual run takes $t=60.9240$ seconds, which is slightly higher than our estimate (our estimate is effectively obtained from averaging over many runs). Computing the derivative of log determinant from $\bm A^{-1} (\partial \bm A/\partial \ell)$ by factorizing $\partial \bm A/\partial \ell$ (as a monolithic matrix not a hierarchical matrix) and applying HMAT solve to a matrix right-hand side (instead of vector) can be less accurate and potentially more costly. 

After finding the optimal $\ell^{\ast}=[0.9907,0.9705,0.9732,0.8882]$, to study the quality of the prediction with respect to rank parameter, we find the GP mean with different parameters $k$, namely $k=5, 10, 30, 50, 70$. The result of the mean of log point-wise error is shown in Figure~\ref{taxicab_pdf} (left). The PDFs of the test data and GP estimates are shown in Figure~\ref{taxicab_pdf} (right). We provide the numerical values associated with the mean of log point-wise error, i.e. $\mathbb{E}(\log_{10}(|\bm y - \bm y_{\ast}|))$ and normalized error in mean and standard deviation, i.e. $e_\mu$ and $e_{\sigma}$ Table ~\ref{GP_taxicab}. From these results, as expected, the regression with a higher rank parameter, in general, yields more accurate estimates, but the rank $k=70$ estimate is not better than $k=50$ prediction in all measured norms. From visual inspection of the right pane in Figure~\ref{taxicab_pdf}, almost all predictions provide a sufficiently accurate estimate for the probabilistic content of the target variable $y$. Indeed, the test data line is visible only on a large zoom scale, which is due to the superimposition of several similar lines in the plot.

\begin{figure}[h]
\centering
\includegraphics[width=1.8in]{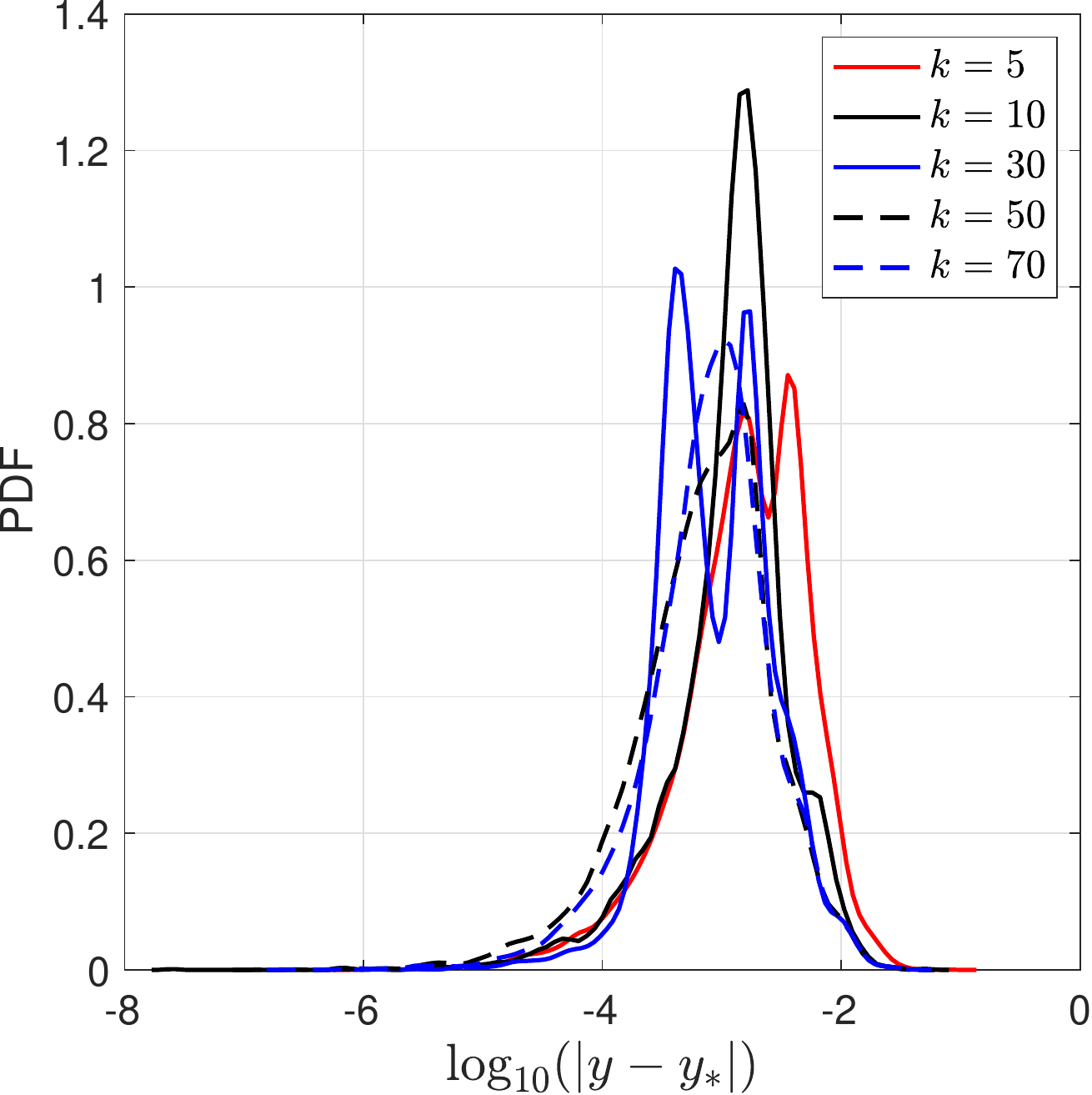}\includegraphics[width=1.8in]{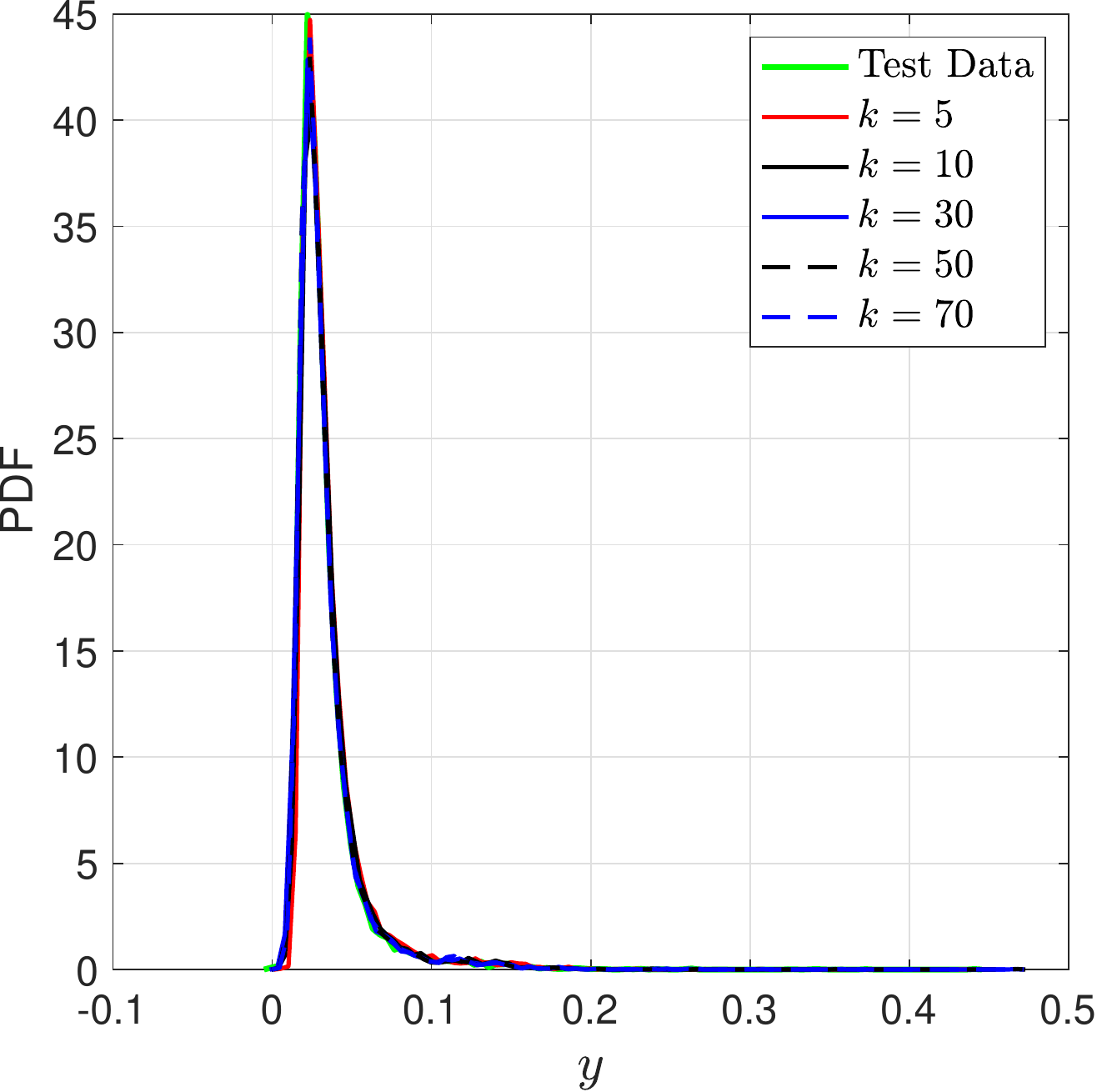}
\caption{\small{Log of point-wise error in GP prediction $\log_{10}(|y-y_{\ast}|)$ (left) and probability distribution functions of test data and GP predictions for target $y$ (right).}}\label{taxicab_pdf}
\end{figure}

\begin{table}[!h]
\caption{The mean of log point-wise error and normalized error in mean and standard deviation for various rank parameters used in GP regression.}
\normalsize
\centering
\begin{tabular}{l c c c c }
\hline\hline
  & $\mu_{\log(|y-y_{\ast}|)}$  & $e_{\mu}$  & $e_{\sigma} $     \\
\hline
Regr. ($k=5$)       &       $-2.81$	   & $6.72 \times 10^{-2}$  &  $4.00 \times 10^{-2}$	\\
Regr. ($k=10$)       &     $-2.92$           & $3.77 \times 10^{-2}$  &  $8.00 \times 10^{-3}$	\\
Regr. ($k=30$)       &     $-3.05$	   & $9.51 \times 10^{-3}$  &  $9.60 \times 10^{-4}$	\\
Regr. ($k=50$)       &     $-3.15$           & $5.81 \times 10^{-3}$  &  $4.52 \times 10^{-3}$	\\
Regr. ($k=70$)       &     $-3.11$           & $4.06 \times 10^{-3}$  &  $8.90 \times 10^{-3}$ \\
\hline
\end{tabular}
\label{GP_taxicab}
\end{table}

\section{Concluding remarks}\label{Sec6}
We develop a scalable tool for GP training and regression. Our approach revolves around hierarchical low-rank matrices that are exploited for efficient GP training. We provide a practical implementation for hierarchical matrix decomposition that entails an aggregation procedure to promote the low-rank structure of the off-diagonal blocks.  The off-diagonal blocks of the hierarchical matrix are factorized with a randomized SVD approach with certified accuracy and scalability. We develop hierarchical error and cost estimates and verify them via empirical studies. Our approach can be viewed as another practical approach for sparse GP regression that scales in the form of $\mathcal{O}(n \log (n))$. The following remarks discuss three research directions as possible extensions to the current GP-HMAT approach: 1) HSS matrices: The factorization for off-diagonal blocks in our work has been done separately for each off-diagonal block. In a subset class of hierarchical matrices, hierarchical semiseparable (HSS) matrices~\cite{Chandrasekaran05,Xia10_main, Martinsson2011}, the factors representing the off-diagonal blocks admit certain recursive relations that make the hierarchical representation inexpensive to
store and to apply. We plan to investigate the potential of these fast matrix constructions in combination with high-performance computing within the context of scalable GP approaches in our future research. 2) Application to finite element analysis (FEA) and topology optimization: The application of similar numerical methods/solvers e.g. multigrid solvers~\cite{Amir2014} has been considered in the context of linear elasticity and topology optimization. In the same vein, we plan to develop a systematic framework for optimization of parameterized hierarchical matrices in the context of topology optimization (and possibly nonlinear elasticity), which could significantly accelerate the design optimization process. 3) Adaptation of the HMAT solver for quadrature optimization:  In a similar context to evaluating a kernel matrix, we evaluate the Jacobian matrix of a nonlinear system from the polynomial recurrence rules for quadrature optimization~\cite{Keshavarzzadeh_DQ, mldq_vkz} (referred to as designed quadrature). As another research direction, we aim to adapt the computations in the HMAT solver to the designed quadrature framework for accelerated computation of quadrature rules in multiple dimensions.   

\section*{Appendix}
\renewcommand{\thesubsection}{\Alph{subsection}}
\setcounter{section}{6}
 \renewcommand{\thesection}{\arabic{section}}

\section{Remarks}~\label{App_remarks}
This section of the supplementary materials discusses various remarks from different parts of the paper:
\begin{remark_new}\label{rem:kernelspecific}
\textbf{Definition of kernels:} Given the following distance function:\\ $f(\bm x_i, \bm x_j, \ell)= (\bm x_i - \bm x_j)^T \bm M (\bm x_i - \bm x_j),  \quad \bm M =\ell^{-2} \bm{I} $
where $\bm{I}$ is the identity matrix, squared exponential and exponential kernels and their derivatives are given by 
\begin{equation}
	\begin{array}{l l }
	\text{squared exponential:} &k(\bm x_i, \bm x_j, \ell)=\exp(-\displaystyle \frac{1}{2}f), \qquad \frac{\partial k}{\partial \ell} = -\frac{1}{2} k \frac{\partial f}{\partial \ell}   \\
	\\
	\text{exponential kernel:} &\displaystyle k(\bm x_i, \bm x_j, \ell)=\exp(-\sqrt{f}), \qquad \frac{\partial k}{\partial \ell} = -\frac{1}{2\sqrt{f}} k \frac{\partial f}{\partial \ell}. 
	\end{array}
\end{equation}	
We have also included the implementation of the automatic relevance determination (ARD)~\cite{Neal96_SM} kernel within the code, which involves multiple hyperparameters, i.e. for ARD kernel the normalization matrix and its derivative are $\bm M =\text{diag}(\bm \ell)^{-2},~\partial \bm M/\partial \bm \ell =-2\text{diag}(\bm \ell)^{-3}$ where $\bm \ell$ is a vector of hyperparameters with a size equivalent to the dimension of data points. 

\end{remark_new}

\begin{remark_new}\label{rem:graph_partition}
\textbf{Section~\ref{Sec2}:} 
We are mainly interested in finding aggregates such that the resulting off-diagonal blocks exhibit a low-rank structure. This depends on the analyticity of the kernel as well as the location of GP nodes. We have attempted to tackle this problem by pursuing some form of nuclear norm optimization~\cite{brecht10} to reduce the rank of off-diagonal blocks; however, we have found that this problem is inherently difficult, and rank reduction may not be even possible for every generic kernel (e.g. the kernel with $l_1$ distance function in the first numerical example of Section~\ref{Sec5}). On the other hand, our aggregation strategy produces a sufficiently satisfactory result given its simplicity and implementation efficiency.   
\end{remark_new}

\begin{remark_new}\label{rem:orig_Hmatrix}
\textbf{Section~\ref{Sec2}:} In the original $\mathcal{H}$-matrix development, a concept of admissible and inadmissible blocks are introduced which helps specifying the diagonal and off-diagonal blocks. The admissible blocks involve full (or exact) representation of the kernel and the inadmissible blocks are associated with the regions of the matrix where low-rank approximation is used. In those original developments, analytical kernels with specified geometric domains are considered which readily determine the regions of admissible and inadmissible blocks. In this work, we do not assume any predefined geometric domain for the kernel; we endow the hierarchical structure in a user-defined fashion, and resort to a simple aggregation strategy for permutation of degrees of freedom to promote the low-rank structure of the off-diagonal blocks. Another key consideration in our construction which involves matrix inversion without factorizing a large matrix (or matrix-matrix multiplications), is the ability to ``gather'' the decomposed pieces on which a \emph{recursive} procedure can be applied. In other words, we outline a hierarchical decomposition which is consistent with the assembling of the decomposed pieces into a workable, matrix-free, and recursive inversion algorithm.
\end{remark_new}

\begin{remark_new}\label{rem:recursive_index}
\textbf{Section~\ref{Sec2}:} The execution of \texttt{perm\_generator} cf. Algorithm~\ref{alg:perm_generator} on the set of all nodes, once, results in a permutation index set for the entire matrix. Alternatively, one can use only the \texttt{permute} Algorithm~\ref{alg:permute} within the main recursive procedure of e.g. the linear solve, in an on-the-fly manner. Such consideration requires assignment of degrees of freedom into proper places in the resulting vectors within a recursive procedure (i.e. the assignment is performed several times within the recursion). We have found with experience that performing permutation computation once at the beginning of the linear solve results in faster execution overall. In other words, within the linear solve algorithm, first the global permutation indices are found, then the list of nodes and vector $\bm y$ is permuted (permuting list of nodes effectively permutes degrees of freedom in the matrix $\bm A$; however, no large matrix is directly factorized), and finally, when the solution $\bm x$ is obtained, it is reverted back to the original indices. 
\end{remark_new}

\begin{remark_new}\label{rem:oversampling}
\textbf{Section~\ref{Sec3}:} The oversampling parameter is introduced such that for a given target rank $k$, the matrix $\bm Q$ yields sufficiently small error in~\eqref{randomized_app} with high probability. The significance of $p$ is mainly realized in theoretical analysis and further numerical investigation of analytical estimates. Our implementation works with a user-defined parameter $k$. Obviously, if a certain $k$ does not yield the satisfactory approximation, the user has the choice to increase $k$. This is mainly because the target rank of the matrices is not known \emph{a priori} in our setting (or often in practical situations), and an appropriate parameter $k$ can be specified only with testing and investigating the error on a smaller scale problem for which the true solution is available. We present such results in our numerical studies in Section~\ref{Sec5}.  
\end{remark_new}

\begin{remark_new}\label{rem:SRFT}
\textbf{Section~\ref{Sec3}:} In Section 4.6 of ~\cite{nhalko11}, the complexity issue of $\bm Y$  has been discussed and particular structured random matrices, namely \emph{subsampled random Fourier transform} (SRFT) is introduced to remedy the issue. The SRFT is the simpler version of notable/significant algorithms in~\cite{Rokhlin13212} for solving overdetermined linear systems and Section 3.3 and Lemma 3.12 of~\cite{WOOLFE08} for computing an accelerated Fourier transform. The actual implementation of these approaches; however, decreases the complexity of $\mathcal{O}(mnk)$ to only $\mathcal{O}(mn\log(k))$ by exploiting the fundamental savings afforded by Fast Fourier Transform (FFT) and the fact that the orthonormal matrix $\bm Q$ can also be computed from the Fourier transformed version of $\bm A$, i.e. when each row (or column) of $\bm A$ is Fourier transformed. 
\end{remark_new}

\begin{remark_new}\label{rem:impdetail_rsvdid}
\textbf{Section~\ref{Sec3}:} There is a subtle implementation detail in Algorithm~\ref{alg:rsvd_id} which we think is not immediately apparent from Algorithm 5.2 in~\cite{nhalko11}. In step 6 of Algorithm~\ref{alg:rsvd_id}, we compute the QR factorization of the transpose of $\bm A(\mathcal{I}_{ID},:)$. This is done to ensure that the $\bm W \in \mathbb{R}^{n_2 \times k}$ matrix is an orthonormal matrix. Given that $\bm W$ is orthonormal, $\bm \Gamma_r$ (in step 9) will be orthonormal since $\tilde{\bm \Gamma}_r$ is also orthonormal. Using $\bm A(\mathcal{I}_{ID},:)$ (instead of $\bm A(\mathcal{I}_{ID},:)^T$) in step 6 and modifying the next steps accordingly (i.e. $\bm Z = \bm X \bm W$ in step 7 and $\bm \Gamma_r = \bm R^T \tilde{\bm \Gamma}_r$ in step 9 where $\bm W \in \mathbb{R}^{k \times k}$ and $\bm R \in \mathbb{R}^{k \times n_2}$) will result in a non-orthonormal $\bm \Gamma_r$ since $\bm R$ will not be orthonormal. The orthonormality of $\bm \Gamma_l$ and $\bm \Gamma_r$ is important mainly because 1) singular vectors in SVD analysis are expected to be orthonormal (they can be used to invert the input matrix easily if needed), and more importantly, 2) the derivation of singular vector sensitivities, cf. Section~\ref{Sec4} and~\ref{App_svd_der}, is based on this orthonormality property. Nonorthonormal singular vectors will result in erroneous sensitivity analysis.   
\end{remark_new}

\begin{remark_new}\label{rem:svdofderivative}
\textbf{Section~\ref{Sec3}:} To compute the sensitivity of the SVD factorization (with respect to the kernel hyperaprameter), we need to compute the SVD of the derivative of the matrix which is done by slight modifications to Algorithm~\ref{alg:rsvd_id}. The modified algorithm is presented in Algorithm~\ref{alg:rsvd_id_d}. Note that in this pseudocode, the derivative is shown for a scalar hyperparameter case. Indeed, in our sensitivity derivations in Section~\ref{Sec4}, we always show the derivative for one hyperparameter for more clarity. Generalizing to larger number of hyperparameters is trivial by implementing these algorithms via for loop running on multiple hyperparameters which is done in our computer implementation.   
\end{remark_new}

\begin{remark_new}\label{rem:assignmentvar}
\textbf{Section~\ref{Sec4}:} In SMW computations, the variable $\bm x_{D_i}$ involves the solution associated with the diagonal blocks (hence using $._D$). Note that the second dimension is not always $1$ as the right-hand side is concatenated in the recursive procedure. The variable $\bm q_{l_i}$, is a quantity (denoted by $\bm q$) that is obtained from a linear solve with the singular vectors as the right-hand side. The subscript $l$ is used because these terms appear on the left side of the correction matrix $\bm C_{smw}$ in the SMW formula. Analogously, the variable $\bm q_{lr_i}$ is associated with $\bm V \bm A_D^{-1} \bm U$, which includes both $\bm U$ and $\bm V$ in both sides of $\bm A_D^{-1}$. The last term, $\bm q_{ry_i}$ is associated with $\bm V \bm A_D^{-1} \bm y$, which appears on the right side of the correction matrix and also includes $\bm y$. Finally, the variable $\bm s_{qry_i}$ is the solution (hence using $\bm s$) to the linear solve with the right-hand side $\bm q_{ry_i}$.
\end{remark_new}

\begin{remark_new}\label{rem:cell_structure}
\textbf{Section~\ref{Sec4}:} As seen in Algorithm~\ref{alg:smw_ing} (and also Algorithm~\ref{alg:smw_lkl_ing}), in the step that we call the solver function (or likelihood estimation function), we concatenate right-hand side $\bm y$ with the matrix of singular vectors $\bm \Gamma$. The concatenation is straightforward; however, one computational bottleneck is the changing size of the right-hand side $[\bm y~\bm \Gamma]$ as well as the size of the solution to this right-hand side xD\_ql throughout the recursion. The corresponding solutions to individual $\bm y$ and $\bm \Gamma$ are also sliced from xD\_ql, which itself can be memory intensive. To accommodate the recursive computation, there have to be memory allocations for both concatenation and further slicing from xD\_ql. This memory issue deteriorates the scalability of the code; this can be inspected with a useful tool, the MATLAB \emph{profile} function. We find that using cell structure (while less optimal in terms of basic matrix operations) can significantly improve timing as concatenation and slicing in cell structures are much faster (perhaps due to the smaller size of the matrices within a cell structure). We also emphasize that, the memory issue is not completely resolved even with employing cell structure as it appears in our empirical scalability results for large $k$ and $n$. Having considered the memory issue, we report the fastest implementation/best scalability we could obtain from the code in this paper. 

\end{remark_new}

\begin{remark_new}\label{rem:sens_comp}
\textbf{Section~\ref{Sec4}:} One alternative for computation of the derivative of energy term $\bm x^T ({\partial \bm A}/\partial \ell) \bm x$ is to reduce the matrix ${\partial\bm A}/\partial \ell$ and perform $(\partial\bm A/\partial \ell) \bm x$. Forward computation $(\partial \bm A/\partial \ell) \bm x$ can be done outside the recursive loop i.e. after finding $\bm x$ in an straightforward manner; we call this computation non-intrusive. The non-intrusive and  intrusive approaches in this paper (cf. Section~\ref{Sec4} where we compute derivatives and solutions $\bm x$ within the recursion) are equally efficient for finding the derivative of the energy term. However, the efficient computation of log determinant with a non-intrusive approach seems non-trivial as it requires computation of $\bm A^{-1}(\partial \bm A/\partial \ell)$ i.e. linear solve with a large number of right-hand side vectors. One naive approach is to perform fast SVD (e.g. cf. Algorithm~\ref{alg:rsvd_id}) on the large square matrix  $\bm \partial \bm A/\partial \ell$ and perform a much smaller number of linear solves however that can significantly impact the accuracy. On the other hand, performing the derivative of SVD (as it is required for intrusive computation) and recursive implementation of $(\partial\bm A/\partial \ell) \bm x$ yield both derivatives of the energy term and the log determinant in a monolithic and efficient way as shown in Section~\ref{Sec4}.
\end{remark_new}

\section{Maximal independent set}~\label{App_MIS} In this section, we provide the details of a generic node aggregation strategy that is often used in algebraic multigrid~\cite{Vanek96,Bell12,Treister15}. The aggregation procedure is based on the \emph{maximal independent set} (MIS) algorithm~\cite{Bell12} which we present in detail in Algorithm~\ref{alg:MIS1}. The main ingredient of the procedure is the strength matrix $\bm S$ defined by \begin{equation}
\begin{array}{l l}
 & \bm S_{ij} = \begin{cases}
1 & \text{if} ~~|\bm A_{ij}| \geq \theta \sqrt{|\bm A_{ii} \bm A_{jj} |}  \\
0 & \text{othewise}.
\end{cases}
\end{array}
\end{equation}
where $|\bm A_{ij}|$ denotes the manitude of $\bm A_{ij}$ entry and $\theta$ is a tunable aggregation parameter. In kernel matrices, usually $\bm A_{ii}=1, \forall i$; therefore the above inequality simplifies to $\bm A_{ij} \geq \theta$ (See Step 7 in Algorithm~\ref{alg:MIS1}).

\begin{algorithm}
  \caption{Distance-$1$ maximal independent set: MIS$1$($\bm A, \theta$)} 
 \flushleft \textbf{Input:}  Sparse matrix $\bm A \in \mathbb{R}^{n \times n}$, Aggregation parameter ~$\theta$ \\
 \flushleft \textbf{Output:}  Set of root nodes (or MIS nodes) $\bm R~\&$ Set of permutation index sets $\bm P$\\
 \begin{algorithmic}[1]
 \State $\mathcal{I} = \{1,\ldots,n\}$ \Comment{Initial candidate index set}
 \State ${\bm R} \gets 0 \times \mathcal{I}$, $c=1$ \Comment{Initially $\bm R$ is set of zeros with length $n$ }
 \For {$i \in \mathcal{I}$}
 	 \If {$ \bm R_i=0$}  
  		\State $\bm R_i=1,~\bm P_c=\{\}$ \Comment{The unaggregated node $i$ is assigned as a root}
		\State $\bm P_c \gets \bm P_c \cup \{i\}$ \Comment{$i$ added to $\bm P_c$}
		\For {$\{j~ |~  \bm A_{ij} \geq \theta \}$}
			\State $\bm R_j = -1, ~ \bm P_c \gets \bm P_c \cup \{j\}$ \Comment{Node $j$ is not a root $\bm R_j=-1$; $j$ added to $\bm P_c$}	
		\hspace{-0.25cm}	\color{gray}
       \For {$\{k~ |~  \bm A_{jk} \geq \theta \}$} \Comment{MIS$2$ entails an added for-loop on neighbors indices $j$ }
						\State $\bm R_k = -1, ~ \bm P_c \gets \bm P_c \cup \{k\}$
			\EndFor \color{black}		
		\EndFor
		\State $c \gets c+1$
	\EndIf
\EndFor
\State $\bm R = \{i ~ |~  \bm R_i =1 \}$
\State $\bm P = \{ \bm P_i~ |~ i \in (1,2,\ldots,c-1) \}$ 
 \end{algorithmic}
\label{alg:MIS1}
\end{algorithm}

\section{Interpolative decomposition}\label{sec:ID}

The interpolative decomposition (ID) is an effective approach for the compression of rank deficient matrices, which was initially introduced in~\cite{GuMing96} and later elucidated via a computational procedure in~\cite{Cheng05}.  ID identifies a collection of $k$ columns or rows from a matrix $\bm A \in \mathbb{R}^{m \times n}$ with rank $k$ that cover the span of the range of $\bm A$. In particular, from the ID procedure the index set  $\mathcal{I}_{ID} = \{i_1,\ldots,i_k\}$ is computed, which yields $\bm A \simeq \bm X \bm A(\mathcal{I}_{ID},:)$ where $\bm X \in \mathbb{R}^{m \times k}$ is the ID matrix with the condition that $\bm X (\mathcal{I}_{ID},:) = \bm{I}_k \in \mathbb{R}^{k \times k}$ is an identity matrix. 
To find the important rows, a QR factorization is performed on the transpose of the matrix i.e. $[\bm Q_A,\bm R_A,\bm P_{mat}] \gets qr(\bm A^T, \text{`matrix'})$ where $\bm A^T \bm P_{mat} \simeq \bm Q_A \bm R_A$
and $\bm P_{mat} \in \mathbb{R}^{m \times m}$ is a permutation matrix.  Equivalently, one can use  $[\bm Q_A,\bm R_A,\bm P_{vec}] \gets qr(\bm A^T, \text{`vector'})$ to find the permutation indices stored in the vector $\bm P_{vec}$. Subsequently, the ID indices are found from the first $k$ indices in $\bm P_{vec}$, i.e. $\mathcal{I}_{ID} = \bm P_{vec}(1:k)$. 
To find the ID matrix, first the upper-triangular matrix $\bm R \in \mathbb{R}^{k \times m}$ is decomposed into two matrices, $\bm R_A =[\bm R_{A_{1}} ~ \bm R_{A_{2}}]$ where $\bm R_{A_{1}}\in \mathbb{R}^{k \times k}$ and $\bm R_{A_{2}}\in \mathbb{R}^{k \times (m-k)}$. The above equation can be equivalently written as $\bm R_A = \bm R_{A_{1}} [\bm{I} ~ \bm T]$  where $\bm T = (\bm R_{A_1} )^{-1} \bm R_{A_2} \in \mathbb{R}^{k \times (m-k)}$ is one building block of the ID matrix. Using the relation $\bm Q \bm R_{A_1} [\bm{I}~\bm T] \bm P^T_{mat}\simeq\bm A^T$ where $\bm{I} \in \mathbb{R}^{k \times k}$, it then follows that 
\begin{equation}
\begin{array}{l}
\bm A^T \simeq \bm Q \bm R_{A_1} [\bm{I} ~ \bm T] \bm P^T_{mat} = \bm A^T(:, \mathcal{I}_{ID}) [\bm{I} ~\bm T] \bm P^T_{mat} \\
\\
\bm A^T \simeq \bm A^T(:, \mathcal{I}_{ID}) \bm X^T,~\bm A \simeq \bm X \bm A(\mathcal{I}_{ID},:)
\end{array}
\end{equation}
where $\bm X = ([\bm{I} ~\bm T] \bm P^T_{mat})^T$. The ID matrix $\bm X$ can also be obtained by using the information in the permutation vector by $\bm X = \left([\bm{I} ~\bm T](:,\tilde{\bm P}_{vec})\right)^T$ where $\tilde{\bm P}_{vec} = \text{sort}(\bm P_{vec},\text{`ascend'})$ yields the permutations indices associated with the transpose of the permutation matrix. It should be noted that the computations with permutation vector are appreciably faster than computations with the permutation matrix. 

Apparently applying ID procedure to the full matrix $\bm A \in \mathbb{R}^{m \times n}$ is significantly costly. To arrive at $\bm A \simeq \bm X \bm A(\mathcal{I}_{ID},:)$ with the $\bm X$ matrix obtained from $\bm Q$, consider $\bm Q = \bm X \bm Q(\mathcal{I}_{ID},:)$. It follows that
\begin{equation}
\bm A \simeq \bm Q \bm Q^T \bm A = \bm X \bm Q(\mathcal{I}_{ID},:) \bm Q^T \bm A.
\end{equation}
Since $\bm X(\mathcal{I}_{ID},:)= \bm{I}_k$, the above equation implies that  $\bm A(\mathcal{I}_{ID},:) \simeq \bm Q(\mathcal{I}_{ID},:) \bm Q^T \bm A$. Therefore, as a proxy to $\bm A$ we compute $\bm X \bm A(\mathcal{I}_{ID},:)$ where $\bm X$ is computed from $\bm Q$.

\section{Pseudocodes}~\label{App_algs} 
In this section, we provide pseudocodes associated with the hierarchical decomposition in Section~\ref{S2_2}, the linear solve and likelihood evaluation algorithms denoted by \texttt{back\_solve} and \texttt{lkl\_eval}, as well as their associated pseudocodes, discussed in Section~\ref{Sec4}. 

\begin{algorithm}
 \caption{Node Permutation Based on Size: $[\mathcal{I}_1,\mathcal{I}_2] \gets $\texttt{permute}($\bm N$, $\mathcal{I},~s$)} 
 \flushleft \textbf{Input:}  A Set of nodes $\bm N \in \mathbb{R}^{d \times n}$ with index set  $\mathcal{I} = \{1,\ldots,n\}$, aggregation size ~$s$ \\
 \flushleft \textbf{Output:}  Two sets of indices $\mathcal{I}_1$ and $\mathcal{I}_2$ with sizes $s$ and $n-s$\\
 \begin{algorithmic}[1]
 \State Compute the set $\bm A_1 = \{A_{11},\ldots,A_{1n}\}$ (with respect to the first node in $\bm N(:,\mathcal{I})$) 
 \State $[\tilde{\bm A}_1, \bm P] \gets \text{sort}(\bm A_1,\text{`descend'})$ 
\State $\mathcal{I}_1 \gets \{P_i\}_{i=1}^{s}$ and $\mathcal{I}_2 \gets \{P_i\}_{i=s+1}^{n}$ 
\end{algorithmic}
\label{alg:permute}
\end{algorithm}

\begin{algorithm}
 \caption{ Recursive Permutation Generator: $\tilde{\mathcal{I}} \gets $\texttt{perm\_generator}$(\bm N, \mathcal{I}$,~$\eta$)} 
 \flushleft \textbf{Input:}  A Set of nodes $\bm N \in \mathbb{R}^{d \times n}$ with index set  $\mathcal{I} = \{1,\ldots,n\}$,~$\eta$ \\
 \flushleft \textbf{Output:}  A set of permuted nodes $\tilde{\mathcal{I}}$\\
 \begin{algorithmic}[1]
 \State Compute $\nu$ based on $\# \mathcal{I}$ cf. Equation~\eqref{sl_size} 
 \State $[\mathcal{I}_1, \mathcal{I}_2] \gets$ \texttt{permute}($\bm N, \mathcal{I}$,~ $\nu$)  \Comment{See Algorithm~\ref{alg:permute}}
\State $\tilde{\mathcal{I}}_1 \gets$ \texttt{perm\_cutoff}($\bm N, \mathcal{I}_1$,~$\eta$)  \Comment{See Algorithm~\ref{alg:perm_cutoff}}
\State $\tilde{\mathcal{I}}_2 \gets$ \texttt{perm\_cutoff}($\bm N, \mathcal{I}_2$,~$\eta$)  
\State $\tilde{\mathcal{I}} \gets \tilde{\mathcal{I}}_1 \cup \tilde{\mathcal{I}}_2$
\end{algorithmic}
\label{alg:perm_generator}
\end{algorithm}

\begin{algorithm}
  \caption{ Permutation Generation Based on $\eta$: $\tilde{\mathcal{I}} \gets$ \texttt{perm\_cutoff}($\bm N$, $\mathcal{I}$,~$\eta$)} 
 \flushleft \textbf{Input:}  A Set of nodes $\bm N \in \mathbb{R}^{d \times n}$ with index set  $\mathcal{I} = \{1,\ldots,n\}$,~$\eta$ \\
 \flushleft \textbf{Output:}  A set of permuted nodes $\tilde{\mathcal{I}}$\\
 \begin{algorithmic}[1]
 \If{$\# \mathcal{I} \leq \eta$ }
 \State{$\tilde{\mathcal{I}} \gets \mathcal{I}$}
 \Else
 \State{$\tilde{\mathcal{I}} \gets$~\texttt{perm\_generator}($\bm N, \mathcal{I}$,~$\eta$ ) } \Comment{See Algorithm~\ref{alg:perm_generator}; note that this is a recursive step.}
\EndIf 
\end{algorithmic}
\label{alg:perm_cutoff}
\end{algorithm}

\begin{algorithm}
\caption{Hierarchical Matrix Inversion: $\bm x \gets $\texttt{back\_solve}($\bm N, \bm y, [\eta,k]$)}
\flushleft \textbf{Input:} List of nodes $\bm N \in \mathbb{R}^{d \times n}$, right hand side $\bm y \in \mathbb{R}^{n \times d_y}$, $\eta$ and $k$
\flushleft \textbf{Output:}  Solution $\bm x$\\
\begin{algorithmic}[1]
\If{$d_y =1$}
    \State Find the global permutation index, $\bm P \gets \texttt{perm\_generator}(\{ i\}_{i=1}^{n}, 
    \eta)$
    \State Permute indices in $\bm N$ and $\bm y$, i.e. $\bm N \gets \bm N(:,\bm P),~\bm y \gets \bm y(\bm P,:)$
\EndIf
\State Compute $\nu$ and form two sets of indices $\mathcal{I}_1=\{i\}_{i=1}^{\nu}$ and $\mathcal{I}_2=\{i\}_{i=\nu+1}^{n}$
\State $[\bm \Gamma_l, \bm \Gamma_m, \bm \Gamma_r] \gets$ \texttt{rsvd\_id}($\bm N,[\mathcal{I}_1,\mathcal{I}_2],k$) \Comment{See Algorithm~\ref{alg:rsvd_id}}
\State $\bm x \gets \texttt{smw}(\bm N,[\mathcal{I}_1,\mathcal{I}_2],[\bm \Gamma_l, \bm \Gamma_m, \bm \Gamma_r], \bm y, [\eta, k])$ \Comment{See Algorithm~\ref{alg:smw}}
\If{$d_y =1 $}
    \State $\bm x_{new}(\bm P,:) \gets \bm x$
    \State $\bm x \gets \bm x_{new}$
\EndIf
\end{algorithmic}\label{HMI} 
\end{algorithm}

\begin{algorithm}
\caption{Main SMW computation for linear solve: \texttt{smw}($\bm N,[\mathcal{I}_1,\mathcal{I}_2],[\bm \Gamma_l, \bm \Gamma_m, \bm \Gamma_r], \bm y, [\eta, k]$)}
\flushleft \textbf{Input:} List of nodes $\bm N \in \mathbb{R}^{d \times n}$, right-hand side $\bm y \in \mathbb{R}^{n \times d_y}$, two sets of indices $\mathcal{I}_1$ and $\mathcal{I}_2$, the SVD factorization of the off-diagonal block associated with $\mathcal{I}_1$ and $\mathcal{I}_2$, $\eta$ and $k$\\
 \flushleft \textbf{Output:}  Solution $\bm x$\\
\begin{algorithmic}[1]
\State Form two lists of nodes $\bm N_1 \in \mathbb{R}^{d \times n_1}, \bm N_2 \in \mathbb{R}^{d \times n_2}$ and two right-hand side vectors (matrices) $\bm y_1 \in \mathbb{R}^{n_1 \times d_y},~\bm y_2 \in \mathbb{R}^{n_2 \times d_y}$
\State $[\bm x_{D_1}, \bm q_{l1}, \bm q_{lr1}, \bm q_{ry1}] \gets $ \texttt{smw\_ing}($\bm N_1, \mathcal{I}_1, \bm \Gamma_l, \bm y_1, [\eta, k]$)  \Comment{See Algorithm~\ref{alg:smw_ing}}
\State $[\bm x_{D_2}, \bm q_{l2}, \bm q_{lr2}, \bm q_{ry2}] \gets $ \texttt{smw\_ing}($\bm N_2, \mathcal{I}_2, \bm \Gamma_r, \bm y_2, [\eta, k]$)
\State Form the SMW correction matrix, $\bm C_{smw}$ $\gets \left [ \begin{array}{l l}
 \bm \Gamma^{-1}_m & \bm q_{lr2}\\
 \bm q_{lr1} & \bm \Gamma^{-T}_m 
 \end{array} \right ] $
\State Compute $\bm x$ using Equation~\ref{smw_eq} 
\end{algorithmic}\label{alg:smw}
\end{algorithm}

\begin{algorithm}
\caption{Ingredients of SMW computations: \texttt{smw\_ing}($\bm N, \mathcal{I}, \bm \Gamma, \bm y, [\eta, k]$)}\label{alg:smw_ing}
 \flushleft \textbf{Input:}  List of nodes $\bm N$, an index set associated with the list of nodes $\mathcal{I}$, singular vectors $\bm \Gamma$, right-hand side $\bm y$, $\eta$ and $k$ \\
 \flushleft \textbf{Output:}  $[\bm x_{D_i}, \bm q_{l_i}, \bm q_{lr_i}, \bm q_{ry_i}]$ \\
 \begin{algorithmic}[1]
 \If{$\# \mathcal{I}$ $\leq$ $\eta$ }
 \State Compute full $\bm A_{ii}$ using the nodes in $\bm N$ \Comment{For GP computations add the fixed regularization parameter $\sigma_n^2$ to the diagonal entries. See Section~\ref{Sec4_5}}
 \State Form the right-hand side by concatenating $\bm y$ and $\bm \Gamma$ i.e. $[\bm y~\bm \Gamma]$ 
 \State xD\_ql $\gets \bm A_{ii} \backslash [\bm y~\bm \Gamma]$
 \State $\bm x_D \gets$ xD\_ql($:,\{1,\ldots,d_y\}$),~ $\bm q_l \gets$ xD\_ql($:,\{d_y+1,\ldots,d_y+d_{\Gamma}\}$)
 \State $\bm q_{lr} \gets \bm \Gamma^T  \bm q_l$, $\bm q_{ry} \gets \bm \Gamma^T  \bm x_{D}$ 
 \Else
 \State Form the right-hand side by concatenating $\bm y$ and $\bm \Gamma$ i.e. $[\bm y~\bm \Gamma]$  
 \State xD\_ql$\gets$  \texttt{back\_solve}($\bm N, [\bm y~\bm \Gamma], [\eta, k]$ )  \Comment{This is a recursive step. See Algorithm~\ref{HMI}}
 \State $\bm x_D \gets$ xD\_ql($:,\{1,\ldots,d_y\}$),~ $\bm q_l \gets$ xD\_ql($:,\{d_y+1,\ldots,d_y+d_{\Gamma}\}$)
 \State $\bm q_{lr} \gets \bm \Gamma^T  \bm q_l$,~ $\bm q_{ry} \gets \bm \Gamma^T  \bm x_{D}$ 
\EndIf 
\end{algorithmic}
\end{algorithm}

\begin{algorithm}
\caption{Hierarchical likelihood evaluation: \texttt{lkl\_eval}($\bm N, \bm y, [\eta,k]$)}
\flushleft \textbf{Input:}  List of nodes $\bm N \in \mathbb{R}^{d \times n}$, right-hand side $\bm y \in \mathbb{R}^{n \times d_y}$, $\eta$ and $k$
 \flushleft \textbf{Output:}  $[\Pi,\partial \Pi/\partial \ell,\Lambda,\partial \Lambda/\partial \ell, \bm x, \bm A_d\bm x]$ \\
\begin{algorithmic}[1]
\If{$d_y =1 $}
    \State Find the global permutation index, $\bm P \gets \texttt{perm\_generator}(\{ i\}_{i=1}^{n}, \eta)$
    \State Permute indices in $\bm N$ and $\bm y$, i.e. $\bm N \gets \bm N(:,\bm P),~\bm y \gets \bm y(\bm P,:)$
\EndIf
\State Compute $\nu$ and form two sets of indices $\mathcal{I}_1=\{i\}_{i=1}^{\nu}$ and $\mathcal{I}_2=\{i\}_{i=\nu+1}^{n}$
\State $[\bm \Gamma_l, \bm \Gamma_m, \bm \Gamma_r],[\bm \Gamma_{l_d}, \bm \Gamma_{m_d}, \bm \Gamma_{r_d}],[\partial \bm \Gamma_l/\partial \ell,\partial \bm \Gamma_m/\partial \ell,\partial \bm \Gamma_r/\partial \ell] \gets$ \texttt{svd\_d}($\bm N,[\mathcal{I}_1,\mathcal{I}_2],k$) \Comment{See Algorithm~\ref{alg:svd_der}}
\State $[\Pi,\partial \Pi/\partial \ell,\Lambda,\partial \Lambda/\partial \ell,\bm x,\bm A_d\bm x] \gets \texttt{smw\_lkl}(\bm N,[\mathcal{I}_1,\mathcal{I}_2],[\bm \Gamma_l, \bm \Gamma_m, \bm \Gamma_r],[\bm \Gamma_{l_d}, \bm \Gamma_{m_d}, \bm \Gamma_{r_d}],[\partial \bm \Gamma_l/\partial \ell, \partial \bm \Gamma_m/\partial \ell, \partial \bm \Gamma_r/\partial \ell], \bm y, [\eta, k])$ \Comment{See Algorithm~\ref{alg:smw_lkl}}
\If{$d_y =1 $}
    \State $\bm x_{new}(\bm P,:) \gets \bm x,~\bm A_d\bm x_{new}(\bm P,:) \gets \bm A_d\bm x$
    \State $\bm x \gets \bm x_{new},~\bm A_d\bm x \gets \bm A_d\bm x_{new}$
\EndIf
\end{algorithmic} \label{alg:lkl_eval}
\end{algorithm}

\begin{algorithm}
\caption{Main SMW computations for likelihood and its derivative: \texttt{smw\_lkl}($\bm N,[\mathcal{I}_1,\mathcal{I}_2],[\bm \Gamma_l, \bm \Gamma_m, \bm \Gamma_r], [\bm \Gamma_{l_d}, \bm \Gamma_{m_d}, \bm \Gamma_{r_d}],[\partial \bm \Gamma_l/\partial \ell, \partial \bm \Gamma_m/\partial \ell, \partial \bm \Gamma_r/\partial \ell], \bm y, [\eta, k]$)}\label{alg:smw_lkl}
\flushleft \textbf{Input:}  List of nodes $\bm N \in \mathbb{R}^{d \times n}$, right-hand side $\bm y \in \mathbb{R}^{n \times d_y}$, two sets of indices $\mathcal{I}_1$ and $\mathcal{I}_2$, SVD, SVD of derivative and derivative of SVD of the off-diagonal block associated with $\mathcal{I}_1$ and $\mathcal{I}_2$, $\eta$ and $k$\\
\flushleft \textbf{Output:}  $[\Pi,\partial \Pi/\partial \ell,\Lambda,\partial \Lambda/\partial \ell,\bm x,\bm A_d\bm x]$\\
\begin{algorithmic}[1] 
\State Form two lists of nodes $\bm N_1 \in \mathbb{R}^{d \times n_1}, \bm N_2 \in \mathbb{R}^{d \times n_2}$ and two right-hand side vectors (matrices) $\bm y_1 \in \mathbb{R}^{n_1 \times d_y},~\bm y_2 \in \mathbb{R}^{n_2 \times d_y}$
\State $[\bm x_{D_1},\bm q_{l1}, \bm q_{lr1}, \bm q_{ry1},\Pi_1,\partial \Pi_1 /\partial \ell,\Lambda_1,\partial \Lambda_1 /\partial \ell,\bm A_d\bm x_{D_1},\bm A_d\bm q_{l_1}] \gets $ \texttt{smw\_lkl\_ing}($\bm N_1, \mathcal{I}_1, \bm \Gamma_l, \bm y_1, [\eta, k]$)
\State $[\bm x_{D_2}, \bm q_{l2}, \bm q_{lr2}, \bm q_{ry2},\Pi_2,\partial \Pi_2 /\partial \ell,\Lambda_2,\partial \Lambda_2 /\partial \ell,\bm A_d\bm x_{D_2},\bm A_d\bm q_{l_2}] \gets $ \texttt{smw\_lkl\_ing}($\bm N_2, \mathcal{I}_2, \bm \Gamma_r, \bm y_2, [\eta, k]$)
\State Form the SMW correction matrix, $\bm C_{smw}$ 
\State Compute $\bm x$ using Equation~\ref{smw_eq} 
\State Compute energy $\Pi$ and sum of the log of eigenvalues $\Lambda$ using Equation~\ref{smw_eq_energy} 
\State Compute $\partial \Pi/\partial \ell$ and $\partial \Lambda/\partial \ell$ using Equations~\eqref{pi_lambda_der},~\eqref{T_der} and~\eqref{correction_der}
\State Compute the multiplication of hierarchical matrix derivative to the solution of the linear system i.e. $\bm A_d\bm x$ using Equation~\eqref{smw_eq_der} \Comment{This is a critical part of hierarchical derivative computation. Within recursion, $\bm x$ mostly is a matrix, not a vector.}
\end{algorithmic}
\end{algorithm}

\begin{algorithm}
\caption{Ingredients of SMW computations for likelihood and its derivative: smw\_lkl\_ing($\bm N, \mathcal{I}, \bm \Gamma, \bm y, [\eta, k]$)} \flushleft \textbf{Input:}  List of nodes $\bm N$, an index set associated with the list of nodes $\mathcal{I}$, singular vectors $\bm \Gamma$, right-hand side $\bm y$, $\eta$ and $k$  \\
 \flushleft \textbf{Output:}  [$\bm x_{D_i}, \bm q_{l_i}, \bm q_{lr_i}, \bm q_{ry_i},\Pi,\partial \Pi/\partial \ell, \Lambda_i, \partial \Lambda_i /\partial \ell,\bm A_d\bm x_{D_i},\bm A_d\bm q_{l_i}]$ \\
 \begin{algorithmic}[1]
 \If{$\# \mathcal{I} \leq \eta$ }
 \State Compute full $\bm A_{ii}$ and $\partial \bm A_{ii}/\partial \ell$ using the nodes in $\bm N$
 \State Form the right-hand side by concatenating $\bm y$ and $\bm \Gamma$ i.e. $[\bm y ~\bm \Gamma]$
 \State xD\_ql $\gets \bm A_{ii} \backslash [\bm y ~\bm \Gamma]$
 \State $\bm x_{D_i} \gets$ xD\_ql($:,\{1,\ldots,d_y\}$),~ $\bm q_{l_i} \gets$ xD\_ql($:,\{d_y+1,\ldots,d_y+d_{\Gamma}\}$) 
 \State $\Pi_i \gets \bm y^T_{0}  \bm x_{D_{0i}},~\Lambda_i \gets f_{\Lambda}(\bm A_{ii})$  \Comment{$\bm y_0$ is the first column of $\bm y$.}
 \State $\bm A_d\bm x_{D_i} \gets (\partial \bm A_{ii}/\partial \ell)\bm x_{D_i},~\bm A_d\bm q_{l_i} \gets (\partial \bm A_{ii}/\partial \ell)\bm q_{l_i}$
 \State $\partial \Pi_i/\partial \ell \gets \bm x^T_{D_{0i}}  \bm A_d\bm x_{D_{0i}} $,~$\partial \Lambda_i/\partial \ell \gets Tr\left(\bm A_{ii} \backslash (\partial \bm A_{ii}/\partial \ell) \right) $
 \State $\bm q_{lr_i} \gets \bm \Gamma^T  \bm q_{l_i}$, $\bm q_{ry_i} \gets \bm \Gamma^T  \bm x_{D_i}$
 \Else
 \State Form the right-hand side by concatenating $\bm y$ and $\bm \Gamma$ i.e. $[\bm y ~\bm \Gamma]$  
 \State $[\Pi_i,\partial \Pi_i/\partial \ell,\Lambda_i,\partial \Lambda_i/\partial \ell,\text{xD\_ql,Ad\_xD\_ql}] \gets$  \texttt{lkl\_eval}($\bm N, [\bm y~\bm \Gamma], [\eta, k]$ )  \Comment{See Algorithm~\ref{alg:lkl_eval}}
 \State $\bm x_{D_i} \gets$ xD\_ql($:,\{1,\ldots,d_y\}$),~ $\bm q_{l_i} \gets$ xD\_ql($:,\{d_y+1,\ldots,d_y+d_{\Gamma}\}$)
  \State $\bm A_d\bm x_{D_i} \gets$ Ad\_xD\_ql($:,\{1,\ldots,d_y\}$),~ $\bm A_d\bm q_{l_i} \gets$ Ad\_xD\_ql($:,\{d_y+1,\ldots,d_y+d_{\Gamma}\}$)
 \State $\bm q_{lr_i} \gets \bm \Gamma^T  \bm q_{l_i}$,~ $\bm q_{ry_i} \gets \bm \Gamma^T  \bm x_{D_i}$ 
\EndIf 
\end{algorithmic} \label{alg:smw_lkl_ing}
\end{algorithm}

\begin{algorithm}
\caption{Randomized SVD with ID including derivative: \texttt{rsvd\_id\_d}($\bm N, [\mathcal{I}_1,\mathcal{I}_2],k$)}
\flushleft \textbf{Input:}  List of nodes $\bm N \in \mathbb{R}^{d \times n}$, two sets of indices $\mathcal{I}_1$ and $\mathcal{I}_2$ with $\# \mathcal{I}_1=n_1$ and $\# \mathcal{I}_2=n_2$ and the rank parameter $k$  \\
\flushleft \textbf{Output:}  $\bm A \simeq \bm \Gamma_l \bm \Gamma_m \bm \Gamma_r^T$,~$\partial \bm A/\partial \ell \simeq \bm \Gamma_{l_d} \bm \Gamma_{m_d} \bm \Gamma_{r_d}^T$\\
\begin{algorithmic}[1]
\State Compute the matrix $\tilde{\bm A} \in \mathbb{R}^{n_1 \times n_{innprod}}$ and $\partial \tilde{\bm A}/\partial \ell \in \mathbb{R}^{n_1 \times n_{innprod}}$ from the kernel function using $n_{innprod}$ cf. Equation~\eqref{n_inn}
\State Follow the steps 2-9 in Algorithm~\ref{alg:rsvd_id} for both $\bm A$ and $\partial \bm A/\partial \ell$ individually with a fixed common $\bm \Omega$
\end{algorithmic}\label{alg:rsvd_id_d}
\end{algorithm}


\section{Theoretical analysis}~\label{App_theory}
\subsection{Error estimate}~\label{EE_SM} In what follows, we first list some relevant basic linear algebra results. Unless otherwise stated, all matrix norms are written with respect to the Frobenius norm.  

\begin{itemize}
\item $\| \bm A \bm B \| \leq \| \bm A \| \| \bm B \|$ and $\| \bm A + \bm B \| \leq \| \bm A \| + \| \bm B\|$. These rules can be easily generalized to $n > 2$ matrices.
\item Error in product: $\| \bm A \bm B - \hat{\bm A}\hat{ \bm B}  \| \leq \| \bm A - \hat{\bm A} \| \| \bm B \| + \| \bm B - \hat{\bm B} \| \| \bm A \|.$
\item $\| \bm A \|_F \leq \sqrt{r} \| \bm A \|_2$ where $r$ is the rank of the matrix $\bm A$. We assume that $\bm A^{-1}$ exists, e.g. via Moore–Penrose inverse and has equivalent rank to $\bm A$, i.e. $r$. Then $\| \bm A^{-1} \|_F \leq \sqrt{r} \| \bm A^{-1} \|_2 = \sqrt{r}/\sigma_{min}(\bm A)$ where $\bm \sigma_{min}(\bm A)$ is the smallest singular value of $\bm A$.
\item $\|\bm A \|_2 \leq \| \bm A \|_F$. Using $\bm A - \hat{\bm A}$ instead of $\bm A$ and given $\|\bm A - \hat{\bm A}\|_F \leq \epsilon$ yield $\|\bm A - \hat{\bm A}\|_2 \leq \epsilon$. The $\ell_2$ norm error of inverse is  $\|\bm A^{-1} - \hat{\bm A}^{-1}\|_2 \leq \epsilon / \sigma_{min}(\bm A)$. Using the result in the previous item, gives the Frobenius norm of error $\|\bm A^{-1} - \hat{\bm A}^{-1}\|_F \leq \sqrt{r} \epsilon / \sigma_{min}(\bm A)$.
\end{itemize}

\subsubsection{Lemma~\ref{lemma_UCV}}\label{proof_lemmaUCV}\begin{proof}
We show this for a product of two matrices, e.g. $\bm A \bm B$.  We know $\| \bm A \bm B - \hat{\bm A} \hat{\bm B} \| = \| \bm A (\bm B - \hat{\bm B}) + (\bm A - \hat{\bm A}) \hat{\bm B}\| \leq \epsilon$. In the worst case, the norm $\| \bm B - \hat{\bm B} \|$ is bounded by $ \epsilon/\beta$ 
when $\| \bm A- \hat{\bm A} \| = 0$. Similarly, $\| \bm A - \hat{\bm A} \| \leq \epsilon/\beta$. To find $\epsilon/\beta^2$, we replace $\bm B$ with $\bm B \bm C$ in the joint product $\bm A \bm B$ and follow the same reasoning for $\bm A \bm B \bm C$. It is easy to see that the bound is $ \epsilon/\beta^2$ in this case. 
\end{proof}

\subsubsection{Lemma~\ref{lemma_frob_norm}}\label{proof_lemmafrobnorm}\begin{proof}
We start with the SMW formula: 
 \begin{equation*}
 \begin{array}{l l l }
\|  (\bm A_D+\bm U \bm C \bm V)^{-1} \| & \leq  &\| \bm A_D^{-1} \| + \| \bm A_D^{-1} \bm U (\bm C^{-1}+\bm V \bm A_D^{-1} \bm U)^{-1} \bm V \bm A_D^{-1} \| \\
\\
 & \leq  & \| \bm A_D^{-1}\|+  \beta^2 \| \bm A_D^{-1}\|^2  \| \bm C^{-1}_{smw} \|.  
\end{array}
\end{equation*}
Assuming $\| \bm A_D^{-1} \| > 1$ which is almost always true in our numerical computations, we neglect the first term on the right-hand side which is smaller than the second term, and write the Frobenius norm of the inverse at level $L-1$ as
\begin{equation*}
 \begin{array}{l l l }
\alpha_{L-1} = \| \bm A_D ^{(L-1)^{-1}} \| & \leq & \beta_L^2  \| \bm A_D^{(L)^{-1}} \|^2 (\sqrt{2k}/\sigma_{C_{min}})\\
\\
 & \leq  & \beta_L^2 \alpha_L^2 \kappa.   
\end{array}
\end{equation*}
Note that every term in the right-hand side of the inequality is computable. In particular, for a case when the small block within a larger block is of size $n_{min}$, $\alpha_L$ is computed using one diagonal block with size $n_{min}$. One simple estimate for $\alpha_L$ in that case is obtained by noting that the matrices are all regularized with a constant $\sigma_{n}$ in our GP regression. Therefore, the matrix and the minimum singular value can be considered as full rank and $\sigma_{n}$, respectively. These considerations yield $\alpha_L = \sqrt{n_{min}}/\sigma_n$.
Similarly, we can write the Frobenius norm of the inverse for the subsequent levels as
\begin{equation*}
\begin{array}{l l l}
\alpha_{L-2} & \leq & \beta_{L-1}^2 \alpha_{L-1} \kappa  \\
\\
&\leq& \beta_{L-1}^2 \beta^4_{L} \alpha_L^4 \kappa^3\\
\\
& \vdots &
\end{array}
\end{equation*}
which results in  $\alpha_{L-i}  \leq  \alpha_L ^{2^i} \kappa^{2i-1}\displaystyle \prod_{j=0}^{i-1} \beta_{L-j}^{2^{i-j-1}}$.  
\end{proof}

\subsubsection{Proposition~\ref{prob_frob_err}}\label{proof_propfroberr}\begin{proof}
The estimate is again achieved by analyzing the SMW formula. We know that $\bm A^{-1}_{D} \bm U$ and $\bm V \bm A^{-1}_{D}$ has identical submatrices; therefore, $f_{\alpha}(\bm A^{-1}_{D} \bm U)=f_{\alpha}(\bm V \bm A^{-1}_{D} )$ and $f_{\epsilon}(\bm A^{-1}_{D} \bm U)=f_{\epsilon}(\bm V \bm A^{-1}_{D} )$. We are looking for the estimate\\ $f_{\epsilon}( \bm A^{-1} )=f_{\epsilon} \left( (\bm A_D+\bm U \bm C \bm V)^{-1} \right)$, which is written by applying the triangular inequality and basic properties of the combination of matrix norms (presented earlier in this section) to the SMW formula, as
 \begin{equation*}
 \begin{array}{l l l }
f_{\epsilon} \left( (\bm A_D+\bm U \bm C \bm V)^{-1} \right) & \leq  & f_{\epsilon}(\bm A_D^{-1}) + 2 f_{\epsilon}(\bm A_D^{-1} \bm U) f_{\alpha} (\bm C^{-1}_{smw}) + f_{\epsilon}(\bm C^{-1}_{smw}) f^2_{\alpha} (\bm A_D^{-1} \bm U) 
\end{array}
\end{equation*}
To analyze the right-hand side of the above inequality, we expand only the second and third terms. At the end of the proof, we will have terms associated with $f_{\epsilon}(\bm A_D^{-1})$ which have much larger constants compared to $1$ in the above inequality, i.e. the constant $1$ is absorbed in the larger coefficient, which will be derived shortly. Using the variable $\kappa=f_{\alpha} (\bm C^{-1}_{smw})$ defined in Lemma~\ref{lemma_frob_norm}, and considering the relation for the error of inverse $\kappa f_{\epsilon}(\bm C_{smw}) = f_{\epsilon}(\bm C^{-1}_{smw})$, yields
\begin{equation}\label{proof1}
 \begin{array}{l l l }
f_{\epsilon}(\bm A^{-1}) & \leq  &  2 \kappa f_{\epsilon}(\bm A_D^{-1} \bm U)  + \kappa f_{\epsilon}(\bm C_{smw}) f^2_{\alpha} (\bm A_D^{-1} \bm U). 
\end{array}
\end{equation}
Having $\bm C_{smw} = \bm C^{-1} + \bm V \bm A^{-1}_{D} \bm U$, assuming the error in $\bm C^{-1}$ is negligible compared to error in $\bm V \bm A^{-1}_{D} \bm U$, and considering $f_{\alpha}(\bm U)=f_{\alpha}(\bm V),~f_{\epsilon}(\bm U)=f_{\epsilon}(\bm V)$, yields
\begin{equation}\label{proof2}
 \begin{array}{l l l }
f_{\epsilon}(\bm C_{smw}) & \leq  &  f_{\epsilon}(\bm U) f_{\alpha}(\bm A_D^{-1} \bm U) + f_{\alpha}(\bm U) f_{\epsilon}(\bm A_D^{-1} \bm U).
\end{array}
\end{equation}
Plugging in the estimate in~\eqref{proof2} into~\eqref{proof1}, we find:
\begin{equation}\label{proof3}
 \begin{array}{l l l }
f_{\epsilon}(\bm A^{-1}) & \leq  &  2 \kappa f_{\epsilon}(\bm A_D^{-1} \bm U)  + \kappa f^3_{\alpha} (\bm A_D^{-1} \bm U) f_{\epsilon}(\bm U) +  \kappa f^2_{\alpha} (\bm A_D^{-1} \bm U) f_{\alpha} (\bm U) f_{\epsilon}(\bm A_D^{-1} \bm U).
\end{array}
\end{equation}
We note that the first term in the right-hand side can be absorbed in the last term as typically $2 \ll f^2_{\alpha} (\bm A_D^{-1} \bm U) f_{\alpha} (\bm U)$. Focusing on the third term, we bound the error term $f_{\epsilon}(\bm A_D^{-1} \bm U)$ via $f_{\epsilon}(\bm A_D^{-1} \bm U) \leq f_{\epsilon}(\bm A_D^{-1} )  f_{\alpha} (\bm U) +  f_{\alpha} (\bm A_D^{-1} ) f_{\epsilon}(\bm U)$. Using this estimate in~\eqref{proof3} and simple manipulations yield
\begin{equation}\label{proof4}
 \begin{array}{l l l }
f_{\epsilon}(\bm A^{-1}) & \leq  &   \kappa f^2_{\alpha} (\bm A_D^{-1}) f^4_{\alpha}(\bm U) f_{\epsilon}(\bm A_D^{-1})  + 2\kappa f^3_{\alpha} (\bm A_D^{-1}) f^3_{\alpha}(\bm U) f_{\epsilon}(\bm U) 
\end{array}
\end{equation}
We are now ready to write the hierarchical error using the notation for the Frobenius norm of $\bm A^{-1}_D$ and $\bm U$ and their errors at level $i$, i.e. $\alpha_i,~\beta_i, \epsilon_{D,i}$ and $\epsilon_{OD,i}$. The estimate for hierarchical levels is written via the following recursive formula: 
\begin{equation}\label{proof5}
 \begin{array}{l l l }
\epsilon_{D,i-1} & \leq  &   a_i  \epsilon_{D,i} + b_i  
\end{array}
\end{equation}
where $a_i = \kappa \alpha_i^2 \beta_i^4$ and $b_i = 2\kappa \alpha_i^3 \beta_i^3 \epsilon_{OD,i}$.
Note that $\epsilon_{D,L} = 0$ at the deepest level $L$ for a case when the larger block consists of blocks with size $n_{min}$ as there is no approximation for the diagonal block at that size. For larger block size matrices we use $\varepsilon$ and write the chain of errors as
\begin{equation*}
\begin{array}{l l l}
\epsilon_{D,L-1} & \leq &  a_L \varepsilon + b_L \\
\\
\epsilon_{D,L-2} &\leq& a_{L-1} a_L \varepsilon + a_{L-1} b_L + b_{L-1}\\
\\
& \vdots &
\end{array}
\end{equation*}
Manipulating the indices, the following direct formula for the error is obtained:
\begin{equation}\label{direct_error1}
\epsilon_{D,L-i}  \leq \left(\prod_{j=L-i+1}^{k} a_j \right) \varepsilon + b_{L-i+1} +  \displaystyle \sum_{k=L-i+2}^{L} \left (\prod_{j=L-i+1}^{k-1} a_j \right ) b_{k}. 
\end{equation}

\end{proof}

To utilize the error estimate in practice, the formula in~\eqref{direct_error2} is applied once and the formula in~\eqref{direct_in_text} is applied several times, from bottom (lower right corner) to top (upper left corner) of the matrix, until the error for the original matrix is found. For instance, considering $n=10^6$, the computation of error for the original matrix takes one time application of~\eqref{direct_error2} (for computing the error from $n_{min}=100$ blocks to a $n=10^3$ block) and three times application of~\eqref{direct_in_text} (with updated $\varepsilon$ each time, for computing the error for $n=10^3 \rightarrow 10^4$, $n=10^4 \rightarrow 10^5$ and $n=10^5 \rightarrow 10^6$ blocks). 

\subsubsection{Lemma~\ref{lemma:storage_complexity}}\label{prooflemma:storage_complexity}
\begin{proof}
{\fontsize{8.2}{8.2}\selectfont 
\begin{equation}
\begin{array}{ l l l l l}
C_{\mathcal{H},St}(\bm T, k) &\displaystyle = & \displaystyle \sum_{r\times s \in \mathcal{L}^{-}(\bm T)} C_{F,st} (\# r,\# s)  & + &  \displaystyle \sum_{r\times s \in \mathcal{L}^{+}(\bm T)} C_{R,st }(\# r,\# s,k) \\
  \\
    & {\displaystyle \leq} & \displaystyle \sum_{r\times s \in \mathcal{L}^{-}(\bm T)} n^2_{min} & + &\displaystyle \sum_{r\times s \in \mathcal{L}^{+}(\bm T)} k (\# r + \# s) \\
    \\
        & \leq & \displaystyle \sum_{r \times s \in \mathcal{L}(\bm T)} n_{min} (\#r+\#s) & \stackrel{\text{Def}~\eqref{Csp}}\leq  &\displaystyle \sum_{i \in L} \sum_{r \in \bm T^{(i)}_{\mathcal{I}}} 2 C_{sp} n_{min} \# r  \\
       \\
        & \leq &  \displaystyle \sum_{i \in L} C_{sp} n_{min} \# \mathcal{I} &\leq& 2 \# L C_{sp} n_{min} \# \mathcal{I} \\
        \\
        & \leq & C_1 n \log(n)
\end{array}
\end{equation} }
\end{proof}

\subsubsection{Lemma~\ref{lemma:matvec}}\label{prooflemma:matvec}
\begin{proof}
In full blocks, the storage requirements are $\# r \#s$. The matvec operation for full blocks involves $2\#r\#s - \#r$ multiplications and $\# r$ additions. In reduced blocks, the storage requirements are $k(\# r+ \#s)$ and the matvec operation involves $2k(\#r+\#s) - \#r - k$ multiplications and $\# r$ additions.
\end{proof}

\subsubsection{Lemma~\ref{lemma:truncation}}\label{prooflemma:truncation}
\begin{proof}
The proof of this lemma follows from the fact that the cost is bounded by the cost of truncation (or SVD analysis) in large leaves, which is in the form of $k^2 (n+m)$:
\begin{equation*}
C_{\mathcal{H},Trunc} \leq k \sum_{r \times s \in \mathcal{L}^{+}(\bm T)} k (\# r + \#s)  \leq k C_{\mathcal{H},St}.
\end{equation*}
For the second inequality, see the second line in~\eqref{eq:storage_complexity}.
\end{proof}

\section{Log likelihood gradients}~\label{Appendix_loglkl_d}
To compute the sensitivities for $\Pi$ and $\Lambda$ cf. Equations~\eqref{smw_eq_energy} and~\eqref{pi_lambda_der}, we need to compute the following derivative terms,
\begin{equation}\label{T_der}
\begin{array}{l l l}
\displaystyle \frac{\partial T_1}{\partial \ell}  = \displaystyle -\bm x^T_{D_{01}} \frac{\partial \bm A_{11}}{\partial \ell} \bm x_{D_{01}} - \bm x^T_{D_{02}} \frac{\partial \bm A_{22}}{\partial \ell} \bm x_{D_{02}}, \qquad \displaystyle \frac{\partial T_{22}}{\partial \ell}  =  \displaystyle -T_{22} \frac{\partial \bm C_{smw}}{\partial \ell} T_{22} \\
\\
\displaystyle \frac{\partial T_{21}}{\partial \ell}  = \left[\displaystyle \bm x^T_{D_{01}} \frac{\partial \bm \Gamma_l}{\partial \ell}  - \bm x^T_{D_{01}}   \frac{\partial \bm A_{11}}{\partial \ell} \bm q_{l_1} \quad \bm x^T_{D_{02}} \displaystyle \frac{\partial \bm \Gamma_r}{\partial \ell}  - \bm x^T_{D_{02}}   \frac{\partial \bm A_{22}}{\partial \ell} \bm q_{l_2} \right ]\\
\\
\displaystyle \frac{\partial T_{23}}{\partial \ell}  = \left[\bm x^T_{D_{02}} \displaystyle \frac{\partial \bm \Gamma_r}{\partial \ell}  - \bm x^T_{D_{02}}   \frac{\partial \bm A_{22}}{\partial \ell} \bm q_{l_2} \quad \displaystyle \bm x^T_{D_{01}} \frac{\partial \bm \Gamma_l}{\partial \ell}  - \bm x^T_{D_{01}}   \frac{\partial \bm A_{11}}{\partial \ell} \bm q_{l_1}  \right ]
\end{array}
\end{equation}
Note that there are terms that appear in the form of $T_{22}T_{23}$ and $T_{21}T_{22}$ in the above expressions, which are small size linear solves (involve $\bm C_{smw} \in \mathbb{R}^{2k \times 2k}$ inversion). The above derivations are not complete as we still need to find $\partial \bm C_{smw}/\partial \ell$, i.e. the sensitivity of the SMW correction matrix. Such matrix sensitivity computation requires the sensitivity of each term in the matrix, i.e.
\begin{equation}\label{correction_der}
\begin{array}{l}
\displaystyle \frac{\partial \bm q_{lr_1}}{\partial \ell} =\displaystyle 2 \bm q^T_{l_1}\frac{\partial \bm \Gamma_l}{\partial \ell}  - \bm q^T_{l_1} \frac{\partial \bm A_{11}}{\partial \ell} \bm q_{l_1},\qquad \displaystyle \frac{\partial \bm q_{lr_2}}{\partial \ell} =\displaystyle 2 \bm q^T_{l_2}\frac{\partial \bm \Gamma_r}{\partial \ell}  - \bm q^T_{l_2} \frac{\partial \bm A_{22}}{\partial \ell} \bm q_{l_2},\\
\\
\displaystyle \frac{\partial \bm \Gamma^{-1}_m}{\partial \ell} =- \text{diag}\left(\text{diag}({\partial \bm \Gamma_m}/ {\partial \ell} )\oslash \text{diag}(\bm \Gamma_m)^{\circ 2} \right ). 
\end{array}
\end{equation}
where $.^{\circ 2}$ denotes Hadamard power $2$. In the above derivations, the sensitivities of the SVD factorization of off-diagonal matrix i.e. ${\partial \bm \Gamma_l}/{\partial \ell}, ~{\partial \bm \Gamma_m}/{\partial \ell}$ and ${\partial \bm \Gamma_m}/{\partial \ell}$ are needed. Computing the SVD derivative is relatively straightforward; its details are provided in Appendix~\ref{App_svd_der}. From the above derivations, it is also apparent that the crucial part of the sensitivity analysis (and less intuitive part) is the hierarchical computation of $(\partial \bm A_{ii}/\partial \ell) \bm x_{D_i}$ and $(\partial \bm A_{ii}/\partial \ell) \bm q_{l_i}$. These two terms are obtained by multiplying the sensitivity of the hierarchical matrix $\partial \bm A/\partial \ell$, which includes $\partial \bm A_{11}/\partial \ell$, $\partial \bm A_{22}/\partial \ell$ and their associated off-diagonal matrix sensitivity to the back of equation~\eqref{smw_eq} (which yields $\bm x$). To clarify the sensitivity computation further, we first show the sensitivity of the hierarchical matrix $\bm A$ (in the matrix form):

\begin{equation}~\label{smw_eq_der_onlyA}
\begin{array}{l l l}
\displaystyle \frac{\partial \bm A} {\partial \ell} & = &\left [ \begin{array}{l l}
 \partial \bm A_{11}/\partial \ell & \bm \Gamma_{l_d} \bm \Gamma_{m_d}\bm \Gamma^T_{r_d}\\
 \bm \Gamma_{r_d} \bm \Gamma^T_{m_d}\bm \Gamma^T_{l_d}  & \partial \bm A_{22}/\partial \ell 
 \end{array} \right ] 
 \end{array}
\end{equation}
where $\bm \Gamma_{l_d}, \bm \Gamma_{m_d},\bm \Gamma_{r_d}$ are obtained from the SVD of the off-diagonal matrix sensitivity. It is important to note the difference between the sensitivity of the SVD of the off-diagonal block and the SVD of the off-diagonal matrix sensitivity cf. Appendix~\ref{App_svd_der}. We also define the following notations: $\bm A_d\bm x \coloneqq \displaystyle \frac{\partial \bm A}{\partial \ell} \bm x,~\bm A_d\bm x_{D_i} \coloneqq \displaystyle \frac{\partial \bm A_{ii}}{\partial \ell} \bm x_{D_i},~\bm A_d\bm q_{l_i} \coloneqq \displaystyle \frac{\partial \bm A_{ii}}{\partial \ell} \bm q_{l_i}$.

Using the definitions above and multiplying~\eqref{smw_eq_der_onlyA} to the left of~\eqref{smw_eq} yields:
\begin{equation}~\label{smw_eq_der}
\begin{array}{l l l}
\bm A_d\bm x   
  &=& \left [ \begin{array}{l l}
 \partial \bm A_{11}/\partial \ell & \bm \Gamma_{l_d} \bm \Gamma_{m_d}\bm \Gamma^T_{r_d}\\
 \bm \Gamma_{r_d} \bm \Gamma^T_{m_d}\bm \Gamma^T_{l_d}  & \partial \bm A_{22}/\partial \ell 
 \end{array} \right ] 
  \left(\left [ \begin{array}{l}
 \bm x_{D_1} \\
 \bm x_{D_2}  
 \end{array} \right ] -\left [ \begin{array}{l}
 \bm q_{l_1}\bm s_{qry_1} \\
 \bm q_{l_2} \bm s_{qry_2} 
 \end{array} \right ]  \right) \\
 \\
 &=& \left [ \begin{array}{l }
 \bm A_d\bm x_{D_1}+ \bm \Gamma_{l_d} \bm \Gamma_{m_d}\bm \Gamma^T_{r_d} \bm x_{D_2}\\
 \bm \Gamma_{r_d} \bm \Gamma^T_{m_d}\bm \Gamma^T_{l_d}  \bm x_{D_1}  +  \bm A_d\bm x_{D_2}
 \end{array} \right ] - \left [ \begin{array}{l }
\bm A_d\bm q_{l_1} \bm s_{qry_1}  + \bm \Gamma_{l_d} \bm \Gamma_{m_d}\bm \Gamma^T_{r_d} \bm q_{l_2} \bm s_{qry_2} \\
 \bm \Gamma_{r_d} \bm \Gamma^T_{m_d}\bm \Gamma^T_{l_d}  \bm q_{l_1}  \bm s_{qry_1}+ \bm A_d\bm q_{l_2} \bm s_{qry_2}
 \end{array} \right ]. 
 \end{array}
\end{equation}
The result of above matrix computation in a hierarchical fashion yields key sensitivity terms $\bm A_d\bm x_{D_1}, \bm A_d\bm x_{D_2}, \bm A_d\bm q_{l_1},$ and $\bm A_d\bm q_{l_2}$. The algorithmic computation of these terms is shown in Algorithm~\ref{alg:smw_lkl_ing}.

\section{SVD derivative}\label{App_svd_der}
We follow the derivations in~\cite{Townsend16} and report the final results using the same notation in~\cite{Townsend16}.

The derivatives of SVD factorization are 
\begin{equation}~\label{svd_der1}
\begin{array}{l l l}
d \bm U  & = & \bm U (\bm F \circ [\bm U^T d \bm A \bm V \bm S + \bm S \bm V^T d \bm A^T \bm U]) + (\bm {I}_m - \bm U \bm U ^T) d \bm A \bm V \bm S^{-1}\\
d \bm S & = & \bm{I}_k \circ [\bm U^T d \bm A \bm V] \\
d \bm V  & = & \bm V (\bm F \circ [\bm S \bm U^T d \bm A \bm V + \bm V^T d \bm A^T \bm U \bm S]) + (\bm{I}_n - \bm V \bm V^T) d \bm A^T \bm U \bm S^{-1}
\end{array}
\end{equation}
where $\bm F_{ij}= \delta_{ij}/{(s_{jj}^2-s_{ii}^2)}$ with $s_{ii}$, $\circ$ and $\bm {I}_k$  denoting singular values, Hadamard (entry-wise) multiplication, and a $k \times k$ identity matrix respectively. The following algorithm outlines the steps for computing the sensitivity of SVD factorization: 

\begin{algorithm}
\caption{Derivative of SVD: \texttt{svd\_d}($\bm N, [\mathcal{I}_1,\mathcal{I}_2], k$)}
\flushleft \textbf{Input:}  List of nodes $\bm N \in \mathbb{R}^{d \times n}$, two sets of indices $\mathcal{I}_1$ and $\mathcal{I}_2$ and the rank parameter $k$  \\
\flushleft \textbf{Output:} $[\bm \Gamma_l, \bm \Gamma_m, \bm \Gamma_r],[\bm \Gamma_{l_d}, \bm \Gamma_{m_d}, \bm \Gamma_{r_d}],[\partial \bm \Gamma_l/\partial \ell,\partial \bm \Gamma_m/\partial \ell,\partial \bm \Gamma_r/\partial \ell]$\
\begin{algorithmic}[1]
\State $[\bm \Gamma_l, \bm \Gamma_m, \bm \Gamma_r, \bm \Gamma_{l_d}, \bm \Gamma_{m_d}, \bm \Gamma_{r_d}] \gets \texttt{rsvd\_id\_d}(\bm N,[\mathcal{I}_1, \mathcal{I}_2], k)$ \Comment{See Algorithm~\ref{alg:rsvd_id_d}}
\State Use Equation~\eqref{svd_der1} in conjunction with the factorization $d\bm A \simeq \bm \Gamma_{l_d}\bm \Gamma_{m_d}\bm \Gamma_{r_d}$ to find $\partial \bm \Gamma_l/\partial \ell,~\partial \bm \Gamma_m/\partial \ell$ and $\partial \bm \Gamma_r/\partial \ell$ 
\end{algorithmic}\label{alg:svd_der}
\end{algorithm}


\section{Numerical Example I}\label{numexI}
In the following we provide more extensive computational studies/discussions associated with the first numerical example:

\subsection{Rank structure of kernel matrices} Similarly to the first experiment in Section~\ref{S5_1}, we study the singular values of the resulting kernel matrices for both exponential and squared exponential kernels to get an insight about rank structure of these matrices. To this end, we consider different sizes of indices, $n= \# \mathcal{I}$, perform the \texttt{permute} algorithm and generate the off-diagonal block $\bm A_{12} \in \mathbb{R}^{n_1 \times n_2}$ with  $n_1=0.1n$ and $n_2=0.9n$. The distribution of nodes is similar to the aforementioned experiment. The hyperparameters are fixed in both kernels, i.e. $\ell=\sqrt{2},1$. We then consider thresholds for the magnitude of singular values to find the number of effective rows (or columns) in the low-rank factorization. In particular, we count the number of singular values that satisfy $\sigma_i > 10^{-3},~\sigma_i > 10^{-8}$ and $\sigma_i > 10^{-3},~\sigma_i > 10^{-12}$ in exponential and squared-exponential kernels (by performing direct SVD). The result is shown in Figure~\ref{svl_index}. It is apparent that the squared-exponential kernel yields a lower rank matrix. The rate of increase in the effective rank for the more stringent case of $\sigma_i > 10^{-12}$ is quite slow which is promising for low-rank factorization of large size matrices. Extrapolating the curve with diamond markers in the right plot (visually) for a large size, e.g. $n=10^6$ yields effective rank in the order of $45\sim50$. This means that a $k=50$ SVD factorization of a large off-diagonal block, $\bm A_{12} \in \mathbb{R}^{10^5~\times~9\times10^5}$ (with nodes $\bm N \in [0,1]^2$ and $\ell=1$) is a plausible approximation.    

\begin{figure}[h]
\centering
\includegraphics[width=1.5in]{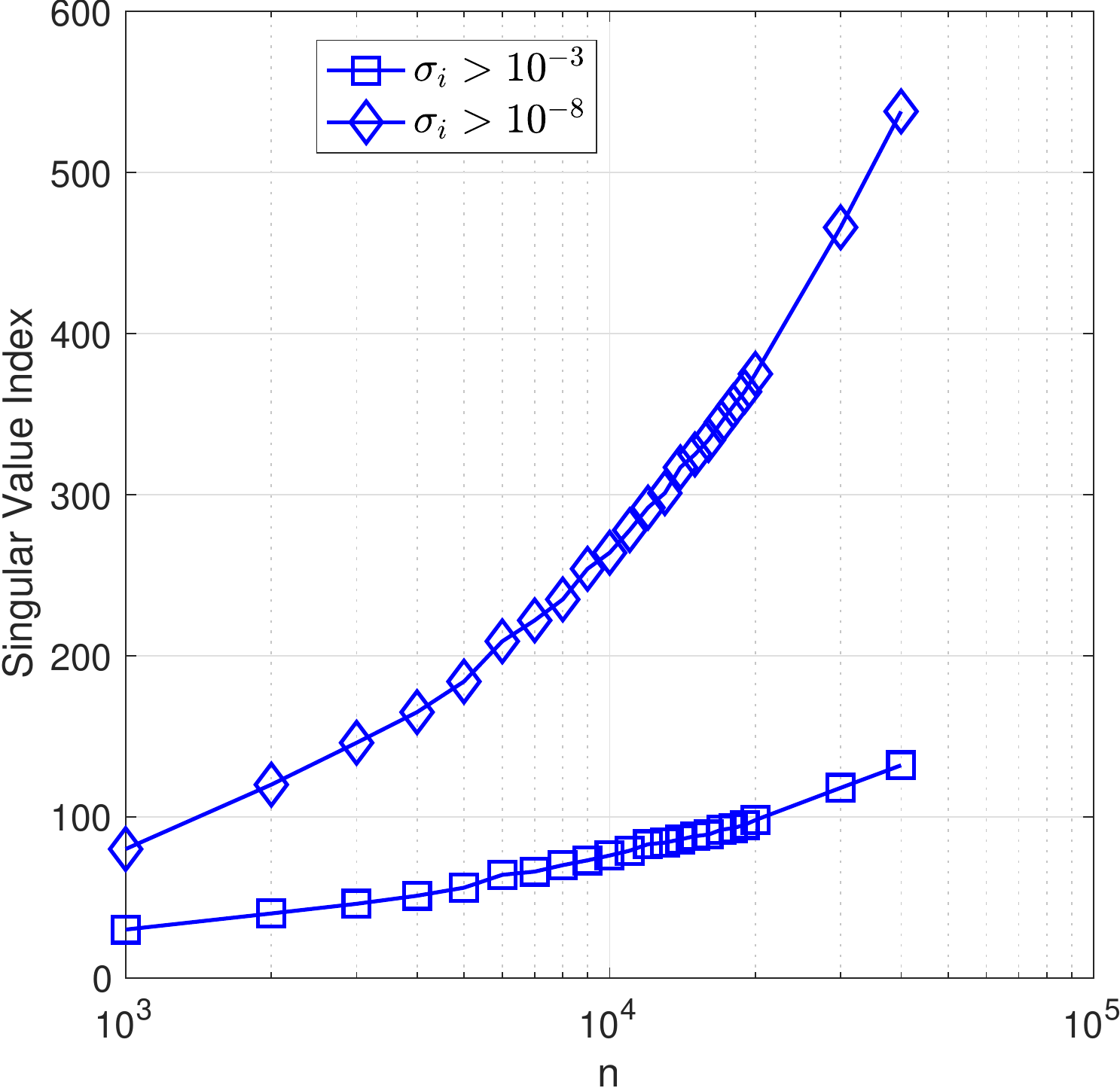}
\includegraphics[width=1.51in]{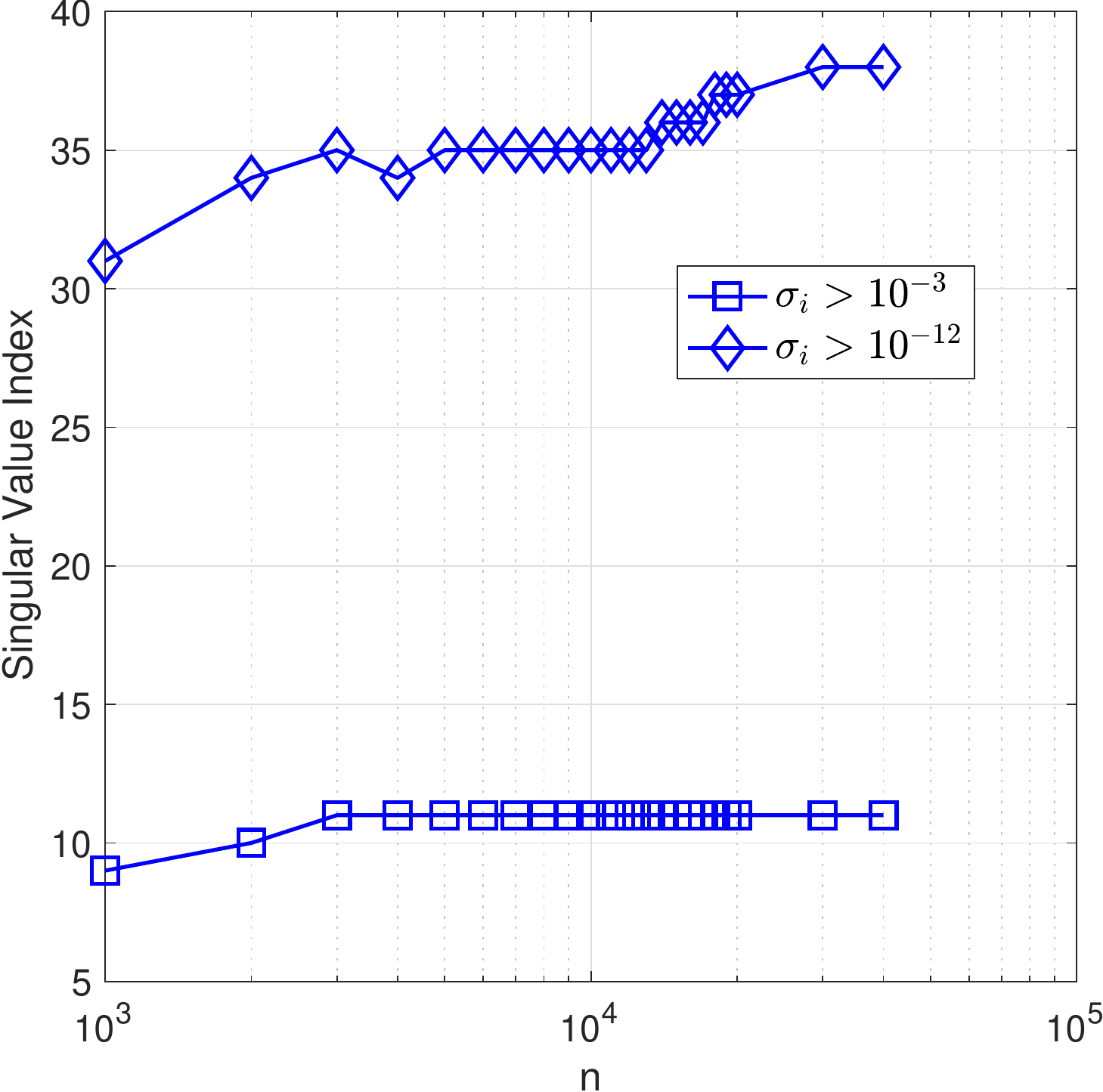}\\
\caption{\small{The number of singular values that satisfies a magnitude threshold for the exponential (left) and squared exponential (right) kernels.}}\label{svl_index}
\end{figure}

\subsection{Effect of subsampling on the range approximator} For this experiment associated with the subsampling discussion in Section~\ref{Sec3_1} cf. Equation~\eqref{n_inn}, we consider a relatively large rectangular matrix (for which we can perform direct computations) with the total number of nodes $\# \mathcal{I} = 15000$. We again use \texttt{permute} to generate $\# \mathcal{I}_1$ and $\# \mathcal{I}_2$ and subsequently $\bm A \gets \bm A_{12} \in \mathbb{R}^{5000 \times 10000}$. We also set $k=45,~p=5$, and therefore we consider rank $ k+p = 50$ factorization of the matrix. We study the effect of subsampling on 1) direct statistical computation of the error $\| \bm A - \bm Q \bm Q^T \bm A \|_F$ and 2) two main ingredients in the theoretical estimate of the error (cf. Theorem~\ref{theorem1}), $\bm \Omega_1$ and $\bm \Omega_2$, explained briefly next. 

\noindent
\textbf{1) Statistical computation of the error $\| \bm A - \bm Q \bm Q^T \bm A \|_F$:} We first statistically study the error $\| \bm A - \bm Q \bm Q^T \bm A \|_F$ for different choices of $n_{innprod}$, i.e. $n_{innprod} = \delta (k+p), ~\delta =2,3,4,5$ and full computation i.e. $n_{innprod}=n=10^4$ by performing $100$ runs for each choice. The probability distribution functions (PDFs; obtained via kernel density estimation) of error are shown in the first pane of Figure~\ref{subsampling_Q}. The mean of error associated with the full computation, i.e. with $n_{innprod}=n=10^4$, is $1.087\times10^{-4}$, and the mean corresponding to $\delta=2$ is one order of magnitude larger i.e. $2.327 \times 10^{-3}$. However, as we increase the number of samples slightly, the error quickly converges to the full computation error; e.g. the case of only $5(k+p)=250$ samples yields $2.004\times10^{-4}$, which is close to the full computation error. Also note that this error is pertinent only to $\| \bm A - \bm Q \bm Q^T \bm A \|_F$, and it does not have a significant impact on the error in the computation of $\bm A \bm x = \bm y$. The overall framework for linear solve (or likelihood evaluation) is more prone to inaccuracies due to the lack of sufficient rank consideration $k$ (see Figures~\ref{lkl_values} and~\ref{lkl_values_UncBand}) compared to a less significant parameter $n_{innprod}$. Therefore, depending on the application of the overall framework in different contexts, consideration of smaller $n_{innprod}$ might be possible. 

\noindent
\textbf{2) Building blocks of the proof of Theorem~\ref{theorem1}:} The proof (explained in detail in Section 10 of~\cite{nhalko11}) involves two random submatrices $\bm \Omega_1 = \bm V^T_1 \bm \Omega$ and $\bm \Omega_2 = \bm V^T_2 \bm \Omega$ where $\bm V_1 \in \mathbb{R}^{k \times k}$ and $\bm V_2 \in \mathbb{R}^{k \times n-k}$ are two subblocks of the right singular vectors of $\bm A \in \mathbb{R}^{m \times n}$, denoted by $\bm V$, i.e. $\bm V = [\bm V_1 ~\bm V_2]$. Denoting the psuedo-inverse of $\bm \Omega_1$ by $\bm \Omega^{\dagger}_1$, the random submatrices admit  $\mathbb{E} (\|\bm \Omega^{\dagger}_1\|^2_F) = {k}/{(p-1)}$ and $\mathbb{E}(\|\bm \Omega_2\|^2_F) = 1$ (by applying two standard results cf.~\cite{nhalko11}). The choice of minimum $2k$ samples in our subsampling strategy ensures that $\bm V_2$ has at least $k$ columns. We again statistically study the mean of the Frobenius norm of $\bm \Omega^{\dagger}_1$ and $\bm \Omega_2$ using similar choices of $n_{innprod}$ cf. second and third panes of Figure~\ref{subsampling_Q}. In these results, the subsampling has even less impact on two main ingredients of the theoretical estimate (as all choices of $n_{innprod}$ result in relatively similar distributions for $\|\bm \Omega^{\dagger}_1\|^2_F$ and $\|\bm \Omega_2\|^2_F$), i.e. the upper bound in Theorem~\ref{theorem1}. 
\begin{figure}[h]
\centering
\includegraphics[width=1.5in]{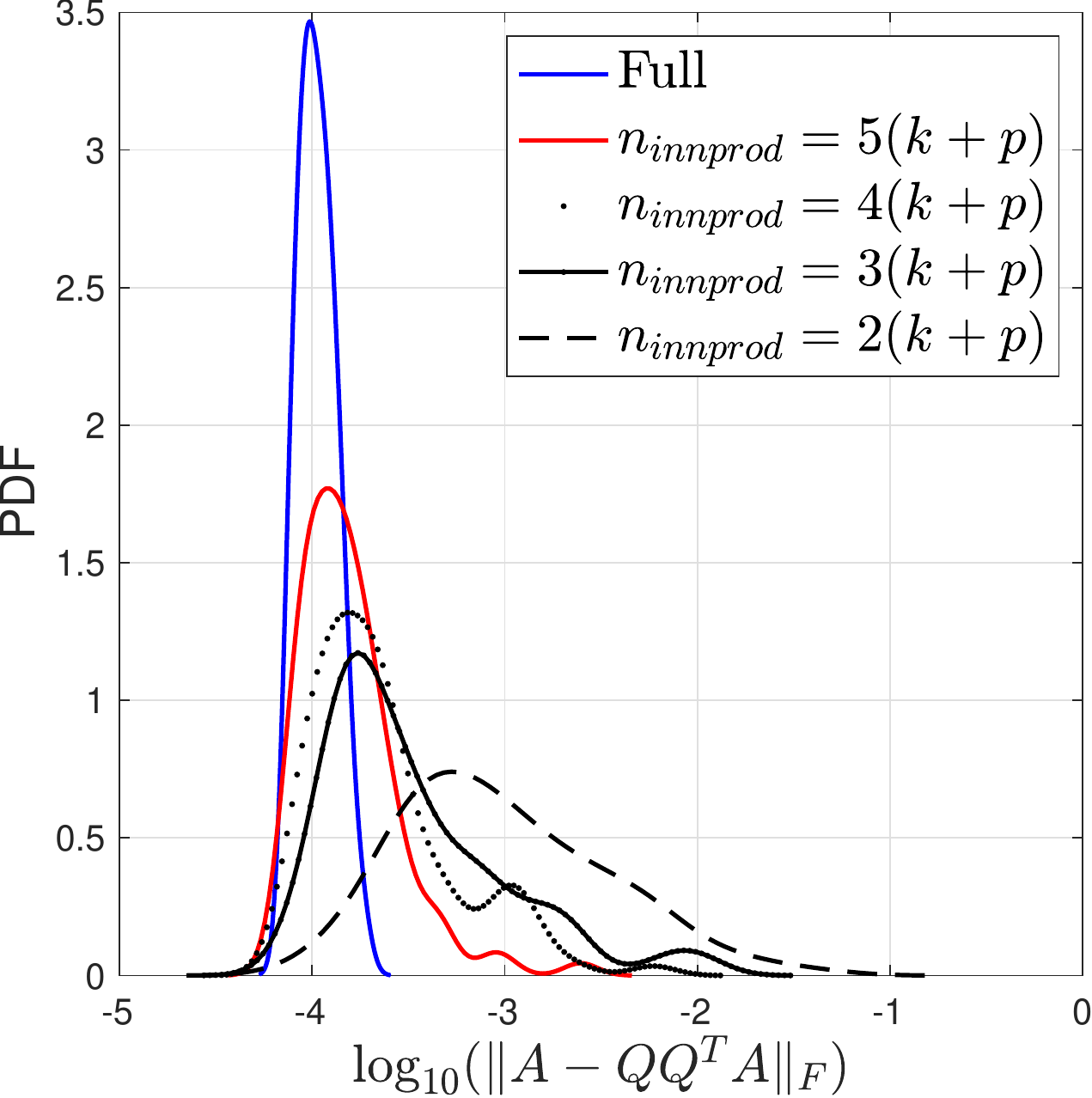}
\includegraphics[width=1.55in]{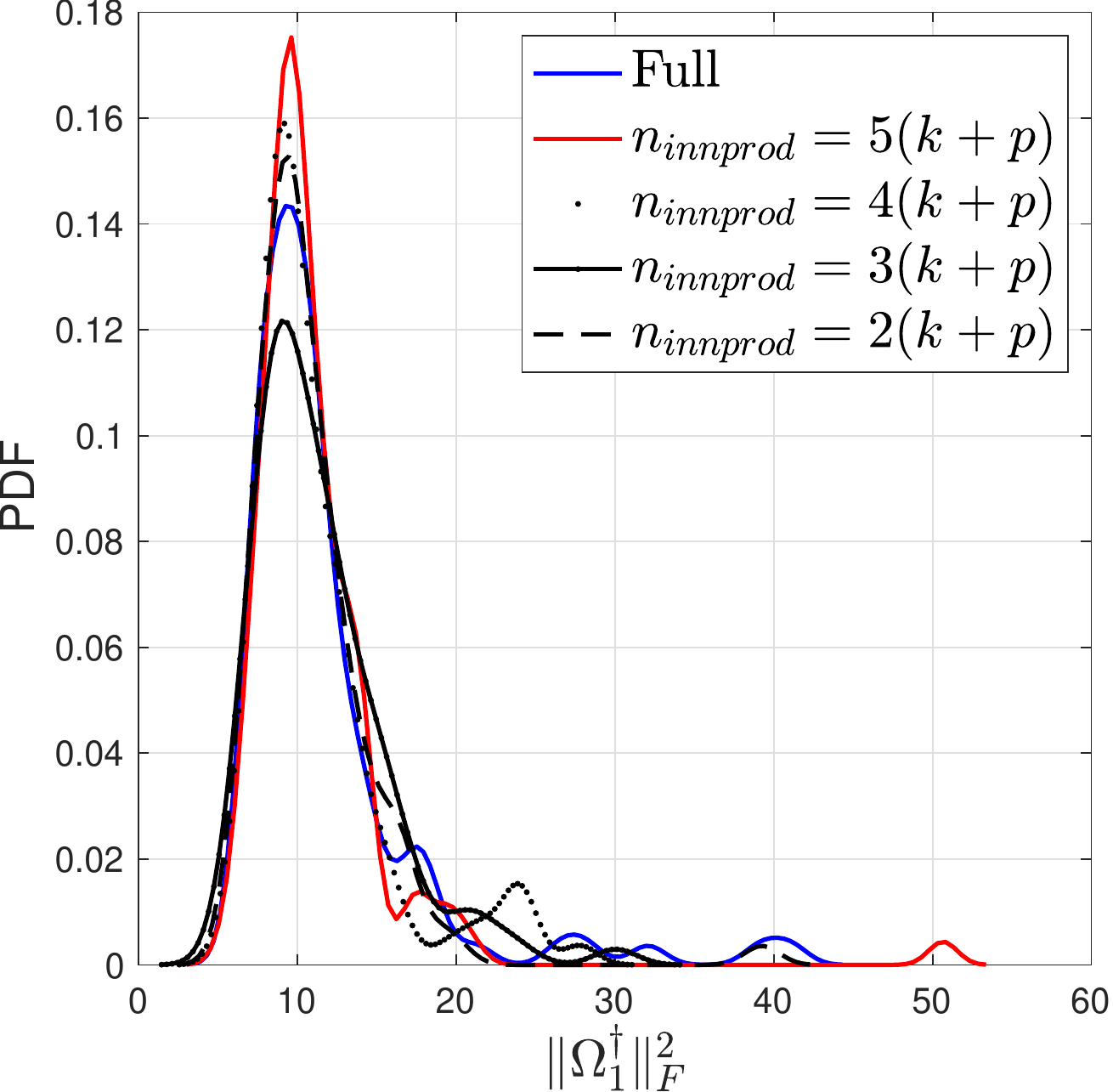}
\includegraphics[width=1.55in]{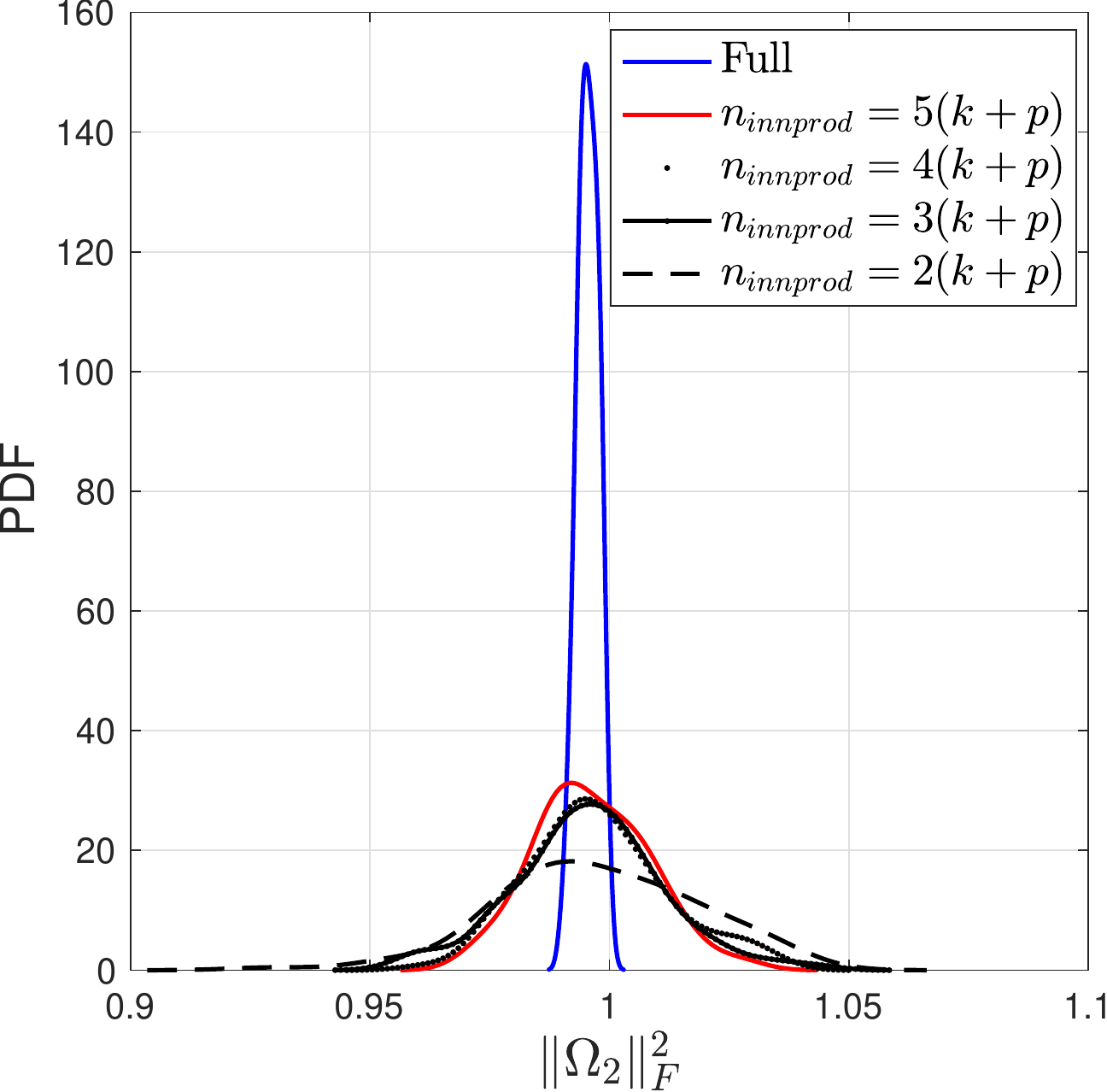}
\caption{\small{Effect of subsampling on the range approximator $\bm Q$ error and two main computational components $\|\bm \Omega^{\dagger}_1\|_F$ and $\|\bm \Omega_2\|_F$ in its theoretical estimate.}}\label{subsampling_Q}
\end{figure}

\subsection{Empirical error} In this example we empirically investigate the hierarchical error for computation of $\bm A \bm x = \bm y$ as discussed in Section~\ref{sec:error}. The matrix with size is $\# \mathcal{I} = 5000$ and we set $\eta=105$. First, we compute the probabilistic error in Theorems~\ref{theorem1} and~\ref{theorem2} on the largest off-diagonal block $\bm A \gets \bm A_{12} \in \mathbb{R}^{1000 \times 4000}$. We again set $k=45,~p=5$ and compute the error $\| \bm A - \bm Q \bm Q^T \bm A \|$, $100$ times. We also compute the right-hand sides in the estimates of Theorems~\ref{theorem1} and~\ref{theorem2} using the parameters mentioned after Theorem~\ref{theorem2}, i.e. $t=e,~u=\sqrt{2p}$, which results in probability of failure $3e^{-p}=0.0202$. The result for  $\log_{10} (\| \bm A - \bm Q \bm Q^T \bm A \|_F)$ is shown in Figure~\ref{err_hierarchical} (left pane). The mean of the log error and the threshold for probability of failure associated with this plot is obtained as $\mu_{\epsilon, emp}=-2.6789$ and $\bar{\epsilon}_{emp}=-2.473$. Taking the log from the right-hand side of estimates yields $\mu_{\epsilon, theo}=-1.5011$ and $\bar{\epsilon}_{theo}=-0.8754$ which means that the empirical values are in agreement with the theoretical estimates in~\ref{theorem1} and~\ref{theorem2}; i.e. the actual error is within the bound provided by the analytical estimates. 

Using the estimates for the error in the off-diagonal blocks (by assuming a log normal distribution for the error) and the hierarchical error estimate in Section~\ref{sec:error} we compute the analytical errors in different levels within the $n=10^3$ blocks and the original matrix. The computation of hierarchical error depends on $\beta_i$ and $\kappa$ estimates. The actual evaluation of analytical estimates for these quantities can potentially result in very large values (as these involve Frobenius norm on large matrices). Instead, we assign particular small values to these quantities and investigate the probabilistic error for the inversion algorithm, i.e. we set $\beta_i=1,~\kappa=1.02$. The empirical probabilistic errors for different levels are shown in the second and third panes. Finally, we normalize the estimates for the inversion error in the original matrix in both analytical and empirical scenarios with respect to their mean, and find the distribution of this normalized error via kernel density estimation. The result is shown in the right pane of Figure~\ref{err_hierarchical}. The distribution of normalized errors (denoted by $\tilde{\epsilon}_{D,0}$ notation) in both cases are almost in agreement; however, in terms of constant these two estimates are far away from each other, i.e. $\mu_{\log(\epsilon),theo}=3.34$ and $\mu_{\log(\epsilon),emp}=-3.94$.

\begin{figure}[h]
\centering
\includegraphics[width=1.19in]{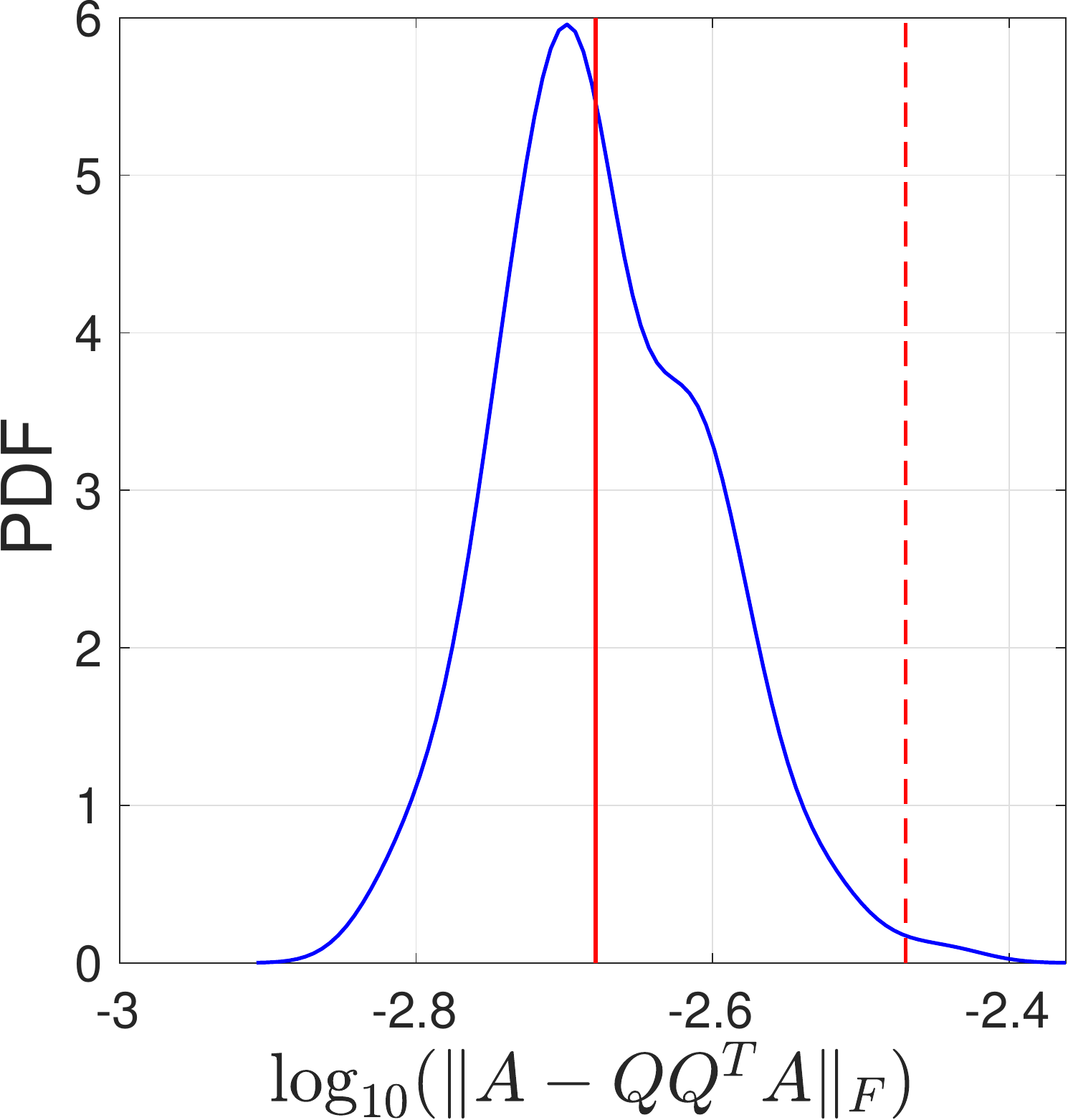}
\includegraphics[width=1.23in]{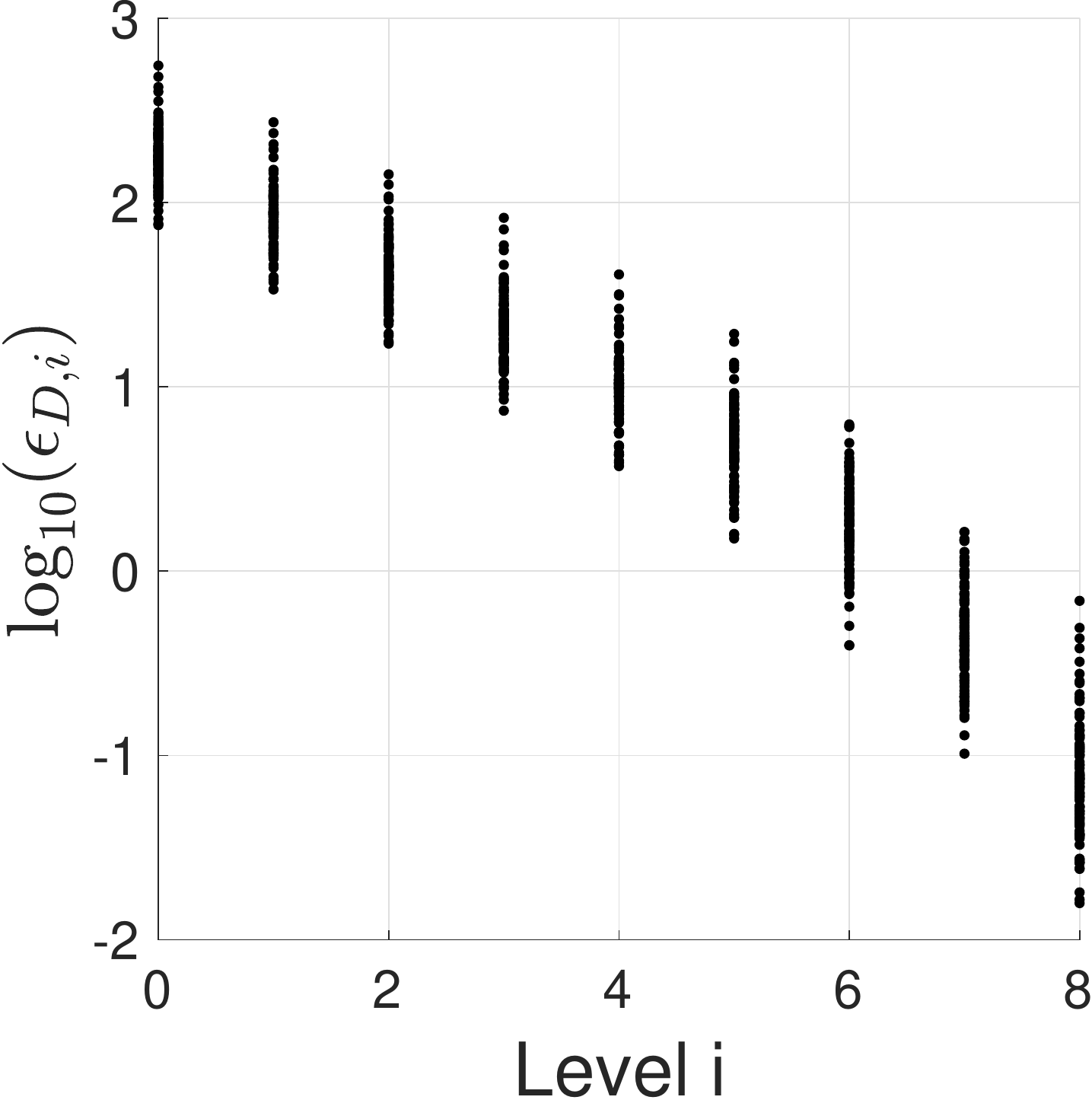}
\includegraphics[width=1.24in]{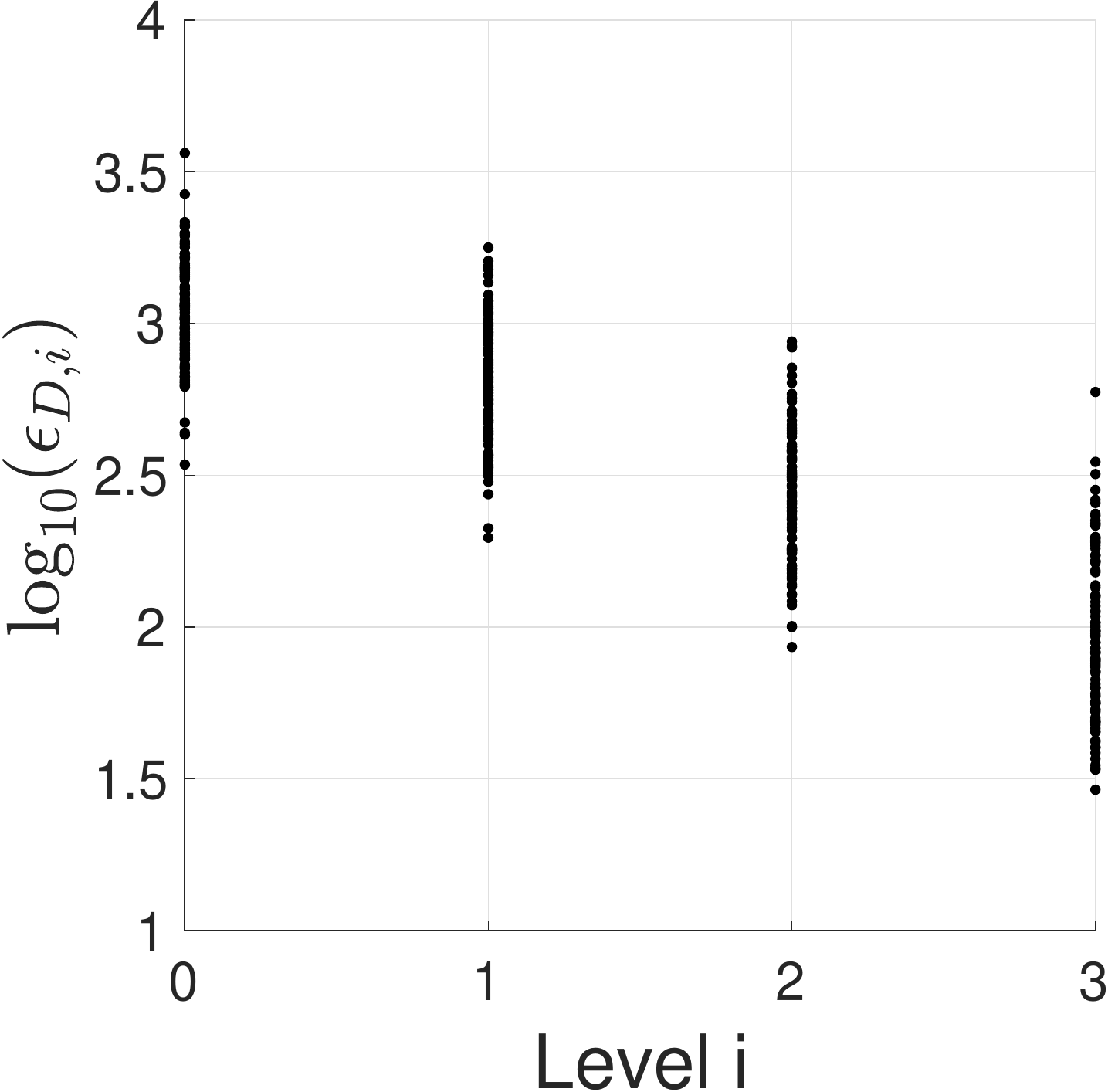}
\includegraphics[width=1.23in]{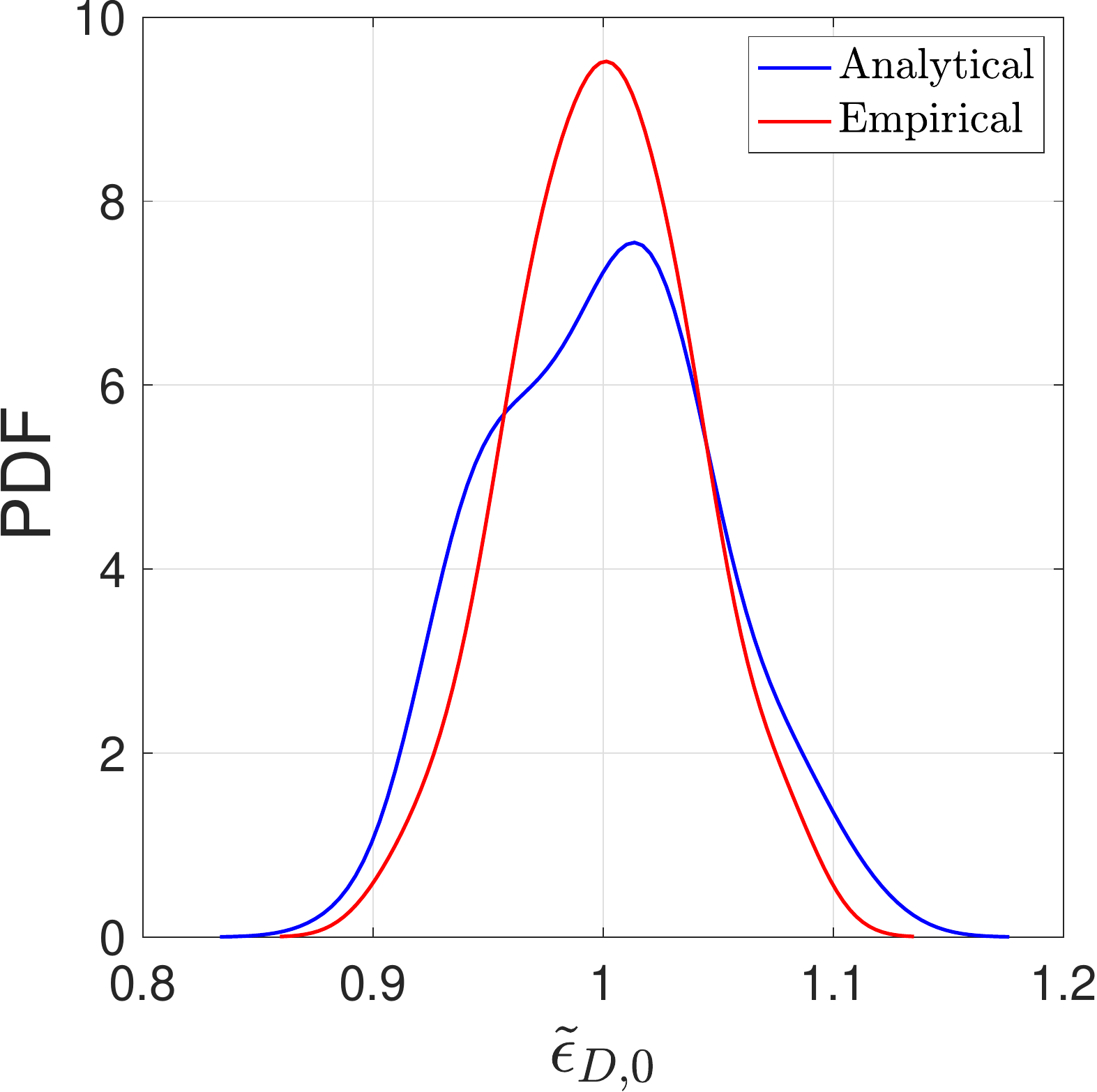}
\caption{\small{A computational study on the results of Theorems~\ref{theorem1} and~\ref{theorem2} (first pane), hierarchical errors in the block with $1000$ nodes (second pane), hierarchical errors in the block with $5000$ nodes (third pane), and the distribution of normalized error (analytical and empirical) in the original matrix with $5000$ nodes (right pane).}}\label{err_hierarchical}
\end{figure}

\subsection{Accuracy of likelihood terms}  Results in this subsection are associated with discussions in Sections~\ref{Sec4_1} and~\ref{Sec4_4}. To study the accuracy of the likelihood terms and their derivative, we consider a kernel matrix with size $\# \mathcal{I} = 5000$. The scalar hyperparameter for two cases of squared exponential and exponential kernel is set to $\ell = 1,~\sqrt{2}$. We also consider $\eta=105$ in this case and set $\sigma^2_n=10^{-3}$ to find $\tilde{\bm A} = \bm A + \sigma^2_n \bm{I}$. Figure~\ref{lkl_values} shows the results for the exponential (top row) and squared exponential kernels (bottom row). The left and right columns in row show the results for two scenarios: 1) no node permutation and 2) node permutation based on the procedure in Section~\ref{S2_2}. 
\begin{figure}[h]
\centering
\includegraphics[width=2.0in]{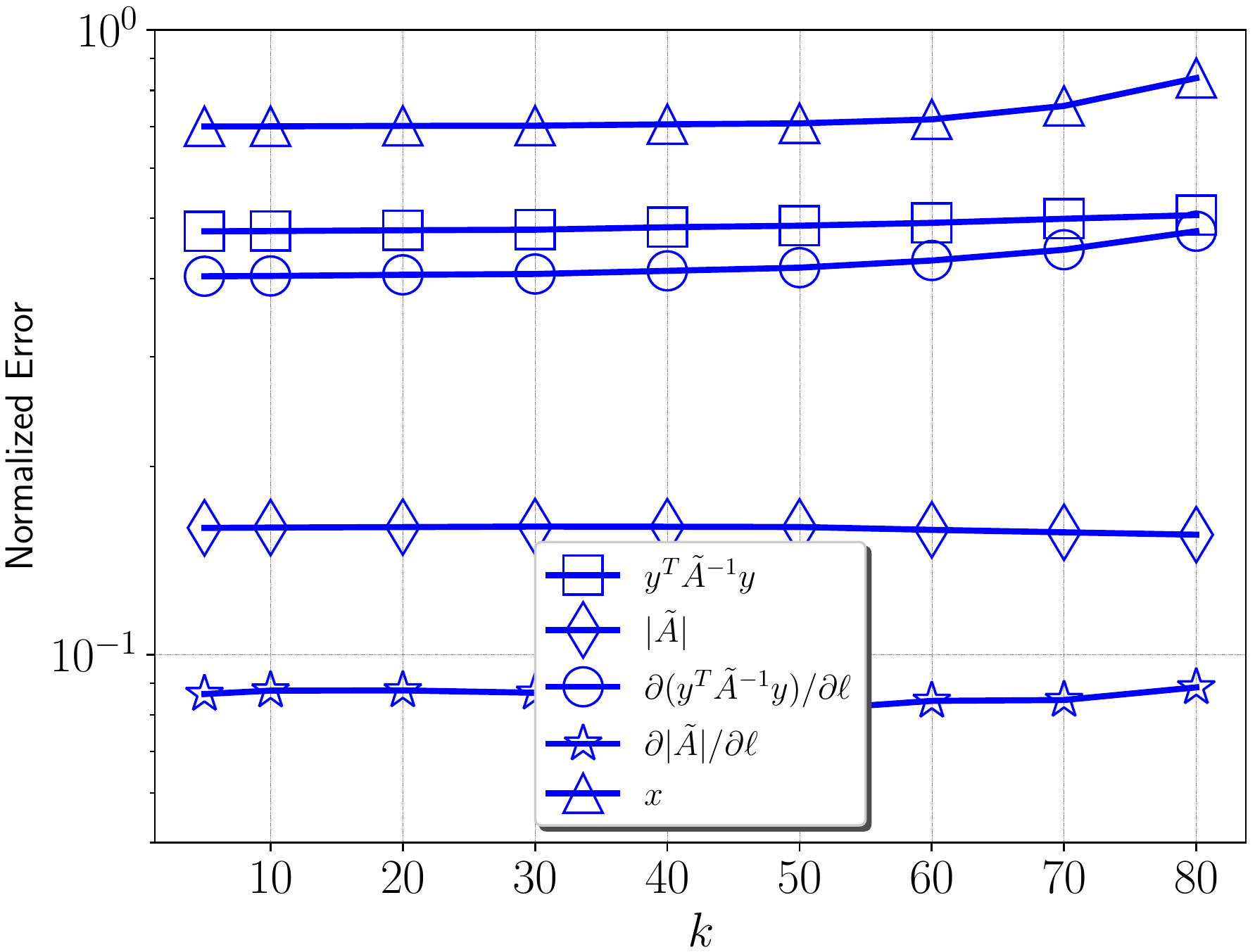}
\includegraphics[width=2.0in]{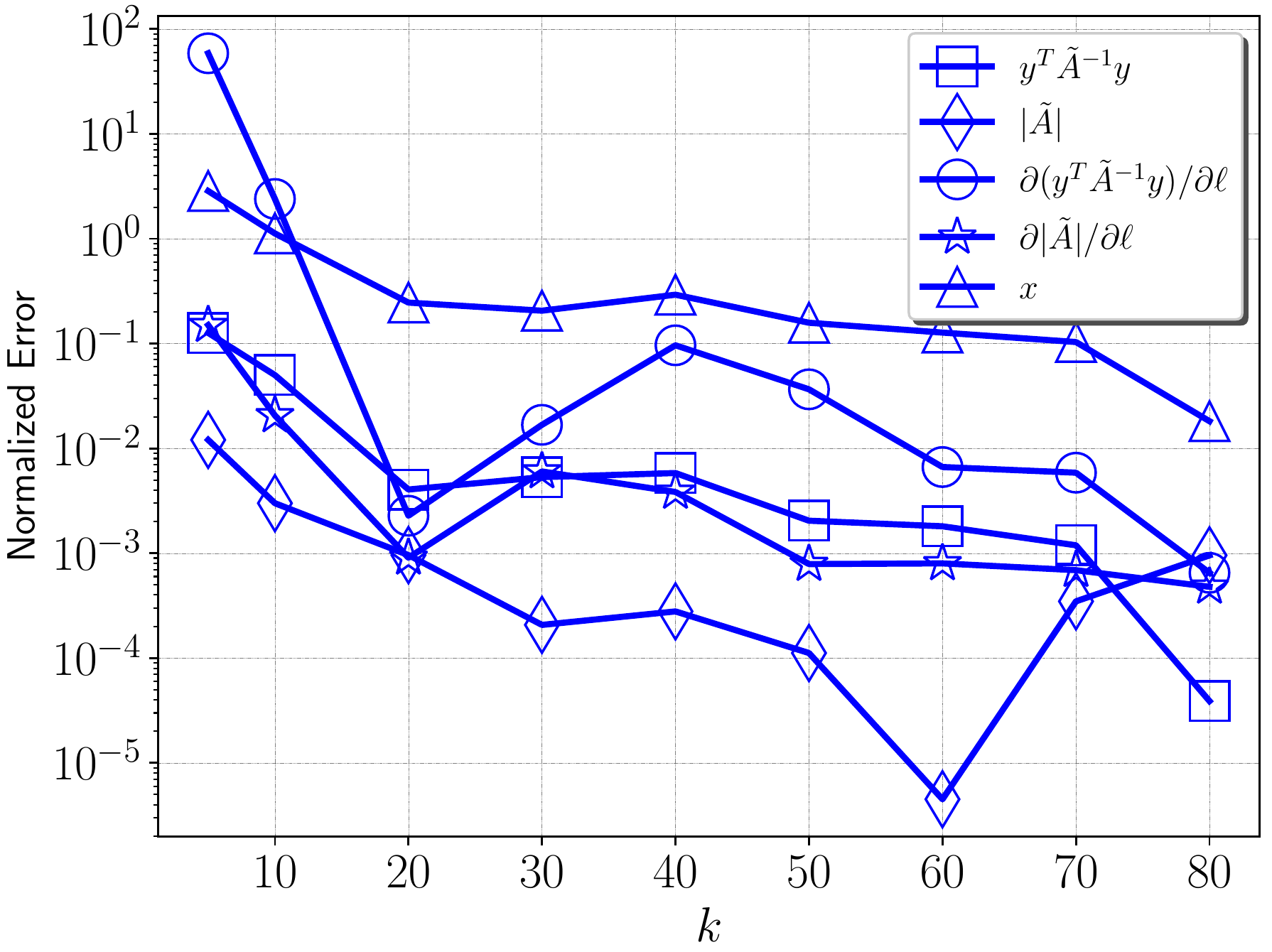}
\includegraphics[width=2.0in]{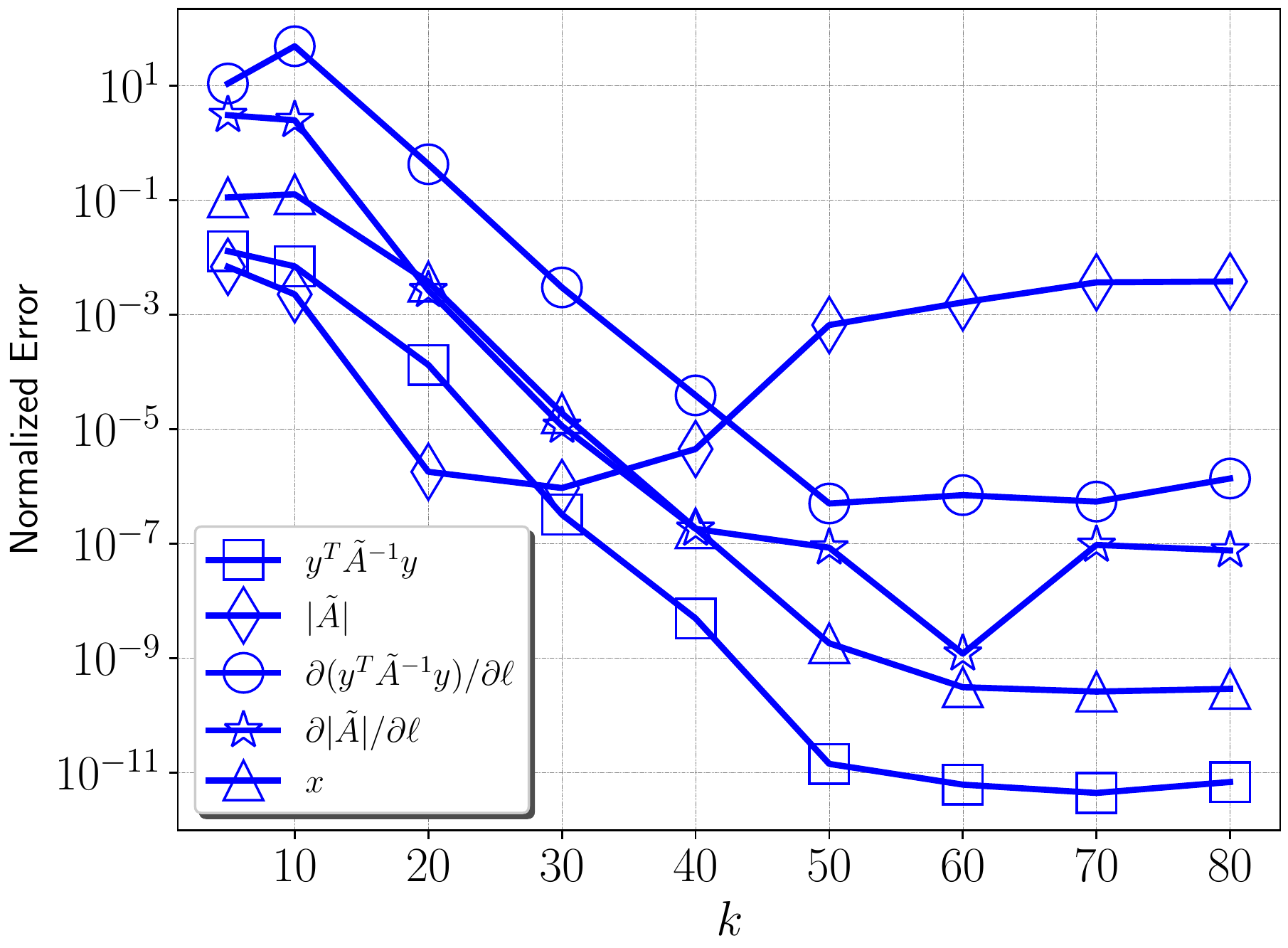}
\includegraphics[width=2.01in]{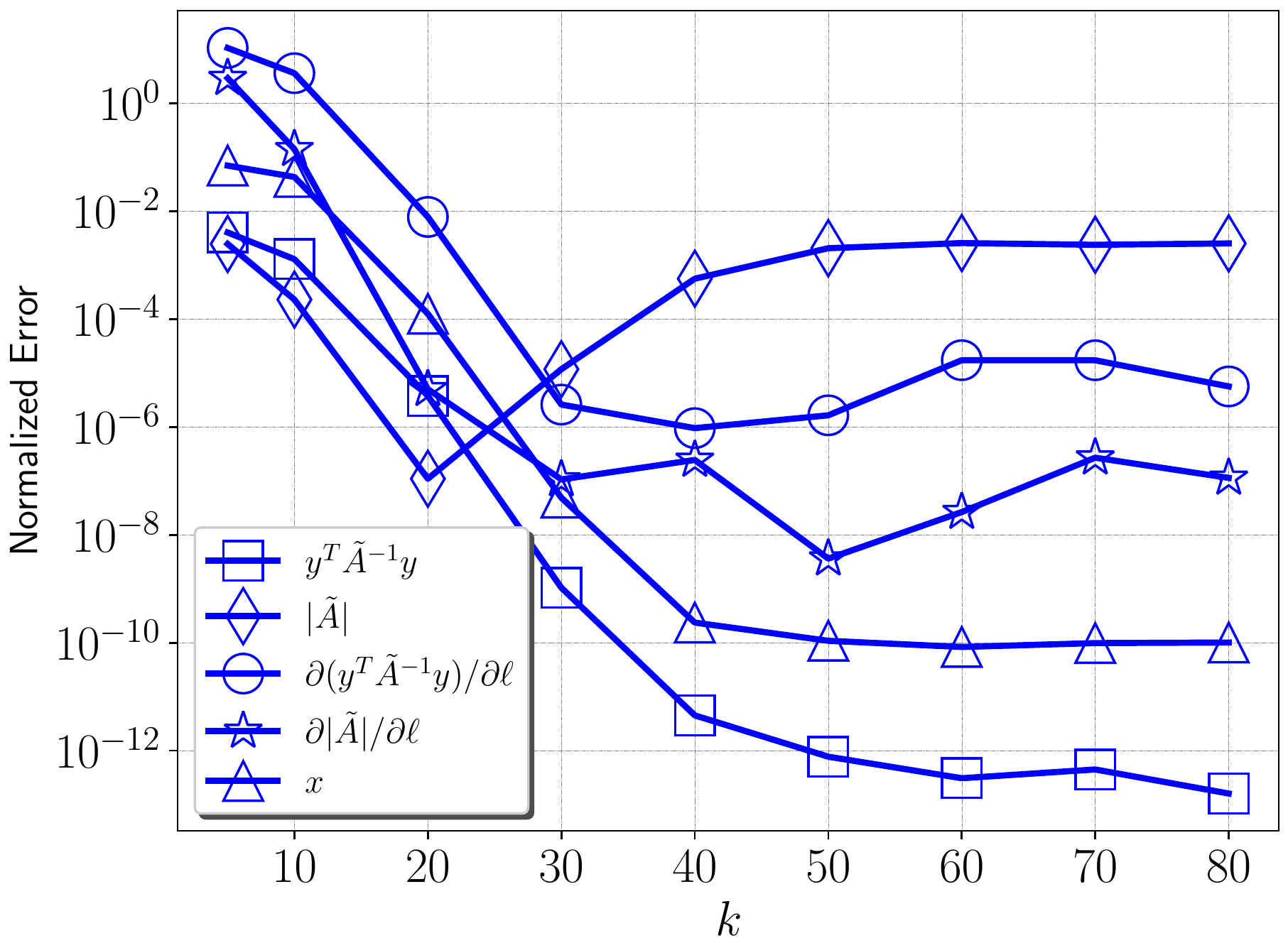}
\caption{\small{One realization of normalized error in computation of likelihood terms and $\bm x$ with respect to $k$: exponential kernel (top row), squared exponential kernel (bottom row), considering no node permutation (left column), considering node permutation (right column).}}\label{lkl_values}
\end{figure}
As expected, the squared exponential kernel which has a lower rank structure yields significantly better accuracies for the range of $k$ considered in this example. It is also apparent that the node permutation increases the accuracy of the linear solve as well as the likelihood computation and its gradient.  

\subsection{Statistical study on the accuracy of energy term} To study the statistical performance of the approach, for a kernel matrix with $n=5000$ and $\eta=1050$ cf. Section~\ref{S2_2}, we run the code $100$ times and find estimates for the mean (denoted by $\mu$) and standard deviation (denoted by $\sigma$) of energy for both exponential and squared exponential kernels, cf. Figure~\ref{lkl_values_UncBand}. From these results, it is apparent that the normalized errors associated with squared exponential kernel are smaller, which is due to the lower rank structure of the resulting kernel matrix as shown in Figure~\ref{smoothness_rank}. The small variation in the error is also apparent, which means more certain (and more accurate) approximations are obtained from the lower rank squared exponential kernel matrix.

\begin{figure}[h]
\centering
\includegraphics[width=1.8in]{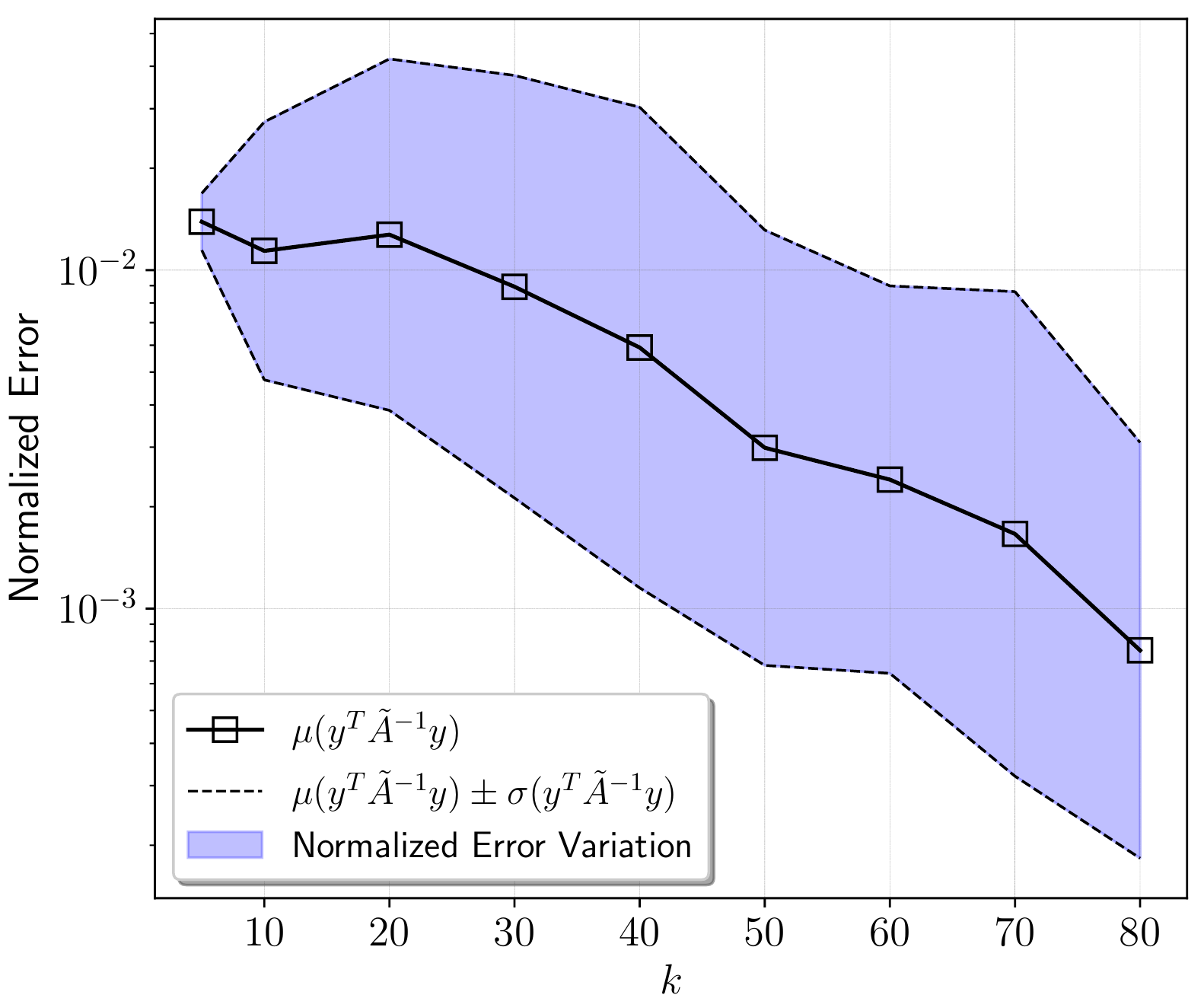}
\includegraphics[width=1.8in]{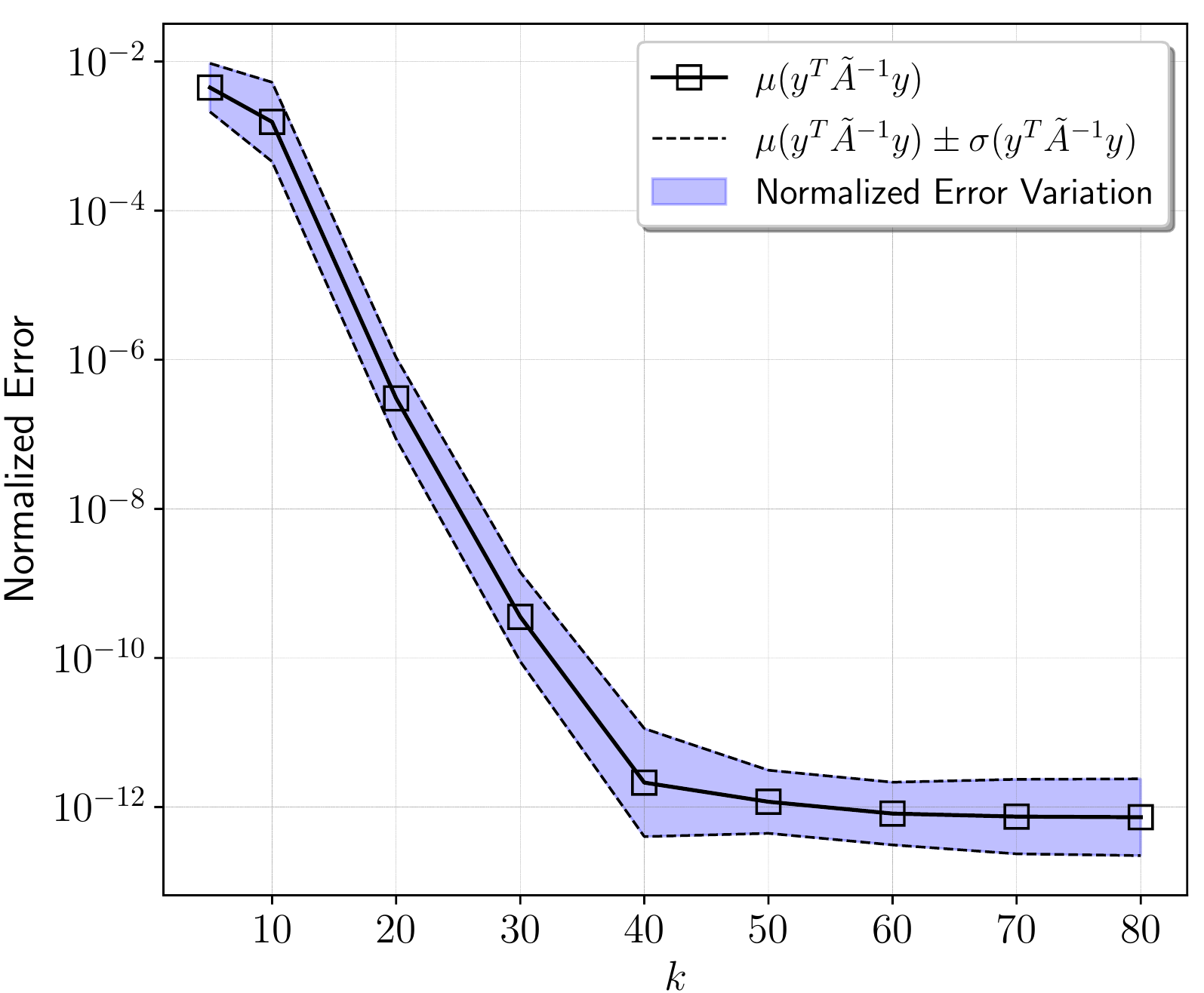}
\caption{\small{Statistical normalized errors for the energy term associated with the exponential kernel (left) and squared exponential kernel (right). Statistical results are obtained by $100$ runs of the code for each $k$. }}\label{lkl_values_UncBand}
\end{figure}

\section{Numerical Example II}\label{numexII}
In this section, we provide further details on the numerical example in Section~\ref{S5_2}. The geometry and boundary condition of the structure as well as the nominal distribution of von Mises stress are shown in Figure~\ref{fig_struct}. The structure is assumed to be comprised of two materials (hence double-phase) shown via gray and darker gray regions within the structure. The nominal elastic modulus in the darker region is double the one in the lighter region. 

\begin{figure}[h]
\centering
\includegraphics[width=2.41in]{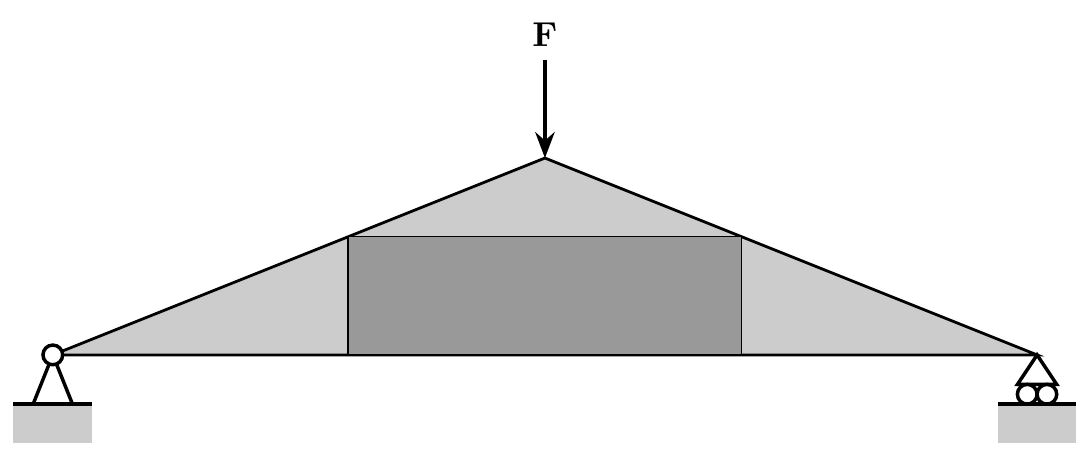} \qquad\includegraphics[width=2.38in]{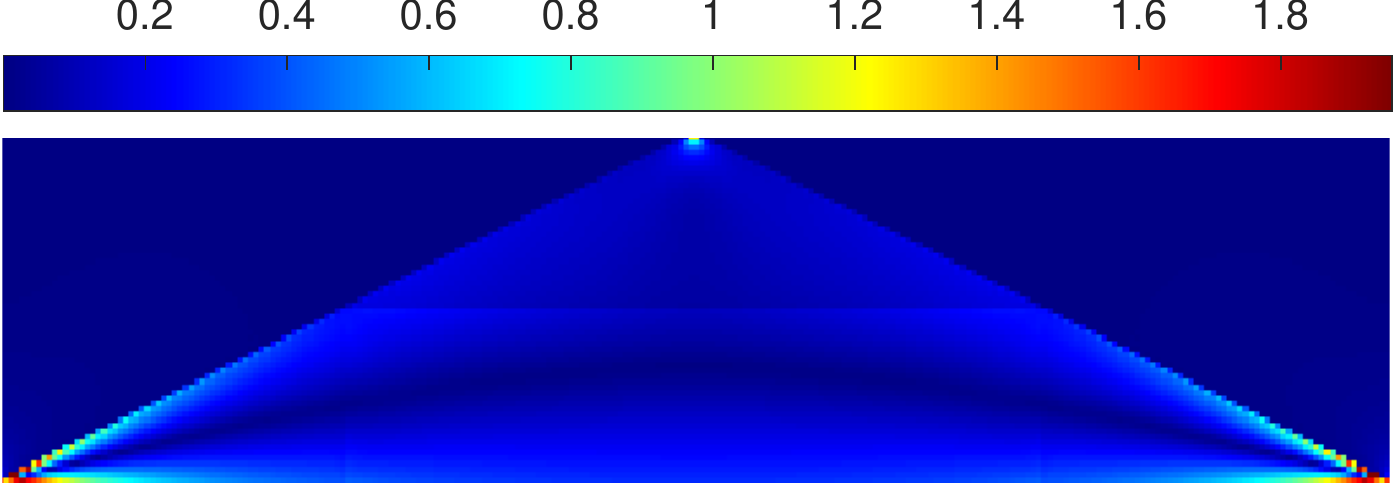}
\caption{\small{Loading and boundary conditions for a triangular double-phase linear elastic structure (left). Distribution of nominal von Mises stress (right).}}\label{fig_struct}
\end{figure}
The elastic modulus in the domain is represented via a Karhunen-Lo\'{e}ve expansion in the form of:
\begin{equation}
E(\bm x, \bm \xi) = \sum_{i=1}^{10} \sqrt{\lambda_i} E_i (\bm x) \xi_i
\end{equation}
where $\bm \xi = \{\xi_i \}_{i=1}^{10}$ are uniformly distributed random (parametric) variables, i.e. $\xi_i \sim U[0,1]$.  The eigenvalues $\lambda_i$ and eigenvectors $E_i(\bm x)$ are obtained from an eigen-analysis on the exponential kernel in 2D. We refer the readers to~\cite{mldq_vkz,Teckentrup15} for further details on this eigen-analysis. We consider a relatively high-dimensional KL expansion, i.e. with $10$ modes, to test the performance of the solver with the ARD kernel with multiple hyperparameters. 

Figure~\ref{optimal_ell} shows the optimal hyperparameter values as the result of statistical boosting. These results are shown for each dimension of data (which are in total $10$ dimensions) and for all (statistical boosting) models, therefore total of $10$ optimal values. We compute the error in the test data in each case and select the one with the smallest error. The third and fourth panes show the point-wise error for $q_1$ and $q_2$ quantities for a test data set. The point-wise error is only shown with respect to one dimension (feature) of the data, i.e. $\xi_5$.
\begin{figure}[h]
\centering
\includegraphics[width=1.2in]{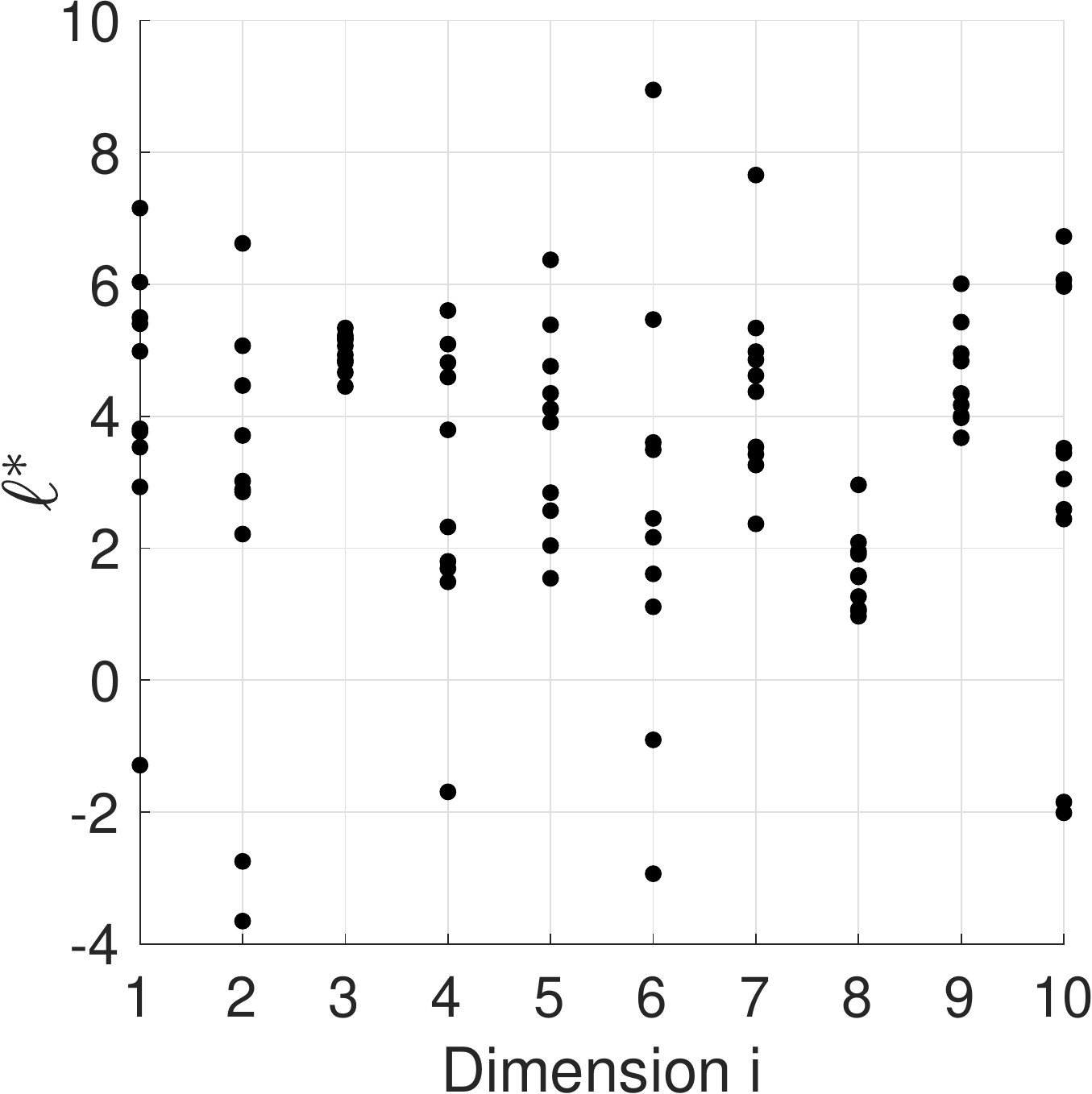}
\includegraphics[width=1.24in]{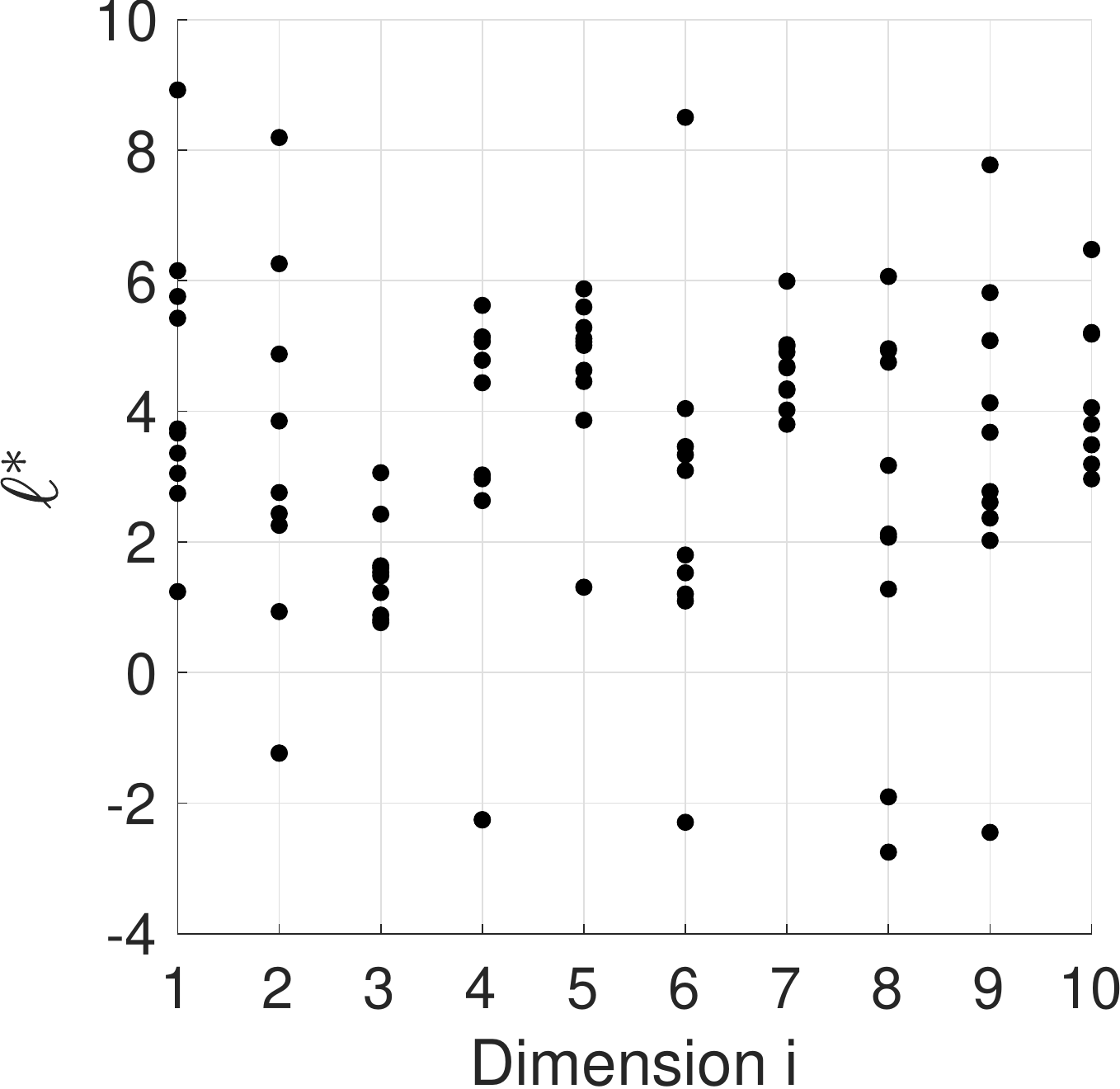}
\includegraphics[width=1.24in]{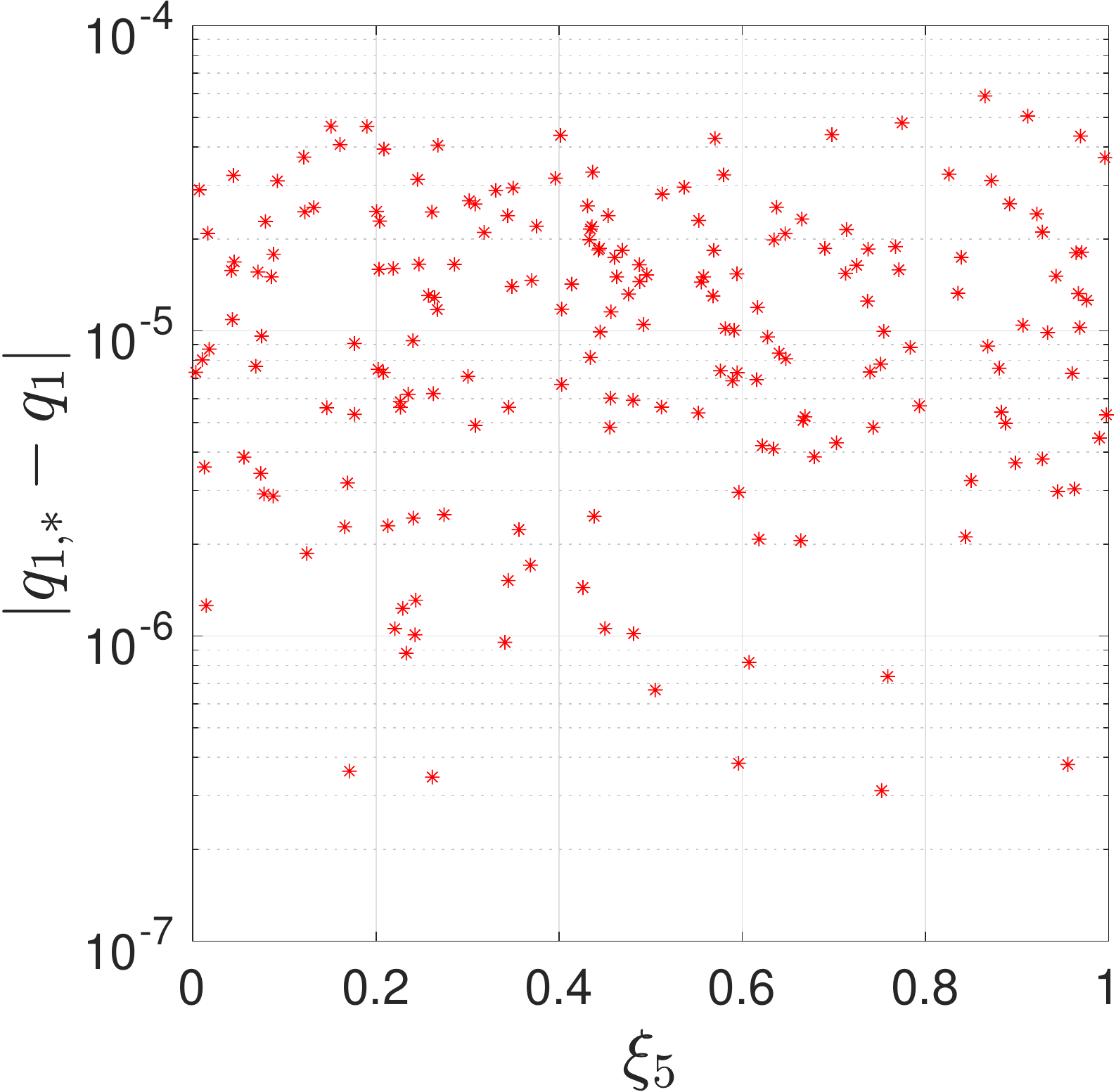}
\includegraphics[width=1.24in]{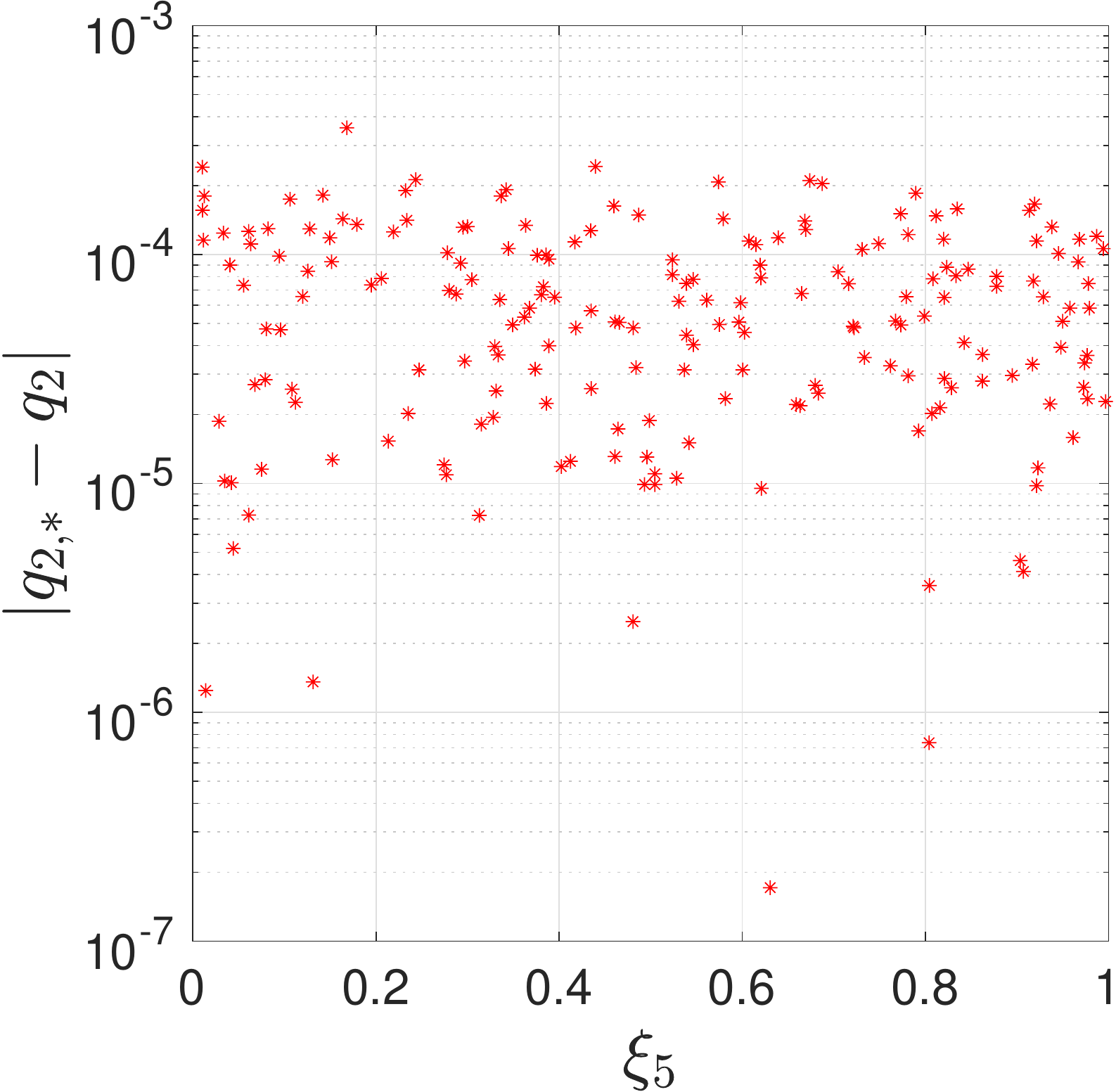}
\caption{\small{Optimal hyperparameter values with respect to data dimensions (first and second panes), and the point-wise error of GP prediction (third and fourth panes) for both $q_1$ and $q_2$.}}\label{optimal_ell}
\end{figure}

\section*{Acknowledgments}
The first author thanks Dr. Majid Rasouli from the University of Utah for his insights and assistance on details of the algebraic multigrid method.

\end{document}